\documentclass[12pt,reqno]{amsart}
\usepackage{amsmath,amsthm,amsfonts,amssymb,amstext,amsxtra}
\usepackage{mathrsfs,color}
\usepackage{txfonts}
\usepackage[colorlinks=true,linkcolor=black,citecolor=blue,urlcolor=blue,
  pagebackref,hypertexnames=false]{hyperref}
\hypersetup{pdftitle={Weighted composition-differentiation operators on mixed-norm, derivative Hardy and F(p,q,s) spaces},
  pdfauthor={Jihua Sun, Junming Liu}}
\hypersetup{linktoc=all,bookmarksnumbered=true,bookmarksopen=true,bookmarksdepth=2}
\allowdisplaybreaks

\newtheorem{theorem}{Theorem}[section]
\newtheorem{proposition}{Proposition}[section]
\newtheorem{lemma}{Lemma}[section]
\newtheorem{corollary}[theorem]{Corollary}
\theoremstyle{definition}

\newtheorem{example}{Example}[section]
\theoremstyle{remark}

\numberwithin{equation}{section}
\title[Weighted composition--differentiation operators]
{Weighted composition--differentiation operators on\\
 mixed-norm, derivative Hardy and $F(p,q,s)$ spaces}
\author{Jihua Sun, Junming Liu }
\address{Jihua Sun, School of Mathematics and Statistics, Guangdong University of Technology, Guangzhou, Guangdong, 510520, P.~R.~China}
\email{2112414014@mail2.gdut.edu.cn}

\address{Junming Liu, School of Mathematics and Statistics, Guangdong University of Technology, Guangzhou, Guangdong, 510520, P.~R.~China}
\email{jmliu@gdut.edu.cn}
\date{\today}

\begin{document}
\baselineskip=16pt
\begin{abstract}
We study weighted composition--differentiation operators on mixed-norm,
derivative Hardy and $F(p,q,s)$ spaces. Explicit dyadic criteria yield
boundedness, compactness and two-sided norm and essential-norm estimates,
including endpoint and quasi-Banach cases. We also obtain multiplier and
approximation results, Schatten criteria on weighted Dirichlet spaces, and
exact singular values for monomial weights and linear symbols.
\end{abstract}
\maketitle
\pagestyle{myheadings}
\baselineskip=16pt

\begingroup
\small
\pdfbookmark[1]{Contents}{contents}
\tableofcontents
\endgroup
\medskip

\section{Introduction}\label{sec:introduction}

\subsection{Notation and function spaces}\label{subsec:intro-spaces}

Let $\mathbb D=\{z\in\mathbb C:|z|<1\}$ and let $H(\mathbb D)$ be the
space of analytic functions on $\mathbb D$. We use normalized area
measure $dA=dx\,dy/\pi$ and normalized circle measure $dm=d\theta/(2\pi)$.
For $\phi\in H(\mathbb D)$ and an analytic self-map $\varphi$ of
$\mathbb D$, the operator considered throughout is
\[
 D_{\phi,\varphi}^{n}f(z)=\phi(z)f^{(n)}(\varphi(z)),
             \qquad n=0,1,2,\ldots,
\]
where $f^{(0)}=f$. Thus $D_{\phi,\varphi}^{0}$ is a weighted composition
operator, while $D_{1,\varphi}^{n}$ is differentiation followed by
composition. The order of these operations matters:
$D_{\phi,\varphi}^{n}$ does not denote the $n$th derivative of
$\phi(f\circ\varphi)$.

\emph{Hardy and mixed-norm spaces.}
For $0<p<\infty$ and $0\le t<1$, put
\[
 M_p(f,t)=\left(\int_0^{2\pi}|f(te^{i\theta})|^p\,
                       \frac{d\theta}{2\pi}\right)^{1/p},
 \qquad M_\infty(f,t)=\max_{|z|=t}|f(z)|.
\]
The Hardy space $H^p$, $0<p\le\infty$, consists of the functions
$f\in H(\mathbb D)$ for which
$\|f\|_{H^p}=\sup_{0\le t<1}M_p(f,t)<\infty$.
We use the usual boundary-value theory for finite Hardy exponents;
see \cite{Duren1970}.
For $0<p,q\le\infty$ and $\alpha>0$, define
\[
 \begin{aligned}
 \|f\|_{H(p,q,\alpha)}^q
   &=\alpha q\int_0^1(1-t)^{\alpha q-1}M_p(f,t)^q\,dt,
                                      &&q<\infty,\\
 \|f\|_{H(p,\infty,\alpha)}
   &=\sup_{0\le t<1}(1-t)^\alpha M_p(f,t),&&q=\infty.
 \end{aligned}
\]
The mixed-norm space $H(p,q,\alpha)$ consists of the analytic functions
with finite corresponding expression. The indices $p$ and $q$ measure
angular and radial integrability separately, and $\alpha$ specifies the
radial weight. Growth estimates and inclusions for these spaces are
studied in \cite{Arevalo2015,Arevalo2016}; their weighted atomic and
Carleson-measure theory is developed in \cite{PelaezRattyaSierra2021}.

For comparison, the weighted Bergman space $A_\gamma^p$, with
$0<p<\infty$ and $\gamma>-1$, has quasi-norm
$\|f\|_{A_\gamma^p}^p=\int_{\mathbb D}|f(z)|^p(1-|z|^2)^\gamma\,dA(z)$.
Then $H(p,p,\alpha)=A_{p\alpha-1}^p$ with equivalent norms;
away from the diagonal, the two mixed indices remain independent.
The supremum definition also gives
$H(p,\infty,0)=H^p$ if extended to $\alpha=0$, but our mixed-norm
results assume $\alpha>0$. Background on weighted Bergman spaces can
be found in \cite{DurenSchuster2004,Zhu2007}.

\emph{Derivative spaces.}
For $0<r\le\infty$, the derivative Hardy space is
\[
 S^r=\{f\in H(\mathbb D):f'\in H^r\},\qquad
 \|f\|_{S^r}=|f(0)|+\|f'\|_{H^r}.
\]
Composition operators on these spaces were studied by Roan
\cite{Roan1978} and MacCluer \cite{MacCluer1987}; the weighted problem
between different derivative Hardy spaces was considered by Contreras
and Hern\'andez-D\'iaz \cite{ContrerasHernandezDiaz2004}.
The associated mixed-norm derivative space is
\[
 S(p,q,\alpha)=\{f\in H(\mathbb D):f'\in H(p,q,\alpha)\},\qquad
 \|f\|_{S(p,q,\alpha)}=|f(0)|+\|f'\|_{H(p,q,\alpha)}.
\]
This is the notation of \cite{Borgohain2026}; its derivative condition
distinguishes $S(p,q,\alpha)$ from the derivative Hardy space $S^p$.

\emph{The $F(p,q,s)$ spaces.}
For $a\in\mathbb D$, set
$\sigma_a(z)=(a-z)/(1-\overline az)$ and
$g(z,a)=-\log|\sigma_a(z)|$. The latter is the Green function of
$\mathbb D$ with pole at $a$. Zhao introduced the family through
Green-function integrals in \cite[Definition~1.1]{Zhao1996}, including
a separate definition at $s=0$; see also \cite[Section~1]{Zhao2021}.
We work in the nontrivial range
\[
 0<p<\infty,\qquad q>-2,\qquad 0\le s<\infty,\qquad q+s>-1.
\]
The restriction $q+s>-1$ excludes the constant-only cases and
guarantees that the space contains all polynomials
\cite[Propositions~2.12--2.14]{Zhao1996},
\cite[Proposition~3.8]{Zhao2021}.
Following the M\"obius-kernel formulation used in
\cite[Section~1]{PauZhao2014}, we use the quasi-norm
\begin{equation}\label{eq:F-definition}
 \|f\|_{F(p,q,s)}^p
 =|f(0)|^p+\sup_{a\in\mathbb D}
       \int_{\mathbb D}|f'(z)|^p(1-|z|^2)^q
                         (1-|\sigma_a(z)|^2)^s\,dA(z).
\end{equation}
The space consists of the functions $f\in H(\mathbb D)$ with finite
quasi-norm. For $s>0$, the right-hand side is comparable to
\[
 |f(0)|^p+\sup_{a\in\mathbb D}
       \int_{\mathbb D}|f'(z)|^p(1-|z|^2)^q g(z,a)^s\,dA(z).
\]
This equivalence is \cite[Theorem~2.4]{Zhao1996}, restated as
\cite[Theorem~3.3]{Zhao2021}; see Section~\ref{sec:preliminaries}. The value at zero is
included because derivative integrals do not distinguish constants.
At $s=0$, this is the usual weighted Dirichlet quasi-norm, and the
supremum is redundant. For $s>0$, the little space
$F_0(p,q,s)$ consists of the functions in $F(p,q,s)$ such that
\[
 \lim_{|a|\uparrow1}\int_{\mathbb D}|f'(z)|^p(1-|z|^2)^q
                   (1-|\sigma_a(z)|^2)^s\,dA(z)=0,
\]
with the inherited quasi-norm. The little-space equivalence is
\cite[Theorem~2.5]{Zhao1996}. Zhao also uses the convention
$F_0(p,q,0)=F(p,q,0)$ \cite[Definition~1.1]{Zhao1996}; here we write
$F(p,q,0)$ directly and reserve $F_0$ for $s>0$.

The classical identifications in \cite[Section~1]{Zhao1996} and
\cite[Section~2]{Zhao2021} give, with equivalent norms,
\[
 F(p,p+\gamma,0)=A_\gamma^p\quad(0<p<\infty,\ \gamma>-1),
 \qquad F(2,1,0)=H^2.
\]
Also, $F(p,p-2,0)$ is the analytic Besov space for $1<p<\infty$,
and the diagonal derivative space satisfies
$S(p,p,\alpha)=F(p,p\alpha-1,0)$ for finite $p$ and $\alpha>0$.
For $s>0$, $F(2,0,s)=Q_s$ and $F_0(2,0,s)=Q_{s,0}$; in particular,
$F(2,0,1)=BMOA$ and $F_0(2,0,1)=VMOA$, the analytic spaces of bounded
and vanishing mean oscillation. See also \cite{Xiao2001}.
For $s>1$,
\[
 F(p,q,s)=\mathcal B^{(q+2)/p},\qquad
 F_0(p,q,s)=\mathcal B_0^{(q+2)/p},
\]
where $\mathcal B^b$ has norm
$|f(0)|+\sup_{z\in\mathbb D}(1-|z|^2)^b|f'(z)|$, and
$\mathcal B_0^b$ requires this weighted derivative to tend to zero at
the boundary. These last identifications follow from
\cite[Theorems~1.3--1.4]{Zhao1996}; see also
\cite[Theorems~2.3--2.4]{Zhao2021}.
For $0<s\le1$ the local integral condition must be retained; the
space cannot in general be identified by its pointwise growth exponent
alone \cite[Theorem~5.2]{Zhao2021}.

\emph{Weighted Dirichlet spaces and conventions.}
For $\beta\in\mathbb R$, we use the coefficient Hilbert space
\[
 \mathcal D_\beta=\left\{f(z)=\sum_{k=0}^\infty c_kz^k:
       \|f\|_{\mathcal D_\beta}^2
          =\sum_{k=0}^\infty(k+1)^{1-\beta}|c_k|^2<\infty\right\}.
\]
This is the coefficient convention in \cite{BarikSenPaul2026}. For
$\beta>-1$, the classical derivative representation gives
$\mathcal D_\beta=F(2,\beta,0)$ with equivalent norms
\cite[Section~1]{Zhao1996}, \cite[Section~2]{Zhao2021}.
With the displayed coefficient norm, $\mathcal D_1=H^2$ and
$\mathcal D_2=A^2$ hold isometrically.
We retain the coefficient norm when computing exact norms and singular
values. The derivative representation and its range of validity are
recorded in Section~\ref{sec:Schatten-exact}.

All integrability expressions are treated as quasi-norms when an index
is less than one, and $1/\infty=0$.
For nonnegative quantities, $A\approx B$ means $cB\le A\le CB$ with
positive constants independent of the varying functions or symbols,
unless otherwise specified. Space parameters and kernel orders are
fixed. In the lattice estimates and their tail versions, constants
are independent of finite supports, truncation depths and tail
indices. No uniformity as space parameters approach excluded endpoints
is asserted; Schatten comparison constants may depend on the fixed
exponent $t$.
For a bounded linear map $A:X\to Y$, its essential norm is
$\|A\|_e=\inf_K\|A-K\|$, where $K:X\to Y$ ranges over compact
linear maps. Schatten classes and approximation numbers will be defined
where they are used.

\subsection{Background and earlier results}\label{subsec:intro-history}

Weighted composition operators relate symbols to the geometry and
integrability of function spaces; see \cite{CowenMacCluer1995}.
For ordinary derivative Hardy spaces, Roan \cite{Roan1978}, MacCluer
\cite{MacCluer1987}, and Contreras and Hern\'andez-D\'iaz
\cite{ContrerasHernandezDiaz2004} relate composition and weighted
composition to Hardy-space problems. Products of composition and
differentiation between Hardy and Bergman spaces were investigated by
Hibschweiler and Portnoy \cite{HibschweilerPortnoy2005}. Higher orders,
weights and Dirichlet-space questions are treated in
\cite{AllenHellerPons2022,Fatehi2024,FatehiHammond2020,HanWang2021,HanWang2022}.
Barik, Sen and Paul \cite{BarikSenPaul2026} study the same operator
on a single weighted Dirichlet space, including essential norms,
Hilbert--Schmidt membership, multipliers and exact norm calculations.

Carleson embeddings provide a common approach to these problems.
Luecking established derivative inequalities and trace-ideal criteria
for Bergman embeddings in \cite{Luecking1985,Luecking1987}.
Pel\'aez, R\"atty\"a and Sierra \cite{PelaezRattyaSierra2021} developed
atomic decompositions and derivative embeddings for weighted mixed norms.
Separation results for sums of these operators on Bergman spaces
appear in \cite{ArroussiLiuTongYang2024,DuLiLiu2024,DuLiu2025}.
These results supply several of the special cases identified in
Subsection~\ref{subsec:intro-scope}.

For $F(p,q,s)$, the Green-kernel, M\"obius-kernel and Carleson-square
formulations originate in \cite[Theorems~2.4--2.5]{Zhao1996}.
Higher-derivative descriptions were developed in \cite{Rattya2003}
and are collected in \cite[Theorem~3.13]{Zhao2021}.
Pau and Zhao \cite{PauZhao2014} study function-value embeddings
$F(p,p\alpha-2,s)\to T_{p,s}^{\infty}(\mu)$, where
$\|h\|_{T_{p,s}^{\infty}(\mu)}^p=\sup_I|I|^{-s}\int_{S(I)}|h|^p\,d\mu$,
and apply them to Riemann--Stieltjes operators and multipliers.
Here $S(I)$ is the Carleson square defined in
Section~\ref{sec:preliminaries}. Their Theorems~3.1, 4.2 and 4.4
distinguish the critical, subcritical and supercritical regimes.
The last assumes $\alpha>1$, $s>0$, $s+p\alpha>1$ and either
$0<p\le1$, or $p>1$ with $s+p(\alpha-1)>1$, or
$1<p\le2$ with $s+p(\alpha-1)=1$.
Our positive-order derivative estimates do not use this
function-value embedding. Generalized weighted composition into
Bloch-type spaces is treated in \cite{Yang2012}, and corresponding
essential-norm estimates appear in \cite{SafarzadehEtAl2026}.

A further motivation is the distinction between pointwise growth and
an integral embedding. The necessary condition in
\cite[Theorem~3.1]{Borgohain2026} fails for the compact inclusion
$S(2,2,3/4)\to H^2$. Section~\ref{sec:counterexample} gives the
counterexample and identifies the stronger common-majorant property
that the proposed condition characterizes.

\subsection{Objectives, methods and contributions}\label{subsec:intro-objectives}

We characterize boundedness and compactness of $D_{\phi,\varphi}^n$
and obtain two-sided operator-norm and essential-norm estimates.
The source and target indices remain independent. In particular,
we do not assume that one of the two terms in
\[
 (D_{\phi,\varphi}^{n}f)'
   =\phi'(f^{(n)}\circ\varphi)
       +\phi\varphi'(f^{(n+1)}\circ\varphi)
\]
already satisfies a Carleson condition. Two fixed kernels on each
rectangle, followed by randomization of finite kernel sums, separate
these terms while controlling the source mixed or local quasi-norm.
Compact dilations give the upper essential-norm estimates. Finite kernel sums
indexed by outer rows, together with compactness and the Fatou
property, give the lower estimates without requiring weak convergence
of the test functions.

The resulting conditions are expressed in terms of the symbols'
local masses. Mixed sequence duality gives four explicit cases for
Hardy targets. For $F$ sources, the cases $s=0$ and $s>1$ reduce
to ordinary sequence norms. When $0<s\le1$, we evaluate the nested
square constraints by a dyadic covering integral if the target
exponent is at least the source exponent, and by a bottom-up power
recursion otherwise. Uniform tails of these quantities characterize
compactness. A fixed boundary pull-back measure also gives automatic
compactness for Hardy targets when the radial source exponent
exceeds the target exponent.

The same estimates yield comparisons under weight shifts, little-space
criteria and an all-weight multiplier characterization.
For symbols with relatively compact range, a Taylor expansion gives
geometric finite-rank approximations and rank-one expansions, with
nuclearity for Banach sources and targets.
In the Hilbert setting, separated finite matrices control the small
Schatten exponents, and a weighted-shift calculation gives exact
singular values for monomial weights and linear symbols.
The precise parameter ranges and comparisons with earlier results
are given in Subsection~\ref{subsec:intro-scope}.

\subsection{Organization of the paper}\label{subsec:intro-organization}

Section~\ref{sec:preliminaries} collects the common analytic and
compactness tools. Section~\ref{sec:counterexample} treats the
counterexample and the common radial majorant.
Sections~\ref{sec:hardy-targets} and \ref{sec:hardy-derivatives}
address Hardy and derivative Hardy targets, respectively.
Section~\ref{sec:F-spaces} proves the $F$-space criteria, evaluates
the sequence condition and treats little spaces and examples.
Section~\ref{sec:comparison-interior} contains the comparison theorem,
supremum targets and symbols with relatively compact range.
Section~\ref{sec:multipliers-approximation} concerns multipliers and
finite-rank approximation, and Section~\ref{sec:Schatten-exact}
concludes with Schatten classes and exact norms.

\subsection{Scope and relation to earlier results}\label{subsec:intro-scope}\label{sec:parameter-scope}

The following comparisons specify which earlier statements are recovered,
which subdomains follow directly from existing theorems, and what the
unified results establish on the remaining stated parameters.
Space identifications here are understood with equivalent norms unless
an isometric assertion is made explicitly.

\smallskip\noindent
\emph{Hardy targets: Theorems~\ref{thm:mixed-hardy}--\ref{thm:mixed-boundary}.}
Theorem~\ref{thm:mixed-hardy} treats $H(p,q,\alpha)\to H^r$ for
$0<p,q\le\infty$, $\alpha>0$, $0<r<\infty$ and every $n\ge0$.
Its finite-index boundedness and norm criterion follows from
\cite[Theorem~3]{PelaezRattyaSierra2021} with
$\omega(t)=\alpha q(1-t)^{\alpha q-1}$ and
$\widehat\omega(t)=(1-t)^{\alpha q}$.
This includes $p\ne q$, not only the Bergman diagonal.
The theorem here includes both infinite source indices and the uniform
tail characterization of the essential norm.
Theorem~\ref{thm:derivative-hardy} transfers these conclusions to
$S(p,q,\alpha)\to H^r$ for $n\ge1$, with exact equality of operator
and essential norms after lowering the order by one.
Theorem~\ref{thm:mixed-boundary} replaces the radial measure family
by a single boundary measure. For a nonzero weight, it includes the
necessary condition $|\varphi^*|<1$ almost everywhere.
Corollary~\ref{cor:Hardy-automatic-compactness} then gives automatic
compactness when $q>r$; Example~\ref{ex:Hardy-critical-q} shows that
$q\ge r$ cannot replace this condition in general.

\smallskip\noindent
\emph{Derivative Hardy targets: Theorem~\ref{thm:mixed-derivative}.}
For $S(p,q,\alpha)\to S^r$ the same mixed indices are allowed,
with $n\ge1$ and $0<r<\infty$.
The zero-order $S^p\to S^q$ theory of
\cite{ContrerasHernandezDiaz2004} concerns ordinary Hardy derivatives,
not the three-parameter derivative condition used here.
On the diagonal $p=q$, the derivative component of the source is
$A_{p\alpha-1}^p$; for $1<p,r<\infty$, the arbitrary-measure
separation theorem \cite[Theorem~1]{DuLiu2025}, applied on each
circle, gives this boundedness subcase.
Theorem~\ref{thm:mixed-derivative} treats independent mixed indices,
their infinite endpoints and a uniform essential-norm tail.
Corollary~\ref{cor:Hardy-derivative-automatic} gives automatic
compactness for $q>r$ also in this derivative-target problem.

\smallskip\noindent
\emph{Local $F$ spaces: Theorems~\ref{thm:F-boundedness}--\ref{thm:F-little-operators}.}
For $n\ge1$ these theorems allow every pair
$0<p_i<\infty$, $q_i>-2$, $s_i\ge0$, $q_i+s_i>-1$ ($i=1,2$),
without ordering $p_1,p_2$.
Theorem~\ref{thm:F-boundedness} gives boundedness and norm estimates;
Theorem~\ref{thm:F-essential-norm} gives the essential norm and
compactness; Theorem~\ref{thm:F-little-operators} gives the little-source
and little-target variants. In the latter, $s_1=0$ is treated with
source $F(p_1,q_1,0)$, and a little target requires $s_2>0$ and
$\phi,\phi\varphi\in F_0(p_2,q_2,s_2)$. For $s_1>0$ this boundedness
criterion uses a little source; when the tail vanishes, the entire
large source maps compactly into the little target.
Proposition~\ref{prop:explicit-sequence} evaluates the sequence
condition for every $s_1\ge0$, including both $p_2\ge p_1$ and
$p_2<p_1$ when $0<s_1\le1$.

At $(p_i,q_i,s_i)=(2,\beta,0)$ with the same $\beta>-1$, the
boundedness and essential-norm criteria recover
\cite[Theorems~1.1--1.2]{BarikSenPaul2026}, without requiring their
preliminary Carleson assumption. Taking also $\phi=1$ and $n=1$
gives the weighted Dirichlet case in \cite[Appendix]{HanWang2022};
$\beta=0$ gives the classical Dirichlet result of
\cite{AllenHellerPons2022}.
The larger zero-local-parameter subdomain
$s_1=s_2=0$, $1\le p_1,p_2<\infty$, is already obtained from
\cite[Theorem~1]{DuLiLiu2024}: putting $g=f'$ reduces the differentiated
operator to
\[
 g\longmapsto\phi'(g^{(n-1)}\circ\varphi)
       +\phi\varphi'(g^{(n)}\circ\varphi):
 A_{q_1}^{p_1}\longrightarrow A_{q_2}^{p_2},
\]
with one additional scalar coordinate.
For $q_1=q_2=s_1=s_2=0$, boundedness and compactness for all
$p_1,p_2>0$ also follow from
\cite[Theorem~1.8]{ArroussiLiuTongYang2024}.
If $s_1=0$ and $1<p_1,p_2<\infty$, \cite[Theorem~1]{DuLiu2025}
gives norm separation for each target measure; the local target
requires the uniform tail used in Theorem~\ref{thm:F-essential-norm}.

For $s_1>0$ and $s_2>1$, the target is
$\mathcal B^{(q_2+2)/p_2}$, and the large-source boundedness, compactness and essential-norm
statements reduce to \cite{Yang2012} and
\cite[Lemma~3.1 and Theorem~3.1]{SafarzadehEtAl2026}.
The full theorems retain both local parameters, including
$0<s_1,s_2\le1$, and combine these classical regimes with the
other admissible index combinations.
Setting $p_1=p_2=2$, $q_1=q_2=0$ and $s_1,s_2>0$ gives positive-order
$Q_{s_1}\to Q_{s_2}$ criteria; $s_1=s_2=1$ gives $BMOA$ and,
with little spaces, $VMOA$.
These identifications do not assert a general zero-order composition
theorem on those spaces.
Corollary~\ref{cor:F-explicit-regions} records the separate formulas,
and Corollary~\ref{cor:F-automatic-compactness} gives automatic
compactness when $s_2=0$ and either $p_2<p_1$ or $s_1>1$.
Example~\ref{ex:BMOA-local-dichotomy} illustrates a genuinely local
case in which the two exponent regimes differ.

\smallskip\noindent
\emph{Comparison and supremum targets: Theorems~\ref{thm:F-comparison} and \ref{thm:supremum-targets}.}
The weighted Dirichlet comparison in
\cite[Proposition~8.1]{BarikSenPaul2026} corresponds to equal source
and target triples with $p_1=p_2=2$, $s_1=s_2=0$ and equal weight
shifts. Theorem~\ref{thm:F-comparison} allows distinct admissible
triples and independent shifts $0\le h\le k$; a strict inequality
$k>h$ forces compactness, including a little target when $s_2>0$.
Theorem~\ref{thm:supremum-targets} supplies the $H^\infty$ and
$S^\infty$ counterparts for the source families listed in its table.
Weighted-type target problems were previously considered in
\cite{SafarzadehEtAl2026,Stevic2009,Yang2012}; the endpoint formulation
is kept separate from the finite Hardy integral criteria.

\smallskip\noindent
\emph{Zero order and relatively compact symbols: Proposition~\ref{prop:subcritical-zero-order} and Theorem~\ref{thm:relatively-compact-symbol}.}
For general self-maps, Proposition~\ref{prop:subcritical-zero-order}
treats zero-order Hardy-target maps when $\alpha+1/p<1$ for
$S(p,q,\alpha)$, or $(q+2)/p<1$ for an $F$ source.
It follows from compact inclusion in $H^\infty$ and also gives a
zero-order derivative-target reduction for $S(p,q,\alpha)$.
When $\|\varphi\|_\infty<1$, Theorem~\ref{thm:relatively-compact-symbol}
allows every $n\ge0$ and every listed source and target, including
their quasi-Banach and little-space cases. For a fixed target,
boundedness and compactness are characterized by the same symbol
memberships, independently of the order and the chosen source;
comparison constants may depend on both.
Its Dirichlet specialization, Corollary~\ref{cor:Dirichlet-interior},
gives the criterion requested in \cite[Problem~8.4]{BarikSenPaul2026}.

\smallskip\noindent
\emph{Multipliers and approximation: Theorems~\ref{thm:F-all-weights} and \ref{thm:interior-approximation}.}
The all-weight equivalence in \cite[Theorem~1.4]{BarikSenPaul2026}
is the $F(2,\beta,0)$ case of Theorem~\ref{thm:F-all-weights}.
The latter treats all admissible triples and their little spaces,
adds compactness, and gives weight-norm estimates for every fixed
$n\ge1$.
Theorem~\ref{thm:interior-approximation} treats every $n\ge0$ under
$\|\varphi\|_\infty<1$ on the full list of source and target spaces
in Theorem~\ref{thm:relatively-compact-symbol}.
It gives geometric approximation numbers and summable rank-one
expansions, including quasi-Banach cases; nuclearity is asserted
only for Banach sources and targets.
Related weighted sup-norm nuclearity is studied in
\cite{BonetGomezColladoJordaJornet2021}.
The rate here is a uniform exponential upper bound, not an exact
rate formula for arbitrary symbols.

\smallskip\noindent
\emph{Schatten classes: Theorem~\ref{thm:Dirichlet-Schatten}.}
The parameter domain is $\beta_1\in\mathbb R$, $\beta_2>-1$,
$n\ge0$, $\beta_1+2n>1$, and every $0<t<\infty$.
The choice $t=2$, $\beta_1=\beta_2>-1$ and $n\ge1$ recovers
\cite[Theorem~1.3]{BarikSenPaul2026}.
The full statement permits different indices and all Schatten exponents,
including $n=0$, $\beta_1>1$, and $n\ge2$, $1-2n<\beta_1\le-1$.
For $\beta_2>1$, the target has a weighted Bergman representation,
so a scalar pull-back and \cite{Luecking1987} already give the criterion.
For $\beta_2=1$, the Hardy boundary measure provides another scalar
route, as explained after the theorem.
The two-term proof applies to the whole stated range $\beta_2>-1$,
using finite-matrix separation also when $0<t<2$.

\smallskip\noindent
\emph{Exact models: Theorem~\ref{thm:monomial-exact}.}
For $\phi=z^m$ and $\varphi=\rho z$, this theorem permits arbitrary
real Dirichlet indices, $m,n\ge0$ and $|\rho|\le1$.
At $m=0$, $n=1$ and equal indices it recovers the coefficient
expression used in \cite[Section~7]{BarikSenPaul2026}; at $m=0$
and $\beta_1=\beta_2=1$ it recovers the Hardy norm in
\cite[Theorem~3.5]{Fatehi2024}.
The statement also gives the complete singular-value sequence,
all Schatten thresholds and the exact exponential rate for
$0<|\rho|<1$. These are coefficient-norm identities, unlike the
equivalent-norm comparisons used in the general criteria.

\smallskip\noindent
\emph{Parameters not treated by the general criteria.}
The displayed domains exclude $\alpha=0$ and $q_i=-2$.
General zero-order maps on derivative sources outside the subcritical
or relatively compact-symbol cases are not obtained by setting $n=0$
in a positive-order theorem. This includes the critical value-growth
thresholds and the remaining supercritical zero-order problem in
\cite{Borgohain2026}. Proposition~\ref{prop:radial-majorant}
characterizes a stronger majorant property; it does not supply a
general boundedness criterion in that regime.
The positivity used in the boundary-measure argument is genuine:
at $\alpha=n=0$, the extended convention $H(p,\infty,0)=H^p$
allows bounded inclusions $H^p\to H^r$ for $0<r\le p\le\infty$,
although the identity symbol has boundary modulus one.
For arbitrary symbols the Schatten criterion does not address
$\beta_1+2n\le1$ or $\beta_2\le-1$.
These are limits of the present results, not a list of cases known
to be open in the literature.

\section{Preliminaries}\label{sec:preliminaries}

We use the spaces and parameter conventions of
Subsection~\ref{subsec:intro-spaces}. The following kernel estimates,
compact approximations and test functions will be used throughout.

\subsection{Kernel estimates and Carleson squares}

We first record two standard integral estimates. The area estimate is
the supercritical case of \cite[Lemma~3.7]{Zhao2021}; see also
\cite{DurenSchuster2004,Zhu2007}.

\begin{lemma}\label{lem:circle-area-kernels}
For $u>1$, $\beta>-1$ and $\lambda>\beta+2$,
\[
 \int_0^{2\pi}\frac{d\theta}{|1-re^{i\theta}|^u}
       \le C(1-r)^{1-u},\qquad 0\le r<1,
\]
and
\begin{equation}\label{eq:weighted-area-kernel}
 \int_{\mathbb D}\frac{(1-|z|^2)^\beta}{|1-\overline wz|^\lambda}\,dA(z)
       \le C(1-|w|^2)^{\beta+2-\lambda},\qquad w\in\mathbb D.
\end{equation}
\end{lemma}

\begin{proof}
On $[-\pi,\pi]$, $|1-re^{i\theta}|\approx1-r+|\theta|$,
which proves the circle estimate by integration. For the area estimate,
angular integration reduces the question to
$\int_0^1(1-t)^\beta(1-|w|t)^{1-\lambda}\,dt$.
Use $1-|w|t\ge((1-|w|)+(1-t))/2$ and substitute
$u=(1-t)/(1-|w|)$. The resulting integral is finite because
$\beta>-1$ and $\lambda>\beta+2$; see also \cite{DurenSchuster2004}.
\end{proof}

For $s>0$, the inequality $-\log x\ge(1-x^2)/2$, $0<x<1$,
gives one direction of the Green-function and M\"obius-kernel comparison.
For the reverse inequality fix $a$ and denote the integral in
\eqref{eq:F-definition}, before taking the supremum, by $I_a(f)$.
On $|\sigma_a(z)|\ge1/2$, the two kernels are comparable.
On $|\sigma_a(z)|<1/2$, local submean estimates for $|f'|^p$, using
disks of radius comparable to $1-|a|^2$ contained in
$\{|\sigma_a|<3/4\}$, give
\[
 |f'(z)|^p\le C(1-|a|^2)^{-q-2}I_a(f).
\]
Both the weight $(1-|z|^2)^q$ and the Jacobian under
$z=\sigma_a(\zeta)$ have fixed comparisons on $|\zeta|<1/2$.
Thus the remaining Green integral is at most
\[
 C I_a(f)\int_{|\zeta|<1/2}
                     \left(\log\frac1{|\zeta|}\right)^s\,dA(\zeta)
 \le C' I_a(f).
\]
The last integral is finite, and the constants are independent of $a$.
Taking suprema proves the asserted norm equivalence. This argument
also shows that the two defining integrals vanish simultaneously as
$|a|\uparrow1$, so the little-space definition is consistent with the
Green-function convention as well.

For an arc $I\subset\partial\mathbb D$, let $|I|$ be its normalized
length and let
$S(I)=\{\rho\zeta:\zeta\in I,\ 1-|I|\le\rho<1\}$.
The full circle is allowed, with $S(\partial\mathbb D)=\mathbb D$.
For a positive Borel measure $\mu$, write
$\|\mu\|_s=\sup_I\mu(S(I))/|I|^s$.
This square convention agrees with \cite[Section~1]{PauZhao2014}.
Changing the square height by a fixed positive factor gives an equivalent
Carleson norm provided that the full-disk mass is retained; otherwise
shorter squares need not detect mass in a central disk.
Thus $\|\mu\|_0=\mu(\mathbb D)$, and for $s>0$ finiteness means that
$\mu$ is an $s$-Carleson measure. We use the
standard kernel characterization \cite[Lemma~2.2]{Zhao1996},
\cite[Lemma~3.2]{Zhao2021}:
\[
 \|\mu\|_s\approx\sup_{a\in\mathbb D}
       \int_{\mathbb D}\frac{(1-|a|^2)^s}
                              {|1-\overline az|^{2s}}\,d\mu(z).
\]
For $s=0$, this is equality. For $s>0$, the lower bound follows by
choosing $a=(1-|I|)\zeta_I$, where $\zeta_I$ is the midpoint of $I$;
for arcs of length at least $1/2$ one can use $a=0$.
For the upper bound, set $h=1-|a|$ and split the disk into squares of
lengths $h,2h,4h,\ldots$, centered in the direction of $a$, stopping
at the full circle. On the successive differences the kernel is at
most $C h^s/(2^k h)^{2s}$, whereas the measure is at most
$\|\mu\|_s(2^k h)^s$. Summing the geometric series proves the bound;
$|a|\le1/2$ is controlled directly by $\mu(\mathbb D)$.
Using $1-|\sigma_a(z)|^2=(1-|a|^2)(1-|z|^2)/|1-\overline az|^2$,
we obtain
\begin{equation}\label{eq:F-squares}
 \|f\|_{F(p,q,s)}^p
 \approx |f(0)|^p+
       \sup_I\frac{1}{|I|^s}\int_{S(I)}
                  |f'(z)|^p(1-|z|^2)^{q+s}\,dA(z).
\end{equation}
For $s>0$ this is the square characterization in
\cite[Theorem~2.4]{Zhao1996} and \cite[Theorem~3.3]{Zhao2021};
the bounded-measure case $s=0$ is noted after those theorems.

We will use the integral-mean estimate
\begin{equation}\label{eq:mean-growth}
 M_p(g,r)\le (1-r)^{-\alpha}\|g\|_{H(p,q,\alpha)},
 \qquad g\in H(p,q,\alpha),\quad 0\le r<1.
\end{equation}
This estimate is recorded in \cite[Section~3]{Arevalo2015}. Indeed, for
$q<\infty$, the monotonicity of $M_p(g,t)$ gives
\[
 \|g\|_{H(p,q,\alpha)}^q
 \ge\alpha q\int_r^1(1-t)^{\alpha q-1}M_p(g,r)^q\,dt
 =(1-r)^{\alpha q}M_p(g,r)^q.
\]
For $q=\infty$, \eqref{eq:mean-growth} follows directly from the definition.

\subsection{Local estimates and compact approximation}

The completeness and local evaluation properties of $F(p,q,s)$ are
classical; see \cite[Theorem~2.10 and Proposition~2.11]{Zhao1996}
for $p\ge1$ and \cite[p.~485]{PauZhao2014} for $0<p<1$.
We include the arguments with the exact quasi-norms used here, together
with the compact-dilation facts needed below.

\begin{lemma}\label{lem:space-basics}\label{lem:F-dilations}
Each of $H^r$, $S^r$, $H(p,q,\alpha)$, $S(p,q,\alpha)$ and
$F(p,q,s)$, with the parameters specified above, is complete and
has the Fatou property: if $f_k\to f$ locally uniformly, then
$\|f\|\le\liminf_k\|f_k\|$, with infinity allowed.
For any of these spaces $X$, every $0<R<1$ and integer $m\ge0$,
\begin{equation}\label{eq:source-local-interior}
 \max_{|z|\le R}|f^{(m)}(z)|\le C_{R,m,X}\|f\|_X.
\end{equation}
The dilation $C_\rho f(z)=f(\rho z)$ is a compact contraction on $X$
for $0<\rho<1$, and
\[
 \lim_{\rho\uparrow1}\sup_{\|f\|_X\le1}
       \max_{|z|\le R}|(f-C_\rho f)^{(m)}(z)|=0.
\]
\end{lemma}

\begin{proof}
For $H^r$ and $H(p,q,\alpha)$, subharmonicity and the integral-mean
bound give local pointwise estimates; at an infinite exponent the
supremum bound suffices. For $F(p,q,s)$, the defining integral at
$a=0$ controls $f'$ in $L^p((1-|z|^2)^{q+s}dA)$ and hence locally
by subharmonicity. In a derivative space integrate from $0$ to bound
$f$ itself. Cauchy's formula then gives \eqref{eq:source-local-interior}.
Fatou's lemma in each defining integral, followed by the required
suprema, proves the Fatou property; the same assertion for a supremum
norm is immediate. A quasi-norm Cauchy sequence therefore has a local
analytic limit, and applying Fatou to its differences proves
completeness.

Monotonicity of integral means makes $C_\rho$ contractive on the Hardy
and mixed-norm spaces, and also on their derivative spaces since
$(C_\rho f)'=\rho f'(\rho\,\cdot)$. For $F(p,q,s)$, let $P_\rho$
be the Poisson kernel. Subharmonicity gives
\begin{equation}\label{eq:F-Poisson-dilation}
 |f'(\rho z)|^p\le\int_0^{2\pi}P_\rho(\theta)
                 |f'(e^{i\theta}z)|^p\,\frac{d\theta}{2\pi}.
\end{equation}
Integrating against the weight defining the $F$ quasi-norm and
rotating $z$ replaces $a$ by $e^{i\theta}a$. Tonelli's theorem and
$\int P_\rho\,dm=1$, together with
$(C_\rho f)'=\rho f'(\rho\,\cdot)$, yield
\[
 \|C_\rho f\|_{F(p,q,s)}^p
 \le |f(0)|^p+\rho^p\bigl(\|f\|_{F(p,q,s)}^p-|f(0)|^p\bigr)
 \le \|f\|_{F(p,q,s)}^p.
\]
This uses subharmonicity and remains valid for $p<1$.

For later use, \eqref{eq:F-squares} gives
\begin{equation}\label{eq:F-bounded-derivative}
 \|h\|_{F(p,q,s)}\le C\bigl(|h(0)|+\|h'\|_\infty\bigr)
                        \quad\text{when }h'\in H^\infty,
\end{equation}
since the weighted area of $S(I)$ is at most $C|I|^{q+s+2}$.
A bounded sequence in any of the stated spaces has, by Montel and
Fatou, a locally convergent subsequence with its limit in that space.
After a fixed dilation, the functions and their first derivatives
converge uniformly on $\mathbb D$. The defining norms, or
\eqref{eq:F-bounded-derivative}, turn this into norm convergence.
Thus $C_\rho$ is compact. Finally,
$(f-C_\rho f)^{(m)}=f^{(m)}-\rho^mf^{(m)}(\rho\,\cdot)$;
\eqref{eq:source-local-interior} at orders $m$ and $m+1$ proves the
last assertion.
\end{proof}

\begin{lemma}\label{lem:essential-test}
Let $X$ be a quasi-Banach space and let $Y\subset H(\mathbb D)$ be a
complete quasi-normed space whose norm convergence implies local
uniform convergence and whose quasi-norm has the Fatou property.
Let $A:X\to Y$ be bounded. If $\sup_k\|f_k\|_X\le B<\infty$,
$B>0$, and $Af_k\to0$ locally uniformly, then
\[
 \limsup_k\|Af_k\|_Y\le C_Y B\|A\|_e.
\]
Here $C_Y$ depends only on the quasi-triangle constant of $Y$.
Compact operators into any complete quasi-normed space are closed in
operator quasi-norm; in particular, essential norm zero is equivalent
to compactness.
\end{lemma}

\begin{proof}
Choose a subsequence realizing the limit superior. For a compact
$K:X\to Y$, pass to a further subsequence with $Kf_k\to h$ in $Y$.
The local limit of $(A-K)f_k$ is $-h$, so
$\|h\|_Y\le\liminf_k\|(A-K)f_k\|_Y\le B\|A-K\|$.
Let $\kappa$ be a quasi-triangle constant of $Y$. Since $Kf_k\to h$,
$\limsup_k\|Kf_k\|_Y\le\kappa\|h\|_Y$. Hence
$\limsup_k\|Af_k\|_Y\le\kappa(1+\kappa)B\|A-K\|$.
Taking the infimum over $K$ proves the estimate.
For closedness, operator-quasi-norm approximation by compact maps
makes the image of the unit ball totally bounded; completeness makes
its closure compact.
\end{proof}

\begin{lemma}\label{lem:primitive-reduction}
Let $X\subset H(\mathbb D)$ and $Y$ be quasi-Banach spaces, and set
$S(X)=\{f\in H(\mathbb D):f'\in X\}$, with $\|f\|_{S(X)}=|f(0)|+\|f'\|_X$.
If a linear analytic formula $A$ on $S(X)$ annihilates constants,
then $A$ and $AJ:X\to Y$, where $Jg(z)=\int_0^zg(\zeta)\,d\zeta$,
are bounded simultaneously, with equal norms. They are compact
simultaneously and have equal essential norms whenever bounded.
\end{lemma}

\begin{proof}
Differentiation $\partial:S(X)\to X$ is contractive, $J$ is isometric,
and $\partial J=I$. Since $A=(AJ)\partial$, the two factorizations
give equality of norms and equivalence of compactness. Composing
compact approximants with $J$ or $\partial$ gives both inequalities
for the essential norms.
\end{proof}

The little-space square condition is
\cite[Theorem~2.5]{Zhao1996}. The dilation and polynomial-closure
characterization is recorded in \cite[Theorem~3.10(ii)]{Zhao2021};
for $p>1$ and $q\ge0$, it appears in
\cite[Proposition~2.3]{LindstromPalmberg2007}.
We give the proof for the quasi-norm in \eqref{eq:F-definition}.

\begin{lemma}\label{lem:F-little-approximation}
Let $0<p<\infty$, $q>-2$, $s>0$ and $q+s>-1$.
For $f\in F(p,q,s)$, the following conditions are equivalent:
\begin{enumerate}
 \renewcommand{\labelenumi}{\textup{(\arabic{enumi})}}
 \setlength{\itemsep}{2pt}\setlength{\parskip}{0pt}
 \item $f\in F_0(p,q,s)$.
 \item $\displaystyle\lim_{|I|\downarrow0}\frac{1}{|I|^s}
       \int_{S(I)}|f'(z)|^p(1-|z|^2)^{q+s}\,dA(z)=0$.
 \item $\displaystyle\lim_{\rho\uparrow1}
                         \|f-C_\rho f\|_{F(p,q,s)}=0$.
 \item $f$ is in the closure of the polynomials in $F(p,q,s)$.
\end{enumerate}
In particular, $F_0(p,q,s)$ is closed and complete, and every
function analytic on a neighborhood of $\overline{\mathbb D}$ belongs
to it. Each $C_\rho$, $\rho<1$, maps $F(p,q,s)$ compactly into
$F_0(p,q,s)$. For $s=0$ (and hence $q>-1$), polynomials are dense
in $F(p,q,0)$ as well.
\end{lemma}

\begin{proof}
Put $d\mu_f(z)=|f'(z)|^p(1-|z|^2)^{q+s}\,dA(z)$. This is a finite
$s$-Carleson measure by \eqref{eq:F-squares}. The defining integral
for the little space is
\[
 \int_{\mathbb D}\frac{(1-|a|^2)^s}{|1-\overline az|^{2s}}\,d\mu_f(z).
\]
If it tends to zero as $|a|\uparrow1$, choose
$a=(1-|I|)\zeta_I$ as in the proof of \eqref{eq:F-squares}. The
kernel is at least $c|I|^{-s}$ on $S(I)$, proving (2).
Conversely, suppose (2) holds. Given $\varepsilon>0$, choose
$h_0>0$ so that $\mu_f(S(I))\le\varepsilon|I|^s$ when
$|I|\le h_0$. For $h=1-|a|$ small, use the concentric dyadic squares
of lengths $h,2h,4h,\ldots$ from the proof of
\eqref{eq:F-squares}. The terms whose outer square has length at
most $h_0$ have sum at most $C\varepsilon\sum_{k\ge0}2^{-ks}$.
In the remaining region the kernel is at most
$C h^s h_0^{-2s}$, after changing the constant for the single
transition square. Its contribution tends to zero, since $\mu_f$
is finite. Letting first $|a|\uparrow1$ and then
$\varepsilon\downarrow0$ proves (1).

The space described by (2) is closed. Indeed, the inequality
$|v+w|^p\le\max\{1,2^{p-1}\}(|v|^p+|w|^p)$ and
\eqref{eq:F-squares} give
\[
 \frac{\mu_f(S(I))}{|I|^s}
 \le C\|f-g\|_{F(p,q,s)}^p
               +C\frac{\mu_g(S(I))}{|I|^s}.
\]
Approximate $f$ in quasi-norm by members $g$ satisfying (2), and
then let $|I|\downarrow0$. Completeness follows from
Lemma~\ref{lem:F-dilations}. If $f'$ is bounded, then
$\mu_f(S(I))/|I|^s\le C\|f'\|_\infty^p|I|^{q+2}\to0$.
This proves the assertion about functions analytic past the closed
disk, including the polynomials and all dilations $C_\rho f$.

Assume (2). Integrating \eqref{eq:F-Poisson-dilation} over $S(I)$
and rotating the variable gives
\[
 \frac{\mu_{C_\rho f}(S(I))}{|I|^s}
 \le\rho^p\int_0^{2\pi}P_\rho(\theta)
       \frac{\mu_f(S(e^{i\theta}I))}{|I|^s}\,\frac{d\theta}{2\pi}.
\]
Thus the small-arc bound in (2) holds uniformly for $C_\rho f$,
$0<\rho<1$. By the elementary inequality for a $p$th power, it also
makes $\mu_{f-C_\rho f}(S(I))/|I|^s$ uniformly small on sufficiently
short arcs.

We claim that $\mu_{f-C_\rho f}(\mathbb D)\to0$ as $\rho\uparrow1$.
On each compact disk this follows from local uniform convergence of
the derivatives. On $R<|z|<1$, monotonicity of integral means gives
\[
 M_p((f-C_\rho f)',t)^p
 \le C\bigl(M_p(f',t)^p+\rho^pM_p(f',\rho t)^p\bigr)
 \le C' M_p(f',t)^p.
\]
After integration in $t$ with the weight
$2t(1-t^2)^{q+s}$, the last bound is uniformly small for $R$ close
to $1$, because $\mu_f(\mathbb D)<\infty$. This proves the claim.
For arcs with $|I|\ge h_0$, we now have
\[
 \frac{\mu_{f-C_\rho f}(S(I))}{|I|^s}
       \le h_0^{-s}\mu_{f-C_\rho f}(\mathbb D)\longrightarrow0.
\]
Together with the small-arc bound and
$(f-C_\rho f)(0)=0$, equation \eqref{eq:F-squares} proves (3).

For fixed $\rho<1$, the Taylor polynomials of $C_\rho f$ and their
first derivatives converge uniformly on $\overline{\mathbb D}$.
Equation \eqref{eq:F-bounded-derivative} turns this into convergence
in $F(p,q,s)$. Hence (3) implies (4). Closedness and the polynomial
case already proved show that (4) implies (1). Finally,
$C_\rho:F(p,q,s)\to F(p,q,s)$ is compact by
Lemma~\ref{lem:F-dilations}, its range is in $F_0(p,q,s)$, and this
subspace is closed with the inherited quasi-norm. It is therefore
compact also with this smaller codomain.
For $s=0$ and $q>-1$, the same global weighted-integral argument,
which needs no small-square assumption, gives $\|f-C_\rho f\|_{F(p,q,0)}\to0$ directly.
Taylor approximation of $C_\rho f$ then proves polynomial density.
\end{proof}

\subsection{A fixed partition and two kernels}

Put $r_j=1-2^{-j}$ and $\delta_j=2^{-j}$ for $j=0,1,\ldots$. For
$0\le l<2^{j+3}$, let
\[
 Q_{j,l}=\left\{z:r_j\le |z|<r_{j+1},\quad
       \frac{2\pi l}{2^{j+3}}\le\arg z<\frac{2\pi(l+1)}{2^{j+3}}\right\},
 \quad
 a_{j,l}=\frac{r_j+r_{j+1}}2
       \exp\!\left(\frac{2\pi i(l+1/2)}{2^{j+3}}\right).
\]
The displayed definition is for $z\ne0$, with $\arg z\in[0,2\pi)$.
We adjoin the origin to $Q_{0,0}$; the sets $Q_{j,l}$ then form a
Borel partition of $\mathbb D$, and
$|1-\overline{a_{j,l}}w|\approx\delta_j$ on $Q_{j,l}$.
Notice also that $|a_{j,l}|\ge1/4$. For $0<u,v\le\infty$, write
\[
 \|c\|_{\ell^{u,v}}
 =\left\|\left\{\|\{c_{j,l}\}_l\|_{\ell^u}\right\}_j\right\|_{\ell^v},
\]
with the usual supremum convention at infinity. All sums in $l$ have the
range specified above.

\begin{lemma}\label{lem:two-kernels}
Fix $A>0$, positive integers $M$ and $n$, and $0<r<\infty$. For
$u=0,1$, put
\[
 g_{j,l,u}(z)=\frac{\delta_j^{M+u-A}}
                    {(1-\overline{a_{j,l}}z)^{M+u}}.
\]
Then, for $w\in Q_{j,l}$ and $x,y\in\mathbb C$,
\begin{equation}\label{eq:two-kernels}
 \frac{|x|^r}{\delta_j^{r(A+n-1)}}
 +\frac{|y|^r}{\delta_j^{r(A+n)}}
 \le C\sum_{u=0}^1
       |xg_{j,l,u}^{(n-1)}(w)+yg_{j,l,u}^{(n)}(w)|^r.
\end{equation}
The constant is independent of $j,l,w,x,y$.
\end{lemma}

\begin{proof}
Write $a=a_{j,l}$ and $\delta=\delta_j$. If $(b)_k=b(b+1)\cdots(b+k-1)$,
with $(b)_0=1$, then
\[
 g_{j,l,u}^{(n-1)}(w)
 =\frac{(M+u)_{n-1}\overline a^{\,n-1}\delta^{M+u-A}}
             {(1-\overline aw)^{M+u+n-1}},
 \qquad
 g_{j,l,u}^{(n)}(w)
 =\frac{(M+u+n-1)\overline a}{1-\overline aw}
       g_{j,l,u}^{(n-1)}(w).
\]
Since $|a|\ge1/4$ and $|1-\overline aw|\approx\delta$ on $Q_{j,l}$,
we have $|\delta^{A+n-1}g_{j,l,u}^{(n-1)}(w)|\approx1$.
Set $x_0=\delta^{-A-n+1}x$, $y_0=\delta^{-A-n}y$, and
$\gamma=\overline a\delta/(1-\overline aw)$. Both $|\gamma|$ and
$|\gamma|^{-1}$ are bounded by fixed constants. After multiplication by
factors whose moduli are bounded above and below, the two expressions
on the right of \eqref{eq:two-kernels} become
$x_0+(M+n-1)\gamma y_0$ and $x_0+(M+n)\gamma y_0$. Their difference is
$\gamma y_0$, which first controls $y_0$ and then $x_0$. Taking $r$th powers
proves the assertion, using
$|v+w|^r\le\max\{1,2^{r-1}\}(|v|^r+|w|^r)$.
\end{proof}

\section{A counterexample and a common radial majorant}\label{sec:counterexample}

The necessary condition in \cite[Theorem~3.1]{Borgohain2026} is false.
That theorem assumes $1\le p<\infty$, $1<q\le\infty$,
$1/p+1/q\le1$, $0<\alpha<1$ and $\alpha+1/p>1$, with an analytic
weight and an analytic self-map of $\mathbb D$. Take $p=q=2$,
$\alpha=3/4$, $\phi\equiv1$ and $\varphi(z)=z$. Then
$1/p+1/q=1$ and $\alpha+1/p=5/4>1$, and the symbols satisfy the
analytic hypotheses. The mixed-norm normalization agrees with that
paper. The inclusion
$S(2,2,3/4)\to H^2$ is compact by
Proposition~\ref{prop:compact-inclusion} below, whereas the proposed
necessary quantity is
\[
 \sup_{0\le t<1}\int_0^{2\pi}\frac{d\theta}{(1-t^2)^{1/2}}
       =2\pi\sup_{0\le t<1}(1-t^2)^{-1/2}=\infty.
\]
In the original proof on p.~211, a test parameter $w$, initially
fixed, is replaced by $\varphi(z)$ inside the integral. The bound
for a fixed test function does not justify this substitution.
Proposition~\ref{prop:radial-majorant} identifies the stronger
property described by the integral. Theorems~\ref{thm:mixed-hardy}
and \ref{thm:derivative-hardy} subsequently treat Hardy targets for
mixed-norm sources and their positive-order derivative counterparts.

\begin{proposition}\label{prop:compact-inclusion}
Let $0<p,q\le\infty$ and $0<\alpha<1$. The inclusion
$S(p,q,\alpha)\to H^p$, equal to $D_{1,\operatorname{id}}^0$, is compact.
For $1\le p\le\infty$ it also satisfies
\[
 \|f\|_{H^p}\le |f(0)|+\frac{\|f'\|_{H(p,q,\alpha)}}{1-\alpha},
 \qquad f\in S(p,q,\alpha).
\]
If, in addition, $p<\infty$, $1<q\le\infty$, $1/p+1/q\le1$ and
$\alpha+1/p>1$, then $\phi\equiv1$ and $\varphi(z)=z$ satisfy all
hypotheses of \cite[Theorem~3.1]{Borgohain2026}, but not its asserted
necessary integral condition. In particular, $p=q=2$ and $\alpha=3/4$
give a counterexample in that range.
\end{proposition}

\begin{proof}
For $1\le p\le\infty$, radial integration and Minkowski's inequality
(the supremum inequality when $p=\infty$) give
\[
 M_p(f,r)\le |f(0)|+\int_0^rM_p(f',t)\,dt
 \le |f(0)|+\|f'\|_{H(p,q,\alpha)}\int_0^r(1-t)^{-\alpha}\,dt.
\]
This proves the stated norm bound. Integrating instead from $r_J$
gives
$\|f\|_{H^p}\le M_p(f,r_J)+(1-\alpha)^{-1}
2^{-J(1-\alpha)}\|f'\|_{H(p,q,\alpha)}$.

For $0<p<1$, put
$A_j(\theta)=\sup_{r_j\le u\le r_{j+1}}|f'(ue^{i\theta})|$.
The Poisson inequality on the circle of radius $r_{j+2}$ gives
\[
 \int_0^{2\pi}A_j(\theta)^p\,\frac{d\theta}{2\pi}
       \le C M_p(f',r_{j+2})^p.
\]
Indeed, for $u\in[r_j,r_{j+1}]$, the kernels $P_{u/r_{j+2}}$
are bounded by a fixed multiple of $P_{r_{j+1}/r_{j+2}}$:
their distances from $1$ are comparable. This also holds for the
initial interval, whose ratios lie in $[0,2/3]$.
The radial increment over $[r_j,r_{j+1}]$, or over a subinterval,
is bounded by $\delta_j A_j(\theta)/2$. Subadditivity of the $p$th
power, followed by \eqref{eq:mean-growth}, therefore yields
\[
 \|f\|_{H^p}^p\le M_p(f,r_J)^p+
 C\|f'\|_{H(p,q,\alpha)}^p\sum_{j\ge J}2^{-jp(1-\alpha)}
 \le M_p(f,r_J)^p+C_{p,\alpha}2^{-Jp(1-\alpha)}
                      \|f'\|_{H(p,q,\alpha)}^p.
\]
For radii below $r_J$, the first term suffices by monotonicity of
integral means. Taking $J=0$ also proves boundedness.

For every $p$, a bounded source sequence has a locally uniformly
convergent subsequence with limit in the source, by
Lemma~\ref{lem:space-basics}. Apply the appropriate tail estimate to
the difference from its limit. For fixed $J$ the integral mean at
$r_J$ tends to zero, while the remaining term tends uniformly to zero
as $J\to\infty$. Thus the subsequence converges in $H^p$, proving
compactness.

Under the additional hypotheses, $\alpha p+1-p>0$. For the indicated
symbols, the proposed necessary quantity is
$2\pi\sup_{t<1}(1-t^2)^{-(\alpha p+1-p)}=\infty$.
\end{proof}

The bound $\alpha<1$ is sharp for compactness. Directly from the
mixed-norm definition, $\|z^k\|_{H(p,q,\alpha)}\approx(k+1)^{-\alpha}$.
Thus $(k+1)^{\alpha-1}z^k$ is bounded in $S(p,q,\alpha)$ and tends
locally to zero, but its $H^p$ norm is $(k+1)^{\alpha-1}$.
At $\alpha=1$ this excludes compactness; at $\alpha>1$ it excludes
boundedness.

The integral condition in \cite[Theorem~3.1]{Borgohain2026} does
characterize a stronger property: the image of the unit ball admits a
single pointwise majorant whose radial $L^p$ means are uniformly bounded.
The majorant is a finite-valued Borel function; it is not required to
be analytic.

\begin{proposition}\label{prop:radial-majorant}
Let $0<p<\infty$, $0<q\le\infty$ and $\alpha>0$ satisfy
$\alpha+1/p>1$. Let $\phi\in H(\mathbb D)$ and let $\varphi$ be an analytic
self-map of $\mathbb D$. The following conditions are equivalent:
\begin{enumerate}
 \renewcommand{\labelenumi}{\textup{(\arabic{enumi})}}
 \setlength{\itemsep}{2pt}
 \setlength{\parskip}{0pt}
 \item There is a nonnegative Borel function $G$ on $\mathbb D$ such that
 \[
  \sup_{0\le r<1}\int_0^{2\pi}G(re^{i\theta})^p\,d\theta<\infty,
 \]
 and $|D_{\phi,\varphi}^{0}f(z)|\le G(z)$ for every $z\in\mathbb D$ and
 every $f\in S(p,q,\alpha)$ with $\|f\|_{S(p,q,\alpha)}\le1$.
 \item The symbols satisfy
 \[
  \sup_{0\le r<1}\int_0^{2\pi}
  \frac{|\phi(re^{i\theta})|^p}
  {(1-|\varphi(re^{i\theta})|^2)^{\alpha p+1-p}}\,d\theta<\infty.
 \]
\end{enumerate}
Either condition implies that
$D_{\phi,\varphi}^{0}:S(p,q,\alpha)\to H^p$ is bounded. The converse fails in general,
even when the operator is compact.
\end{proposition}

\begin{proof}
Set $b=\alpha+1/p>1$. We first establish the point-evaluation estimate
\begin{equation}\label{eq:point-evaluation}
 \sup_{\|f\|_{S(p,q,\alpha)}\le1}|f(w)|
 \approx\frac{1}{(1-|w|^2)^{b-1}},
 \qquad w\in\mathbb D.
\end{equation}
For the upper bound, take $g\in H(p,q,\alpha)$, $z\in\mathbb D$ and
$R=(1+|z|)/2$. The Poisson estimate for the subharmonic function $|g|^p$
(see \cite[Chapter~1]{Duren1970}) and \eqref{eq:mean-growth} give
\[
 |g(z)|^p\le\frac{R+|z|}{R-|z|}M_p(g,R)^p
 \le C\frac{\|g\|_{H(p,q,\alpha)}^p}{(1-|z|)^{\alpha p+1}}.
\]
Applying this to $g=f'$ and integrating along a radius, we obtain
\[
 |f(w)| \le |f(0)|+C\|f'\|_{H(p,q,\alpha)} \int_0^{|w|}(1-t)^{-b}\,dt \le
 C\|f\|_{S(p,q,\alpha)}(1-|w|)^{1-b} \le
 C\frac{\|f\|_{S(p,q,\alpha)}}{(1-|w|^2)^{b-1}}.
\]
The last inequality uses $1-|w|^2\le2(1-|w|)$.

For the lower bound, consider
\[
 f_w(z)=\frac{1-|w|^2}{(1-\overline w z)^b},
 \qquad
 f_w'(z)=\frac{b\overline w(1-|w|^2)}{(1-\overline w z)^{b+1}}.
\]
The analytic power is chosen to equal $1$ at $z=0$. We verify directly
that $\sup_{w\in\mathbb D}\|f_w\|_{S(p,q,\alpha)}<\infty$; these are the
usual mixed-norm kernel estimates, as in
\cite[Proposition~2]{Arevalo2015} and \cite{ShieldsWilliams1971}.
By Lemma~\ref{lem:circle-area-kernels}, the circle integral
has the required bound for every exponent greater than one.
With $\gamma=p(b+1)>1$, this gives
$M_p(f_w',r)\le C(1-|w|^2)(1-r|w|)^{-\alpha-1}$.
Suppose first that $q<\infty$. Since
$1-r|w|\ge((1-r)+(1-|w|))/2$, the substitution
$u=(1-r)/(1-|w|)$ gives
\[
 \begin{aligned}
 \|f_w'\|_{H(p,q,\alpha)}^q
 &\le C(1-|w|^2)^q\int_0^1
     \frac{(1-r)^{\alpha q-1}}{(1-r|w|)^{(\alpha+1)q}}\,dr\\
 &\le C\left(\frac{1-|w|^2}{1-|w|}\right)^q
       \int_0^\infty\frac{u^{\alpha q-1}}{(1+u)^{(\alpha+1)q}}\,du
 \le C.
 \end{aligned}
\]
The last integral is finite because $\alpha q>0$ and $q>0$.
For $q=\infty$, the same integral-mean estimate yields
\[
 \begin{aligned}
 \|f_w'\|_{H(p,\infty,\alpha)}
 &\le C\sup_{0\le r<1}
       \left(\frac{1-r}{1-r|w|}\right)^\alpha
       \frac{1-|w|^2}{1-r|w|}
 \le 2C,
 \end{aligned}
\]
since $1-r\le1-r|w|$ and $1-|w|^2\le2(1-r|w|)$.
Also $|f_w(0)|=1-|w|^2\le1$, so the asserted uniform bound follows.
As $f_w(w)=(1-|w|^2)^{1-b}$, normalizing $f_w$ by a fixed upper bound
for its $S(p,q,\alpha)$ quasi-norm proves the lower estimate in
\eqref{eq:point-evaluation}.

Assume (1). At each fixed $z\in\mathbb D$, take the supremum over the
unit ball in the pointwise inequality. By \eqref{eq:point-evaluation},
\[
 G(z)\ge |\phi(z)|\sup_{\|f\|_{S(p,q,\alpha)}\le1}|f(\varphi(z))|
 \ge c\frac{|\phi(z)|}{(1-|\varphi(z)|^2)^{b-1}}.
\]
Raising this inequality to the power $p$ and integrating on each circle
proves (2), since $p(b-1)=\alpha p+1-p$.

Conversely, assume (2) and let $C$ be the upper-bound constant in
\eqref{eq:point-evaluation}. The continuous nonnegative function
$G(z)=C|\phi(z)|(1-|\varphi(z)|^2)^{1-b}$ satisfies (1).
For arbitrary $f\in S(p,q,\alpha)$, homogeneity then gives
\[
 \|D_{\phi,\varphi}^{0}f\|_{H^p}^p
 \le \frac{\|f\|_{S(p,q,\alpha)}^p}{2\pi}
       \sup_{0\le r<1}\int_0^{2\pi}G(re^{i\theta})^p\,d\theta.
\]
This proves boundedness. Finally,
Proposition~\ref{prop:compact-inclusion}, with $p=q=2$ and $\alpha=3/4$,
shows that compactness does not imply (2), and hence does not imply (1).
\end{proof}

\section{Operators into Hardy spaces}\label{sec:hardy-targets}

We first establish the sampling and synthesis estimates for the
fixed partition. They lead to the Hardy-target criteria and their
uniform tails. The finite-index antecedent is identified in
Subsection~\ref{subsec:intro-scope}.

For fixed $0<p,q\le\infty$ and $0<r<\infty$, set
\[
 P=\begin{cases}
       \infty,&p\le r,\\
       p/(p-r),&r<p<\infty,\\
       1,&p=\infty,
    \end{cases}
 \qquad
 Q=\begin{cases}
       \infty,&q\le r,\\
       q/(q-r),&r<q<\infty,\\
       1,&q=\infty.
    \end{cases}
\]
For every nonnegative sequence $d=\{d_{j,l}\}$, one has
\begin{equation}\label{eq:sequence-duality}
 \sup_{\substack{c\text{ finitely supported}\\
                  \|c\|_{\ell^{p,q}}\le1}}
       \sum_{j,l}d_{j,l}|c_{j,l}|^r
   =\|d\|_{\ell^{P,Q}}.
\end{equation}
For prescribed row norms $u_j=\|\{c_{j,l}\}_l\|_{\ell^p}$, the
supremum in row $j$ is $u_j^r\|\{d_{j,l}\}_l\|_{\ell^P}$.
Use H\"older's inequality for $r<p<\infty$, concentrate mass at one
coordinate for $p\le r$, and sum directly for $p=\infty$.
Applying the same argument to the row norms under $\|u\|_{\ell^q}\le1$
proves the identity. Finite truncations suffice when an exponent or
the right-hand side is infinite.

For reference, the sequence norm occurring in the theorems below has the
following explicit form. For a nonnegative array $d$,
\begin{equation}\label{eq:mixed-four-cases}
 \|d\|_{\ell^{P,Q}}=
 \begin{cases}
  \displaystyle\sup_{j,l}d_{j,l},&p\le r,\ q\le r,\\[1.5mm]
  \displaystyle\left(\sum_j(\sup_l d_{j,l})^Q\right)^{1/Q},
       &p\le r<q,\\[1.5mm]
  \displaystyle\sup_j\left(\sum_l d_{j,l}^P\right)^{1/P},
       &q\le r<p,\\[1.5mm]
  \displaystyle\left[\sum_j\left(\sum_l d_{j,l}^P\right)^{Q/P}
                 \right]^{1/Q},&r<\min\{p,q\}.
 \end{cases}
\end{equation}
In this display, $P=p/(p-r)$ when $r<p<\infty$, and $P=1$ when
$p=\infty$; similarly, $Q=q/(q-r)$ when $r<q<\infty$, and $Q=1$
when $q=\infty$. Thus the powers in the sums are always finite.
The first case tests one rectangle at a time. The second also sums over
radial scales, the third sums within each scale, and the fourth retains
both sums. The same formulas apply to a tail by starting the outer
index at $j=J$.

\begin{lemma}\label{lem:discrete-estimates}
Let $0<p,q\le\infty$, $\alpha>0$ and $n\ge0$ be an integer. Then
\[
 \left\|\left\{\delta_j^{\alpha+1/p+n}
          \sup_{w\in Q_{j,l}}|f^{(n)}(w)|\right\}_{j,l}\right\|_{\ell^{p,q}}
 \le C\|f\|_{H(p,q,\alpha)}.
\]
Moreover, put $\tau=\min\{1,p\}$, with $\tau=1$ when $p=\infty$, and
choose an integer $M>\alpha+1/\tau$. For every finitely supported sequence
$c=\{c_{j,l}\}$, the function
\[
 F_c(z)=\sum_{j,l}c_{j,l}
       \frac{\delta_j^{M-\alpha-1/p}}{(1-\overline{a_{j,l}}z)^M}
\]
satisfies $\|F_c\|_{H(p,q,\alpha)}\le C\|c\|_{\ell^{p,q}}$.
If $c_{j,l}=0$ for $j<J$ and $\|c\|_{\ell^{p,q}}\le1$, then, for each
$0<R<1$ and each integer $m\ge0$,
\[
 \max_{|z|\le R}|F_c^{(m)}(z)|
 \le C_{R,m}\,2^{-J(M-\alpha-1/\tau)}.
\]
All constants are independent of the support and the coefficients of $c$.
\end{lemma}

\begin{proof}
Suppose first that $p<\infty$. Let $Q_{j,l}^*$ be the Euclidean
$\delta_j/16$-neighborhood of $Q_{j,l}$. It lies in $\mathbb D$.
Cauchy's formula, followed by the submean inequality for $|f|^p$, gives
\[
 \sup_{w\in Q_{j,l}}|f^{(n)}(w)|^p
 \le C\delta_j^{-np-2}\int_{Q_{j,l}^*}|f(z)|^p\,dA(z).
\]
Indeed, one can first use Cauchy's
formula on the circle of radius $\delta_j/64$ about $w$, and then bound
$|f|^p$ on that circle by its averages over disks of radius
$\delta_j/64$, all contained in $Q_{j,l}^*$. At each fixed $j$ these
neighborhoods have bounded overlap and lie in a radial band of width at
most $C\delta_j$, inside $|z|<r_{j+2}$. Consequently,
\[
 \sum_l\left(\delta_j^{\alpha+1/p+n}
             \sup_{Q_{j,l}}|f^{(n)}|\right)^p
 \le C\delta_j^{\alpha p}M_p(f,r_{j+2})^p.
\]
If $p=\infty$, Cauchy's formula gives instead
$\delta_j^{\alpha+n}\sup_{Q_{j,l}}|f^{(n)}|
 \le C\delta_j^\alpha M_\infty(f,r_{j+2})$.
For $q<\infty$, monotonicity of integral means shows that
\[
 \sum_{j\ge0}\delta_j^{\alpha q}M_p(f,r_{j+2})^q
 \le C\alpha q\sum_{j\ge0}\int_{r_{j+2}}^{r_{j+3}}
          (1-t)^{\alpha q-1}M_p(f,t)^q\,dt
 \le C\|f\|_{H(p,q,\alpha)}^q.
\]
For $q=\infty$, the same conclusion follows directly from the definition.
This proves the first estimate.

We prove the estimate for $F_c$ next. Set
$u_j=\|\{c_{j,l}\}_l\|_{\ell^p}$ and denote the sum over one row by
$F_j$. Fix $0\le t<1$ and put $d=1-t$. The angular spacing of the points
$a_{j,l}$ is comparable to $\delta_j$. The elementary circle-kernel
estimate used in Proposition~\ref{prop:radial-majorant} gives, for
$0<p\le1$,
\[
 M_p(F_j,t)^p
 \le C u_j^p\delta_j^{p(M-\alpha)-1}(d+\delta_j)^{1-Mp}.
\]
Here $Mp>1$ by the choice of $M$.
For $1<p<\infty$, write
$k_l(\theta)=|1-\overline{a_{j,l}}te^{i\theta}|^{-M}$. Since $M>1$,
\[
 \sup_\theta\sum_l k_l(\theta)
 \le C\delta_j^{-1}(d+\delta_j)^{1-M},
 \qquad
 \int_0^{2\pi}k_l(\theta)\,\frac{d\theta}{2\pi}
 \le C(d+\delta_j)^{1-M}.
\]
These bounds follow by comparison with
$(d+\delta_j+|\theta-\arg a_{j,l}|)^{-M}$, with angular distance taken
modulo $2\pi$. Applying H\"older's inequality to the sum in $l$ and
then integrating yields
\[
 M_p(F_j,t)
 \le C u_j\delta_j^{M-\alpha-1}(d+\delta_j)^{1-M}.
\]
The same estimate for $p=\infty$ follows from the bound for
$\sum_l k_l$. Thus, in all cases,
\[
 M_p(F_j,t)
 \le C u_j\delta_j^{-\alpha}
       \left(\frac{\delta_j}{d+\delta_j}\right)^{M-1/\tau}.
\]
Using $M_p(\sum_jF_j,t)^\tau\le\sum_jM_p(F_j,t)^\tau$ and taking
$t=r_{k+1}$, we obtain
\[
 \bigl(\delta_k^\alpha M_p(F_c,r_{k+1})\bigr)^\tau \le C\sum_{j\le
 k}2^{-(k-j)\alpha\tau}u_j^\tau +C\sum_{j>k}2^{-(j-k)(M-\alpha-1/\tau)\tau}u_j^\tau.
\]
Both coefficient sequences decay geometrically. Convolution with either
is bounded on $\ell^{q/\tau}$ by Young's inequality for $q/\tau\ge1$,
by subadditivity for $q/\tau<1$, and by summability for $q=\infty$.
Finally, splitting the defining mixed-norm integral over
$[r_k,r_{k+1})$ and using monotonicity of integral means gives
\[
 \|F_c\|_{H(p,q,\alpha)}
 \le C\left\|\left\{\delta_k^\alpha
                       M_p(F_c,r_{k+1})\right\}_k\right\|_{\ell^q}
 \le C\|u\|_{\ell^q}.
\]
This argument includes the supremum definition when $q=\infty$.

To prove the last assertion, the finite-dimensional H\"older inequality
for $p>1$, and $\ell^p\subset\ell^1$ for $p\le1$, give
\[
 \sum_l|c_{j,l}|
 \le C\delta_j^{-(1/\tau-1/p)}u_j.
\]
On $|z|\le R$, the functions $(1-\overline{a_{j,l}}z)^{-M}$ and
their derivatives of order $m$ are bounded uniformly in $j,l$, with
bounds depending only on $R,m$ and $M$. Hence, if the coefficients
vanish for $j<J$,
\[
 \max_{|z|\le R}|F_c^{(m)}(z)|
 \le C_{R,m}\sum_{j\ge J}\delta_j^{M-\alpha-1/\tau}u_j
 \le C_{R,m}2^{-J(M-\alpha-1/\tau)},
\]
because $u_j\le1$. This completes the proof.
\end{proof}

\begin{lemma}\label{lem:measure-embedding}
Let $0<p,q\le\infty$, $0<r<\infty$, $\alpha>0$, and let $n\ge0$ be
an integer. For a finite positive Borel measure $\mu$ on $\mathbb D$,
put $d_{j,l}=\delta_j^{-r(n+\alpha+1/p)}\mu(Q_{j,l})$. Then
\begin{equation}\label{eq:measure-embedding}
 \sup_{\|f\|_{H(p,q,\alpha)}\le1}
       \int_{\mathbb D}|f^{(n)}(w)|^r\,d\mu(w)
 \approx\|d\|_{\ell^{P,Q}}.
\end{equation}
The constants are independent of $\mu$, and either side is allowed to
be infinite. For an integer $J\ge0$, suppose the tail
$\{d_{j,l}:j\ge J\}$ has finite positive $\ell^{P,Q}$ norm.
Then there is a finite sum $F_c$ as in
Lemma~\ref{lem:discrete-estimates}, with
$\|c\|_{\ell^{p,q}}\le1$ and $c_{j,l}=0$ for $j<J$, such that
\[
 \int_{\{|w|\ge r_J\}}|F_c^{(n)}(w)|^r\,d\mu(w)
 \ge c_0\|\{d_{j,l}:j\ge J\}\|_{\ell^{P,Q}},
\]
where $c_0>0$ depends only on the fixed parameters. The notation for the
tail means that the coordinates with $j<J$ are set equal to zero.
\end{lemma}

\begin{proof}
For the upper bound, let
$b_{j,l}=\delta_j^{n+\alpha+1/p}\sup_{Q_{j,l}}|f^{(n)}|$.
Lemma~\ref{lem:discrete-estimates} and \eqref{eq:sequence-duality} give
\[
 \int_{\mathbb D}|f^{(n)}|^r\,d\mu
 \le\sum_{j,l}d_{j,l}|b_{j,l}|^r
 \le\|d\|_{\ell^{P,Q}}\|b\|_{\ell^{p,q}}^r
 \le C\|d\|_{\ell^{P,Q}}\|f\|_{H(p,q,\alpha)}^r.
\]
The sequence inequality extends from finite sums by monotone convergence.

For the reverse bound, choose $M$ as in Lemma~\ref{lem:discrete-estimates}
and write
$f_{j,l}(z)=\delta_j^{M-\alpha-1/p}(1-\overline{a_{j,l}}z)^{-M}$.
Direct differentiation and the geometry of $Q_{j,l}$ give
\[
 |f_{j,l}^{(n)}(w)|
 =\frac{M(M+1)\cdots(M+n-1)|a_{j,l}|^n
          \delta_j^{M-\alpha-1/p}}
        {|1-\overline{a_{j,l}}w|^{M+n}}
 \ge c\delta_j^{-n-\alpha-1/p},\qquad w\in Q_{j,l},
\]
where the empty product is $1$. The lower bound $|a_{j,l}|\ge1/4$ also
covers $n>0$ on the innermost row.

For a finitely supported $c$, let $\varepsilon_{j,l}$ be independent
random signs, taking the values $1$ and $-1$ with equal probability.
Khintchine's inequality \cite{Kahane1985}, followed by restriction to
one rectangle at a time, gives
\[
 \mathbb E\int_{\mathbb D} \left|\sum_{j,l}\varepsilon_{j,l}c_{j,l}
 f_{j,l}^{(n)}(w)\right|^r\,d\mu(w) \ge c\int_{\mathbb D}
 \left(\sum_{j,l}|c_{j,l}f_{j,l}^{(n)}(w)|^2\right)^{r/2} \,d\mu(w) \ge
 c\sum_{j,l}d_{j,l}|c_{j,l}|^r.
\]
The expectation is a finite average, so this interchange causes no
integrability issue. Every signed sum has $H(p,q,\alpha)$ quasi-norm
at most $C\|c\|_{\ell^{p,q}}$ by Lemma~\ref{lem:discrete-estimates}.
Taking the supremum over $c$ in \eqref{eq:sequence-duality} proves the
lower bound in \eqref{eq:measure-embedding}.

For the last assertion, apply \eqref{eq:sequence-duality} to the tail
and choose $c$, supported there, so that
$\sum d_{j,l}|c_{j,l}|^r$ is at least half its supremum. Apply the preceding
random-sign argument with $\mu$ restricted to $\{|w|\ge r_J\}$.
At least one choice of signs gives an integral no smaller than its
average. Absorb those signs into $c$. The norm and local estimates for
this $F_c$ are exactly those in Lemma~\ref{lem:discrete-estimates}.
\end{proof}

For $\phi\in H(\mathbb D)$, an analytic self-map $\varphi$ of $\mathbb D$,
and $0\le t<1$, define the finite positive Borel measure
\[
 \mu_t(E)=\int_{\{\theta:\,\varphi(te^{i\theta})\in E\}}
                 |\phi(te^{i\theta})|^r\,\frac{d\theta}{2\pi},
 \qquad E\subset\mathbb D\text{ Borel},
\]
and the sequence
$d_{t,j,l}^{(n)}=\delta_j^{-r(n+\alpha+1/p)}\mu_t(Q_{j,l})$.

\begin{theorem}\label{thm:mixed-hardy}
Let $0<p,q\le\infty$, $0<r<\infty$, $\alpha>0$ and $n\ge0$ be an
integer. Let $\phi\in H(\mathbb D)$ and let $\varphi$ be an analytic
self-map of $\mathbb D$. Then the following conditions are equivalent:
\begin{enumerate}
 \renewcommand{\labelenumi}{\textup{(\arabic{enumi})}}
 \setlength{\itemsep}{2pt}\setlength{\parskip}{0pt}
 \item $D_{\phi,\varphi}^{n}:H(p,q,\alpha)\to H^r$ is bounded.
 \item The following condition holds in the applicable parameter range:
 \[
 \begin{cases}
  \displaystyle\sup_{t,j,l} d_{t,j,l}^{(n)}<\infty,&p\le r,\ q\le r,\\[1mm]
  \displaystyle\sup_t\left[\sum_j(\sup_l d_{t,j,l}^{(n)})^Q\right]^{1/Q}<\infty,
                                 &p\le r<q,\\[1mm]
  \displaystyle\sup_{t,j}\left(\sum_l \bigl(d_{t,j,l}^{(n)}\bigr)^P\right)^{1/P}<\infty,
                                 &q\le r<p,\\[1mm]
  \displaystyle\sup_t\left[\sum_j
           \left(\sum_l \bigl(d_{t,j,l}^{(n)}\bigr)^P\right)^{Q/P}\right]^{1/Q}<\infty,
                                 &r<\min\{p,q\}.
 \end{cases}
 \]
 Here $0\le t<1$. In the rows where they are used,
 $P=p/(p-r)$ and $Q=q/(q-r)$, with $P=1$ for $p=\infty$
 and $Q=1$ for $q=\infty$.
\end{enumerate}
When these conditions hold,
\[
 \|D_{\phi,\varphi}^{n}\|_{H(p,q,\alpha)\to H^r}^{\,r}
 \approx\sup_{0\le t<1}\|\{d_{t,j,l}^{(n)}\}_{j,l}\|_{\ell^{P,Q}},
\]
and
\[
 \|D_{\phi,\varphi}^{n}\|_{e,H(p,q,\alpha)\to H^r}^{\,r}
 \approx\lim_{J\to\infty}\sup_{0\le t<1}
       \|\{d_{t,j,l}^{(n)}:j\ge J\}\|_{\ell^{P,Q}}.
\]
In particular, the bounded operator is compact if and only if the last
limit is zero. All comparison constants depend only on
$p,q,r,\alpha,n$ and the fixed partition, not on $\phi$ or $\varphi$.
\end{theorem}

\begin{proof}
The definition of $\mu_t$ gives, for every analytic $f$,
\[
 M_r(D_{\phi,\varphi}^{n}f,t)^r
 =\int_{\mathbb D}|f^{(n)}(w)|^r\,d\mu_t(w).
\]
For finite $p,q,r$, apply \cite[Theorem~3]{PelaezRattyaSierra2021}
with $\omega(t)=\alpha q(1-t)^{\alpha q-1}$ and $\mu=\mu_t$. For either infinite
source index, use Lemma~\ref{lem:measure-embedding}. In both cases the
constants are independent of $t$; taking the two suprema proves the
boundedness assertion and the operator-norm estimate.

Assume from now on that $D_{\phi,\varphi}^{n}$ is bounded, and set
\[
 L_J=\sup_{0\le t<1}
       \|\{d_{t,j,l}^{(n)}:j\ge J\}\|_{\ell^{P,Q}},
 \qquad L=\lim_{J\to\infty}L_J.
\]
The limit exists because $L_J$ decreases and $L_0<\infty$.
Testing the operator on $z^n/n!$ shows that $\phi\in H^r$; thus
$\sup_t\mu_t(\mathbb D)=\|\phi\|_{H^r}^r<\infty$.

By Lemma~\ref{lem:space-basics}, $C_\rho$ is a compact contraction
on $H(p,q,\alpha)$, including its infinite-exponent cases. Hence
$D_{\phi,\varphi}^nC_\rho$ is compact.

For $\|f\|_{H(p,q,\alpha)}\le1$, write $h_\rho=f-C_\rho f$.
These functions have uniformly bounded mixed quasi-norms. For fixed
$J$, Lemma~\ref{lem:space-basics} gives uniform convergence of
$h_\rho^{(n)}$ to zero on $|w|\le r_J$, over the source unit ball.
Splitting the integral at $r_J$ and applying
\eqref{eq:measure-embedding} to the restricted measure on the outer part,
we obtain
\[
 \int_{\mathbb D}|h_\rho^{(n)}|^r\,d\mu_t \le \mu_t(\mathbb D) \max_{|w|\le
 r_J}|h_\rho^{(n)}(w)|^r +C L_J\|h_\rho\|_{H(p,q,\alpha)}^r.
\]
Taking the suprema over $f,t$ and then letting $\rho\uparrow1$ yields
$\|D_{\phi,\varphi}^{n}\|_e^r\le C L_J$. Letting $J\to\infty$ proves
the upper essential-norm estimate.

For the lower estimate, only $L>0$ needs consideration. Choose $t_J$ so
that the sequence norm of its tail is at least $L_J/2$. The last part of
Lemma~\ref{lem:measure-embedding} gives functions $f_J$ with
\[
 \|f_J\|_{H(p,q,\alpha)}\le C_0,
 \qquad
 \|D_{\phi,\varphi}^{n}f_J\|_{H^r}^r\ge cL_J\ge cL,
\]
where $C_0$ is independent of $J$. The local estimate in
Lemma~\ref{lem:discrete-estimates} shows that $f_J$ and all their fixed
derivatives tend locally uniformly to zero. Consequently,
$D_{\phi,\varphi}^{n}f_J\to0$ locally uniformly.

Lemma~\ref{lem:essential-test}, applied to these functions, gives
$L^{1/r}\le C\|D_{\phi,\varphi}^n\|_e$. Together with the upper bound
and closedness of compact operators, this proves the theorem.
\end{proof}

The primitive reduction transfers this criterion to $S(p,q,\alpha)$,
with exact equality of the operator and essential norms.

\begin{theorem}\label{thm:derivative-hardy}
Let $0<p,q\le\infty$, $0<r<\infty$, $\alpha>0$ and $n\ge1$ be an
integer. Let $\phi\in H(\mathbb D)$ and let $\varphi$ be an analytic
self-map of $\mathbb D$. The following conditions are equivalent:
\begin{enumerate}
 \renewcommand{\labelenumi}{\textup{(\arabic{enumi})}}
 \setlength{\itemsep}{2pt}\setlength{\parskip}{0pt}
 \item $D_{\phi,\varphi}^{n}:S(p,q,\alpha)\to H^r$ is bounded.
 \item $D_{\phi,\varphi}^{n-1}:H(p,q,\alpha)\to H^r$ is bounded.
 \item $\displaystyle\sup_{0\le t<1}
          \|\{d_{t,j,l}^{(n-1)}\}_{j,l}\|_{\ell^{P,Q}}<\infty$.
\end{enumerate}
Condition \textup{(3)} is the four-case test in
\eqref{eq:mixed-four-cases} with $d=d_t^{(n-1)}$, uniformly in $t$.
When these conditions hold, the operators in \textup{(1)} and
\textup{(2)} have equal norms and equal essential norms. In particular,
\[
 \begin{aligned}
 \|D_{\phi,\varphi}^{n}\|_{S(p,q,\alpha)\to H^r}^{\,r}
 &\approx\sup_{0\le t<1}
          \|\{d_{t,j,l}^{(n-1)}\}_{j,l}\|_{\ell^{P,Q}},\\
 \|D_{\phi,\varphi}^{n}\|_{e,S(p,q,\alpha)\to H^r}^{\,r}
 &\approx\lim_{J\to\infty}\sup_{0\le t<1}
          \|\{d_{t,j,l}^{(n-1)}:j\ge J\}\|_{\ell^{P,Q}}.
 \end{aligned}
\]
Under these boundedness conditions, the operator in \textup{(1)} is
compact if and only if the operator in \textup{(2)} is compact,
equivalently, if and only if the last limit is zero.
\end{theorem}

\begin{proof}
Apply Lemma~\ref{lem:primitive-reduction} to the mixed-norm
derivative space and the target $H^r$, with $A=D_{\phi,\varphi}^n$. Since $n\ge1$, constants are
annihilated and $AJ=D_{\phi,\varphi}^{n-1}$. The norm and compactness
assertions follow, and Theorem~\ref{thm:mixed-hardy}, at order $n-1$,
gives the stated criteria.
\end{proof}

The order reduction requires $n\ge1$. In both theorems above, the
compactness condition is uniform in $t$. For every fixed $t<1$, the
measure $\mu_t$ is supported on the compact set
$\varphi(\{z:|z|=t\})$, so its tail is eventually zero regardless of
whether the operator is bounded. Thus pointwise vanishing in $t$ cannot
replace the uniform condition.

\subsection{A single boundary measure and automatic compactness}

The radial criteria do not require boundary values. For finite Hardy
exponents, a single boundary measure gives an equivalent formulation.
We write a star for radial limits. A function in $H^r$, $0<r<\infty$,
has radial limits almost everywhere, its $H^r$ norm equals the $L^r$
norm of these limits, and a nonzero such function has nonzero limits
almost everywhere; see \cite[Chapters~1--2]{Duren1970}.
The bounded analytic map $\varphi$ also has radial limits almost
everywhere. At a point where $|\varphi^*|<1$, continuity of the
analytic derivative inside the disk identifies the limit of
$f^{(n)}(\varphi(z))$.

\begin{theorem}\label{thm:mixed-boundary}
Let $0<p,q\le\infty$, $0<r<\infty$, $\alpha>0$, and let
$n\ge0$ be an integer. Let $\phi\not\equiv0$ be analytic and let $\varphi$ be an analytic
self-map of $\mathbb D$. Then
$D_{\phi,\varphi}^n:H(p,q,\alpha)\to H^r$ is bounded if and only if
\begin{enumerate}
 \renewcommand{\labelenumi}{\textup{(\arabic{enumi})}}
 \setlength{\itemsep}{2pt}\setlength{\parskip}{0pt}
 \item $\phi\in H^r$;
 \item $|\varphi^*|<1$ almost everywhere on $\partial\mathbb D$;
 \item the finite measure
 \[
  \mu_{\partial}(E)=\int_{\{\zeta:\varphi^*(\zeta)\in E\}}
                         |\phi^*(\zeta)|^r\,dm(\zeta)
 \]
 satisfies $d_{\partial}^{(n)}\in\ell^{P,Q}$, where
 $d_{\partial,j,l}^{(n)}=
 \delta_j^{-r(n+\alpha+1/p)}\mu_{\partial}(Q_{j,l})$.
\end{enumerate}
The norm in \textup{(3)} is given by the four explicit cases in
\eqref{eq:mixed-four-cases}, with no radial supremum. When these
conditions hold,
\[
 \begin{aligned}
  \|D_{\phi,\varphi}^n\|^{r}
       &\approx\|d_{\partial}^{(n)}\|_{\ell^{P,Q}},\\
  \|D_{\phi,\varphi}^n\|_e^{r}
       &\approx\lim_{J\to\infty}
           \|\{d_{\partial,j,l}^{(n)}:j\ge J\}\|_{\ell^{P,Q}}.
 \end{aligned}
\]
In particular, compactness is equivalent to vanishing of the last
limit. If $\phi\equiv0$, the operator is zero and no condition on
$\varphi^*$ is required.
\end{theorem}

\begin{proof}
Suppose first that the operator is bounded. The test $z^n/n!$ gives
$\phi\in H^r$. Directly from the mixed-norm definition,
$\|z^k\|_{H(p,q,\alpha)}\approx(k+1)^{-\alpha}$: for finite $q$
this is the beta-integral estimate, and for $q=\infty$ it follows
by maximizing $(1-t)^\alpha t^k$. The integral mean of $z^k$ is
$t^k$, so this computation also covers $p=\infty$.
Consequently $f_k(z)=(k+1)^\alpha z^{k+n}$ is bounded in the source.
Let $E=\{\zeta:|\varphi^*(\zeta)|=1\}$. Since
$\phi\varphi^k\in H^r$, its boundary norm gives
\[
 (k+1)^{\alpha r}\left(\frac{(k+n)!}{k!}\right)^r
          \int_E|\phi^*|^r\,dm
    \le\|D_{\phi,\varphi}^nf_k\|_{H^r}^r\le C.
\]
The prefactor tends to infinity, because $\alpha+n>0$. Hence the
integral is zero. A nonzero Hardy function has nonzero boundary values
almost everywhere; thus $m(E)=0$.

For any source function $f$, condition (2) and continuity of
$f^{(n)}$ at each interior point imply
$(D_{\phi,\varphi}^nf)^*=\phi^*f^{(n)}(\varphi^*)$ almost
everywhere. The Hardy boundary-norm identity therefore gives
\begin{equation}\label{eq:boundary-norm-identity}
 \|D_{\phi,\varphi}^nf\|_{H^r}^r
       =\int_{\mathbb D}|f^{(n)}(w)|^r\,d\mu_{\partial}(w).
\end{equation}
Lemma~\ref{lem:measure-embedding} proves (3) and the two-sided
operator-norm estimate.

Conversely, assume (1)--(3). For $0<\rho<1$ and $f$ in the source,
$f_\rho=C_\rho f$ has bounded derivatives on $\mathbb D$.
Thus $\phi(f_\rho^{(n)}\circ\varphi)\in H^r$, and its boundary
values give \eqref{eq:boundary-norm-identity} with $f_\rho$ in place
of $f$. The measure-embedding lemma and the contraction property of
$C_\rho$ yield
\[
 \|D_{\phi,\varphi}^nf_\rho\|_{H^r}^r
     \le C\|d_{\partial}^{(n)}\|_{\ell^{P,Q}}
                                      \|f\|_{H(p,q,\alpha)}^r.
\]
Letting $\rho\uparrow1$ and applying the Hardy Fatou inequality proves
boundedness for the undilated functions.

For the essential norm, put
$L_J=\|\{d_{\partial,j,l}^{(n)}:j\ge J\}\|_{\ell^{P,Q}}$ and
$L=\lim_JL_J$. The operators $D_{\phi,\varphi}^nC_\rho$ are compact
by the dilation argument in Theorem~\ref{thm:mixed-hardy}.
For $\|f\|_{H(p,q,\alpha)}\le1$, split the right-hand side of
\eqref{eq:boundary-norm-identity}, applied to $f-C_\rho f$, at
$|w|=r_J$. On the inner disk the derivatives tend uniformly to zero,
and $\mu_{\partial}$ is finite. On the outer part,
Lemma~\ref{lem:measure-embedding} bounds the integral by $CL_J$.
Taking first $\rho\uparrow1$ and then $J\to\infty$ gives
$\|D_{\phi,\varphi}^n\|_e^r\le CL$.

If $L>0$, the tail assertion of that lemma supplies functions $f_J$
with uniformly bounded source quasi-norms, tending locally uniformly
to zero with their fixed derivatives, and satisfying
\[
 \|D_{\phi,\varphi}^nf_J\|_{H^r}^r\ge cL_J\ge cL.
\]
Lemma~\ref{lem:essential-test} gives the reverse essential-norm bound.
Closedness of compact operators gives the compactness criterion.
\end{proof}

\begin{corollary}\label{cor:Hardy-automatic-compactness}
Let $0<p,q\le\infty$, $0<r<\infty$, $\alpha>0$, and suppose $q>r$.
Every bounded $D_{\phi,\varphi}^n:H(p,q,\alpha)\to H^r$, $n\ge0$,
is compact. The same conclusion holds with source $S(p,q,\alpha)$
when $n\ge1$.
\end{corollary}

\begin{proof}
For a nonzero weight, Theorem~\ref{thm:mixed-boundary} applies.
When $q>r$, the outer sequence exponent is finite:
$Q=q/(q-r)$ for finite $q$, and $Q=1$ for $q=\infty$.
The radial tails of a fixed member of $\ell^{P,Q}$ therefore tend
to zero, regardless of whether $P$ is finite. The theorem gives
compactness. The zero weight is immediate, and the derivative-source
assertion follows from Theorem~\ref{thm:derivative-hardy}.
\end{proof}

The word ``fixed'' matters in this proof. Uniform boundedness of the
family $(d_t^{(n)})$ in a sequence space with finite outer exponent
does not, by itself, imply uniform decay of its tails. The passage to
$\mu_{\partial}$ is what permits the preceding conclusion.

For a geometric version in the first case of
\eqref{eq:mixed-four-cases}, put
$\Delta(w,\eta)=\{z:|(w-z)/(1-\overline wz)|<\eta\}$, where
$0<\eta<1$ is fixed. Then, when $r\ge\max\{p,q\}$, condition (3)
in Theorem~\ref{thm:mixed-boundary} is equivalent to
\[
 \sup_{w\in\mathbb D}
  \frac{\mu_{\partial}(\Delta(w,\eta))}
       {(1-|w|^2)^{r(n+\alpha+1/p)}}<\infty.
\]
Under boundedness, compactness is equivalent to this quotient tending
to zero as $|w|\uparrow1$. To verify both statements, a disk of this
fixed pseudohyperbolic radius meets a uniformly bounded number of
rectangles of comparable size. Conversely, each rectangle can be
covered by a fixed finite number of such disks, with centers at
comparable distance from the boundary. The two comparisons preserve
both the supremum and its boundary limit. The same argument applies
to each $\mu_t$, with constants independent of $t$.

\begin{example}\label{ex:Hardy-critical-q}
The condition $q>r$ in Corollary~\ref{cor:Hardy-automatic-compactness}
cannot be replaced by $q\ge r$. For $\phi(z)=1-z$ and
$\varphi(z)=(1+z)/2$, the operator
$D_{\phi,\varphi}^0:H(2,2,1/2)\to H^2$ is bounded but not compact.
\end{example}

\begin{proof}
The boundary measure is supported on the circle
$w(\theta)=(1+e^{i\theta})/2$, where
$1-|w(\theta)|^2=\sin^2(\theta/2)$ and
$|\phi(e^{i\theta})|^2=4\sin^2(\theta/2)$.
Fix $0<\eta<1$. The portion of this circle in $\Delta(w,\eta)$
has angular length at most $C_\eta(1-|w|^2)$, since the Euclidean
radius of the disk is comparable to $1-|w|^2$ and the circle has
radius $1/2$. On this portion,
$1-|w(\theta)|^2\approx1-|w|^2$. Hence
\[
 \mu_{\partial}(\Delta(w,\eta))\le C_\eta(1-|w|^2)^2.
\]
For $w=w(\theta_0)$ with $\theta_0\to0$, the interval
$|\theta-\theta_0|<c_\eta(1-|w|^2)$ lies in the preimage of the
disk when $c_\eta$ is small enough. Integrating over that interval
gives the reverse inequality. The geometric criterion above has
exponent $r(\alpha+1/p)=2$, so it proves both boundedness and
noncompactness. Theorem~\ref{thm:derivative-hardy} gives the same
verdict for $D_{\phi,\varphi}^1:S(2,2,1/2)\to H^2$.
\end{proof}

\section{Operators into derivative Hardy spaces}\label{sec:hardy-derivatives}

The target is now $S^r$ and the source is $S(p,q,\alpha)$.
The identity
\[
 (D_{\phi,\varphi}^{n}f)'
 =\phi'(f^{(n)}\circ\varphi)
   +\phi\varphi'(f^{(n+1)}\circ\varphi)
\]
does not, on its own, allow either term to be estimated from their sum.
Lemma~\ref{lem:two-kernels} separates the two terms on each fixed
rectangle. Related separation arguments appear in
\cite{BarikSenPaul2026,DuLiLiu2024}.

For the rest of this section, let $P,Q$ be the exponents defined in
Section~\ref{sec:hardy-targets}, and put $b_n=\alpha+1/p+n-1$.
For $0\le t<1$, define
\[
 v_{t,j,l} =\delta_j^{-rb_n} \int_{\{\theta:\,\varphi(te^{i\theta})\in Q_{j,l}\}}
 |\phi'(te^{i\theta})|^r\,\frac{d\theta}{2\pi} +\delta_j^{-r(b_n+1)}
 \int_{\{\theta:\,\varphi(te^{i\theta})\in Q_{j,l}\}}
 |\phi(te^{i\theta})\varphi'(te^{i\theta})|^r \,\frac{d\theta}{2\pi}.
\]
The two integrals are finite at each fixed $t<1$.

\begin{theorem}\label{thm:mixed-derivative}
Let $0<p,q\le\infty$, $0<r<\infty$, $\alpha>0$ and $n\ge1$ be an
integer. Let $\phi\in H(\mathbb D)$ and let $\varphi$ be an analytic
self-map of $\mathbb D$. The following conditions are equivalent:
\begin{enumerate}
 \renewcommand{\labelenumi}{\textup{(\arabic{enumi})}}
 \setlength{\itemsep}{2pt}\setlength{\parskip}{0pt}
 \item $D_{\phi,\varphi}^{n}:S(p,q,\alpha)\to S^r$ is bounded.
 \item The following condition holds in the applicable parameter range:
 \[
 \begin{cases}
  \displaystyle\sup_{t,j,l} v_{t,j,l}<\infty,&p\le r,\ q\le r,\\[1mm]
  \displaystyle\sup_t\left[\sum_j(\sup_l v_{t,j,l})^Q\right]^{1/Q}<\infty,
                                 &p\le r<q,\\[1mm]
  \displaystyle\sup_{t,j}\left(\sum_l v_{t,j,l}^P\right)^{1/P}<\infty,
                                 &q\le r<p,\\[1mm]
  \displaystyle\sup_t\left[\sum_j
           \left(\sum_l v_{t,j,l}^P\right)^{Q/P}\right]^{1/Q}<\infty,
                                 &r<\min\{p,q\}.
 \end{cases}
 \]
 Here $0\le t<1$. In the rows where they are used,
 $P=p/(p-r)$ and $Q=q/(q-r)$, with $P=1$ for $p=\infty$
 and $Q=1$ for $q=\infty$.
\end{enumerate}
When these conditions hold,
\[
 \|D_{\phi,\varphi}^{n}\|_{S(p,q,\alpha)\to S^r}^{\,r}
 \approx
 \frac{|\phi(0)|^r}{(1-|\varphi(0)|^2)^{rb_n}}
       +\sup_{0\le t<1}\|v_t\|_{\ell^{P,Q}},
\]
and
\[
 \|D_{\phi,\varphi}^{n}\|_{e,S(p,q,\alpha)\to S^r}^{\,r}
 \approx\lim_{J\to\infty}\sup_{0\le t<1}
             \|\{v_{t,j,l}:j\ge J\}\|_{\ell^{P,Q}}.
\]
Thus the bounded operator is compact if and only if the last limit is
zero. All comparison constants depend only on the fixed parameters and
the partition, not on the two symbols.
\end{theorem}

\begin{proof}
We first record the derivative point-evaluation estimate
\[
 \sup_{\|f\|_{S(p,q,\alpha)}\le1}|f^{(n)}(w)|
       \approx(1-|w|^2)^{-b_n},\qquad w\in\mathbb D.
\]
The upper bound follows by applying
Lemma~\ref{lem:discrete-estimates} to $f'$ at order $n-1$.
For the lower bound, take the primitive of a single kernel from that
lemma in the rectangle containing $w$. Its source quasi-norm is bounded,
and its $n$th derivative at $w$ has modulus comparable to
$\delta_j^{-b_n}\approx(1-|w|^2)^{-b_n}$.
Consequently the
norm of the scalar map $f\mapsto\phi(0)f^{(n)}(\varphi(0))$ is
comparable to $|\phi(0)|(1-|\varphi(0)|^2)^{-b_n}$.

Suppose (2) holds. At a fixed $t$, use the pull-back measures obtained
from $|\phi'|^r\,d\theta/(2\pi)$ and
$|\phi\varphi'|^r\,d\theta/(2\pi)$ on the circle of radius $t$.
Apply Lemma~\ref{lem:measure-embedding} to $f'$ with orders $n-1$ and
$n$, respectively. Each associated sequence is bounded coordinatewise
by $v_t$. The derivative formula and the elementary inequality for the
$r$th power of a sum therefore give
\[
 M_r((D_{\phi,\varphi}^{n}f)',t)^r
       \le C\|v_t\|_{\ell^{P,Q}}
                    \|f'\|_{H(p,q,\alpha)}^r.
\]
Together with the scalar estimate, this proves boundedness and the
upper operator-norm estimate.

For the converse, suppose (1) holds. Fix $A=\alpha+1/p$ and an integer
$M>\alpha+1/\min\{1,p\}$, and for finitely supported $c$ define
\[
 f_{c,u}(z)=\int_0^z\sum_{j,l}c_{j,l}g_{j,l,u}(\zeta)\,d\zeta,
                  \qquad u=0,1.
\]
Lemma~\ref{lem:discrete-estimates}, applied with $M$ and $M+1$, shows
that $\|f_{c,u}\|_{S(p,q,\alpha)}\le C\|c\|_{\ell^{p,q}}$.
Let $\varepsilon_{j,l}$ be independent random signs and write
$\varepsilon c=\{\varepsilon_{j,l}c_{j,l}\}$. For each fixed $t$,
Khintchine's inequality and \eqref{eq:two-kernels} give
\[
 \sum_{j,l}v_{t,j,l}|c_{j,l}|^r \le C\sum_{u=0}^1\mathbb E
 M_r((D_{\phi,\varphi}^{n}f_{\varepsilon c,u})',t)^r \le
 C\|D_{\phi,\varphi}^{n}\|^r\|c\|_{\ell^{p,q}}^r.
\]
To justify the first inequality, apply Khintchine pointwise, retain the
$(j,l)$th summand only on
$\{\theta:\varphi(te^{i\theta})\in Q_{j,l}\}$, and add over the
disjoint sets. Then apply \eqref{eq:two-kernels} with
$x=\phi'(te^{i\theta})$ and
$y=\phi(te^{i\theta})\varphi'(te^{i\theta})$.
The kernels are fixed before integration. Taking the supremum in
\eqref{eq:sequence-duality} and then over $t$ proves (2). The scalar
estimate gives the remaining part of the lower norm bound.

Assume henceforth that the operator is bounded, and put
\[
 L_J=\sup_{0\le t<1}\|\{v_{t,j,l}:j\ge J\}\|_{\ell^{P,Q}},
                 \qquad L=\lim_{J\to\infty}L_J.
\]
The limit exists and is finite. Testing on $z^n/n!$ and
$z^{n+1}/(n+1)!$ gives $\phi,\phi\varphi\in S^r$. Hence
$\phi'\in H^r$ and
$\phi\varphi'=(\phi\varphi)'-\phi'\varphi\in H^r$.
In particular, the total masses of both pull-back measures are bounded
uniformly in $t$.

Lemma~\ref{lem:space-basics} makes
$D_{\phi,\varphi}^{n}C_\rho$ compact for $\rho<1$.
For $f$ in the unit ball,
set $h_\rho=f-C_\rho f$. Its $S(p,q,\alpha)$ quasi-norm is uniformly
bounded, and its fixed derivatives tend uniformly to zero on every
compact disk, uniformly over that unit ball. Split the two derivative
integrals into $|\varphi(te^{i\theta})|<r_J$ and its complement.
On the first set, their integrals tend to zero as $\rho\uparrow1$,
uniformly in $t$, by the total-mass bounds. The scalar value at zero
also tends to zero. On the second set,
Lemma~\ref{lem:measure-embedding}, applied to the restricted measures,
bounds the sum of the integrals by $CL_J$. Therefore
\[
 \|D_{\phi,\varphi}^{n}\|_e^r
 \le\limsup_{\rho\uparrow1}
       \|D_{\phi,\varphi}^{n}(I-C_\rho)\|^r
 \le CL_J.
\]
Letting $J\to\infty$ proves the upper essential-norm estimate.

For the lower bound, suppose $L>0$. Choose $t_J$ such that its tail
norm is at least $L_J/2$. By \eqref{eq:sequence-duality}, choose
finitely supported $c$ with $\|c\|_{\ell^{p,q}}\le1$, supported in
$j\ge J$, for which
$\sum v_{t_J,j,l}|c_{j,l}|^r\ge L_J/4$.
The same random-sign estimate as above, now integrated only where
$|\varphi(t_Je^{i\theta})|\ge r_J$, selects one value of $u$ and one
choice of signs giving a function $f_J$ such that
\[
 \|f_J\|_{S(p,q,\alpha)}\le C_0,
       \qquad \|D_{\phi,\varphi}^{n}f_J\|_{S^r}^r\ge cL_J\ge cL.
\]
The last estimate in Lemma~\ref{lem:discrete-estimates} shows that
$f_J$ and its fixed derivatives tend locally uniformly to zero.
The same is true of $D_{\phi,\varphi}^{n}f_J$.

Lemma~\ref{lem:essential-test}, with target $S^r$, now gives
$L^{1/r}\le C\|D_{\phi,\varphi}^n\|_e$. The upper bound and closedness
of compact operators complete the proof.
\end{proof}

\begin{corollary}\label{cor:Hardy-derivative-automatic}
In the parameter range of Theorem~\ref{thm:mixed-derivative}, if
$q>r$, boundedness of
$D_{\phi,\varphi}^n:S(p,q,\alpha)\to S^r$ is equivalent to compactness.
\end{corollary}

\begin{proof}
The proof of Theorem~\ref{thm:mixed-derivative} and
Lemma~\ref{lem:measure-embedding} show that both
$D_{\phi',\varphi}^{n-1}$ and $D_{\phi\varphi',\varphi}^{n}$ are
bounded from $H(p,q,\alpha)$ to $H^r$. They are compact by
Corollary~\ref{cor:Hardy-automatic-compactness}.
The derivative of $D_{\phi,\varphi}^nf$ is their sum applied to $f'$.
Differentiation maps the source boundedly into $H(p,q,\alpha)$,
and taking a primitive maps $H^r$ boundedly into $S^r$.
The remaining scalar coordinate $\phi(0)f^{(n)}(\varphi(0))$ has
rank at most one. Thus the entire operator is compact.
\end{proof}

\section{Operators between \texorpdfstring{$F(p,q,s)$}{F(p,q,s)} spaces}
\label{sec:F-spaces}

We use Zhao's square formulation \eqref{eq:F-squares} to prove
sampling and synthesis estimates for every admissible triple.
After evaluating the associated sequence condition, we apply these
estimates to boundedness, essential norms and little-space ranges.

\begin{lemma}\label{lem:positive-Carleson-kernel}
Let $s\ge0$, $\beta>-1$, $\beta+2>s$ and $\lambda>\beta+2$.
If $\|\mu\|_s<\infty$, define a positive measure $\nu$ by
\[
 d\nu(z)=(1-|z|^2)^\beta
       \left(\int_{\mathbb D}
          \frac{(1-|w|^2)^{\lambda-\beta-2}}
               {|1-\overline wz|^\lambda}\,d\mu(w)\right)dA(z).
\]
Then $\|\nu\|_s\le C\|\mu\|_s$.
\end{lemma}

\begin{proof}
For $s=0$, integrate over $\mathbb D$ and apply
\eqref{eq:weighted-area-kernel}; the result is
$\nu(\mathbb D)\le C\mu(\mathbb D)$. Suppose $s>0$ and fix an arc
$I$ of length $h<1/4$. Let $I_k$ be its concentric dilate of length
$\min\{2^k h,1\}$, and stop when this is the full circle.
For $w\in S(I_2)$, use \eqref{eq:weighted-area-kernel} to bound the
inner $z$-integral over $S(I)$ by its integral over $\mathbb D$.
After multiplication by $(1-|w|^2)^{\lambda-\beta-2}$, this contribution
is at most $C\mu(S(I_2))\le C\|\mu\|_s h^s$.
For $w\in S(I_{k+1})\setminus S(I_k)$, $k\ge2$, and $z\in S(I)$,
the geometry of the squares gives
$|1-\overline wz|\ge c2^k h$ and $1-|w|^2\le C2^k h$.
Since $\lambda-\beta-2>0$ and
$\int_{S(I)}(1-|z|^2)^\beta\,dA(z)\le C h^{\beta+2}$,
the contribution from this difference is at most
\[
 C h^{\beta+2}(2^k h)^{-\beta-2}
                  \mu(S(I_{k+1}))
 \le C\|\mu\|_s h^s\,2^{-k(\beta+2-s)}.
\]
The series converges because $\beta+2>s$. For $h\ge1/4$, the global
bound from \eqref{eq:weighted-area-kernel} suffices. This proves the
lemma. All interchanges are between nonnegative integrands.
\end{proof}

Using the same points $a_{j,l}$ and numbers $\delta_j$ as before, define
\[
 \|c\|_{\mathcal E_{p,s}}^p
 =\sup_I\frac{1}{|I|^s}
       \sum_{a_{j,l}\in S(I)}|c_{j,l}|^p\delta_j^s,
                  \qquad 0<p<\infty,\quad s\ge0.
\]
In particular, $\mathcal E_{p,0}=\ell^p$ on the set of all pairs
$(j,l)$. Choosing an arc of length comparable to $\delta_j$ containing
$a_{j,l}$ in its square shows that
$|c_{j,l}|\le C\|c\|_{\mathcal E_{p,s}}$. Conversely, a sequence with
one coefficient of modulus $1$ has $\mathcal E_{p,s}$ quasi-norm
bounded by a fixed constant. The quasi-norm is unchanged by changing
the phases of the coefficients and decreases under coordinatewise
restrictions.

The higher-derivative integral characterizations are due to
R\"atty\"a \cite[Theorems~3.2--3.3]{Rattya2003}; see also
\cite[Theorem~3.13]{Zhao2021}. The test primitives below are related
to the kernel families in \cite[Lemma~2.6]{PauZhao2014} and
\cite[Theorem~4.2]{Zhao2021}.
We prove the lattice estimates explicitly, since the argument later
uses finite sums with arbitrary phases and constants independent of
the support, not just the bounds for individual kernels.

\begin{lemma}\label{lem:F-sampling-synthesis}
Let $0<p<\infty$, $-2<q<\infty$, $0\le s<\infty$, $q+s>-1$, and put $A=(q+2)/p$.
For each integer $m\ge1$ and $f\in F(p,q,s)$,
\[
 \left\|\left\{\delta_j^{A+m-1}
          \sup_{w\in Q_{j,l}}|f^{(m)}(w)|\right\}_{j,l}
       \right\|_{\mathcal E_{p,s}}
       \le C\|f\|_{F(p,q,s)}.
\]
Choose a positive integer $M$ such that $M>A+1$ and
$Mp>q+s+2+2\max\{p-1,0\}$. For every finitely supported $c$, put
\[
 G_c(z)=\sum_{j,l}c_{j,l}
            \frac{\delta_j^{M-A}}{(1-\overline{a_{j,l}}z)^M},
            \qquad f_c(z)=\int_0^zG_c(\zeta)\,d\zeta.
\]
Then $\|f_c\|_{F(p,q,s)}\le C\|c\|_{\mathcal E_{p,s}}$.
If the coefficients vanish for $j<J$ and
$\|c\|_{\mathcal E_{p,s}}\le1$, then, for each fixed $R<1$ and integer
$k\ge0$,
\[
 \max_{|z|\le R}|G_c^{(k)}(z)|
                   \le C_{R,k}2^{-J(M-A-1)}.
\]
Finally, for every integer $m\ge1$,
\begin{equation}\label{eq:F-derivative-evaluation}
 \sup_{\|f\|_{F(p,q,s)}\le1}|f^{(m)}(w)|
       \approx(1-|w|^2)^{-A-m+1},\qquad w\in\mathbb D.
\end{equation}
\end{lemma}

\begin{proof}
Write $b_{j,l}=\delta_j^{A+m-1}\sup_{Q_{j,l}}|f^{(m)}|$.
Apply the local Cauchy and submean estimate from
Lemma~\ref{lem:discrete-estimates} to $f'$ with order $m-1$.
On $Q_{j,l}^*$, $1-|z|^2\approx\delta_j$, and hence
\[
 |b_{j,l}|^p\delta_j^s
       \le C\int_{Q_{j,l}^*}|f'(z)|^p
                                  (1-|z|^2)^{q+s}\,dA(z).
\]
The sets $Q_{j,l}^*$ have bounded overlap, also when all rows are
included: at a given point their row sizes are comparable to its
distance from the boundary. If $a_{j,l}\in S(I)$, then $Q_{j,l}^*$ is
contained in the square of a concentric arc whose length is at most
$C|I|$, or in $\mathbb D$ when this exceeds $1$.
Summing the displayed inequality and using \eqref{eq:F-squares}
proves the sampling estimate.

For the synthesis estimate, let
$\mu_c=\sum_{j,l}|c_{j,l}|^p\delta_j^s\delta_{a_{j,l}}$, where
$\delta_a$ denotes unit point mass at $a$. Then
$\|\mu_c\|_s=\|c\|_{\mathcal E_{p,s}}^p$.
Suppose first that $0<p\le1$. Subadditivity gives
\[
 |G_c(z)|^p(1-|z|^2)^{q+s}
 \le\sum_{j,l}|c_{j,l}|^p\delta_j^s
       \frac{\delta_j^{Mp-q-s-2}(1-|z|^2)^{q+s}}
            {|1-\overline{a_{j,l}}z|^{Mp}}.
\]
Since $1-|a_{j,l}|^2\approx\delta_j$, the measure on the right is
bounded by a constant times the measure in
Lemma~\ref{lem:positive-Carleson-kernel}, with
$\beta=q+s$, $\lambda=Mp$ and $\mu=\mu_c$.
Its hypotheses hold: $\beta>-1$, $\beta+2-s=q+2>0$ and
$Mp>\beta+2$. The lemma and \eqref{eq:F-squares}, with $f_c(0)=0$,
prove the required estimate.

Now let $p>1$, and choose
$0<\varepsilon<\min\{q+s+1,q+2\}$.
The geometry of the partition and
\eqref{eq:weighted-area-kernel} with $\beta=0$ imply
\[
 \sum_{j,l}\frac{\delta_j^2} {|1-\overline{a_{j,l}}z|^{2+\varepsilon/(p-1)}} \le
 C\int_{\mathbb D} \frac{dA(w)}{|1-\overline wz|^{2+\varepsilon/(p-1)}} \le
 C(1-|z|^2)^{-\varepsilon/(p-1)}.
\]
Indeed, $A(Q_{j,l})\approx\delta_j^2$ and
$|1-\overline{a_{j,l}}z|\approx|1-\overline wz|$ for
$w\in Q_{j,l}$ and $z\in\mathbb D$.
The latter comparison follows from $|w-a_{j,l}|\le C\delta_j$ and
both lower bounds $|1-\overline wz|\ge c\delta_j$ and
$|1-\overline{a_{j,l}}z|\ge c\delta_j$.
Set $\lambda=Mp-2(p-1)-\varepsilon$.
H\"older's inequality for the finite sum defining $G_c$ gives
\[
 \begin{aligned}
 |G_c(z)|^p
 &\le\left(\sum_{j,l}
       \frac{|c_{j,l}|^p\delta_j^{Mp-q-2p}}
            {|1-\overline{a_{j,l}}z|^\lambda}\right)
       \left(\sum_{j,l}
       \frac{\delta_j^2}
            {|1-\overline{a_{j,l}}z|^{2+\varepsilon/(p-1)}}
       \right)^{p-1}\\
 &\le C(1-|z|^2)^{-\varepsilon}
       \sum_{j,l}\frac{|c_{j,l}|^p\delta_j^{Mp-q-2p}}
                         {|1-\overline{a_{j,l}}z|^\lambda}.
 \end{aligned}
\]
Put $\beta=q+s-\varepsilon$. Multiplication by
$(1-|z|^2)^{q+s}$ rewrites the last bound as
\[
 C\sum_{j,l}|c_{j,l}|^p\delta_j^s
        \frac{\delta_j^{\lambda-\beta-2}(1-|z|^2)^\beta}
             {|1-\overline{a_{j,l}}z|^\lambda}.
\]
Our choice of $\varepsilon$ ensures $\beta>-1$ and $\beta+2>s$,
and the choice of $M$ ensures $\lambda>\beta+2$.
Lemma~\ref{lem:positive-Carleson-kernel} therefore applies once again.
This proves the synthesis estimate for $p>1$ as well.

For the local tail estimate, the coordinate bound for
$\mathcal E_{p,s}$ and the $2^{j+3}$ entries in row $j$ give
$\sum_l|c_{j,l}|\le C\delta_j^{-1}$.
On $|z|\le R$, all fixed derivatives of
$(1-\overline{a_{j,l}}z)^{-M}$ are uniformly bounded. Thus
\[
 \max_{|z|\le R}|G_c^{(k)}(z)|
       \le C_{R,k}\sum_{j\ge J}\delta_j^{M-A-1}
       \le C_{R,k}2^{-J(M-A-1)}.
\]
For the upper bound in \eqref{eq:F-derivative-evaluation}, combine
sampling with the coordinate bound for $\mathcal E_{p,s}$.
For the lower bound, choose the rectangle containing $w$ and take
$f_c$ with only that coefficient nonzero and equal to $1$.
Its quasi-norm is uniformly bounded by the synthesis estimate.
Direct differentiation, using $|a_{j,l}|\ge1/4$ and
$|1-\overline{a_{j,l}}w|\approx\delta_j$, gives
$|f_c^{(m)}(w)|\ge c\delta_j^{-A-m+1}$.
This proves \eqref{eq:F-derivative-evaluation}.
\end{proof}

For $0<p,r<\infty$, $s\ge0$ and a nonnegative sequence $v$, put
\[
 \mathfrak D_{p,s;r}(v)
 =\sup_{\substack{c\text{ finitely supported}\\
                   \|c\|_{\mathcal E_{p,s}}\le1}}
       \left(\sum_{j,l}v_{j,l}|c_{j,l}|^r\right)^{1/r}.
\]
Infinite values are allowed, with zero coefficients contributing zero.
The defining inequality extends to all $c\in\mathcal E_{p,s}$ by
finite restrictions and monotone convergence. In the case $s=0$,
ordinary sequence duality gives
\[
 \mathfrak D_{p,0;r}(v)^r
 =\begin{cases}
      \|v\|_{\ell^\infty},&p\le r,\\
      \|v\|_{\ell^{p/(p-r)}},&r<p.
   \end{cases}
\]

\subsection{Explicit evaluation of the sequence condition}

For $0<s\le1$, the nested square constraints lead to two formulas:
a dyadic covering integral when $r\ge p$, and an integral of a
bottom-up power recursion when $r<p$. The distinction remains necessary
at $s=1$. Dyadic capacities and fractional Carleson
conditions also occur in the tent-space setting of \cite{DafniXiao2004};
the sequence calculations needed here are proved below.

Identify $\partial\mathbb D$ with $[0,1)$ and let $\mathcal T$ be
its binary tree of half-open dyadic arcs, including the full circle.
For $I\in\mathcal T$, write $I_-,I_+$ for its two children. The
angular interval of $Q_{j,l}$ is the arc $I_{j,l}$ of length
$2^{-j-3}=\delta_j/8$. Attach $v_{I_{j,l}}=v_{j,l}$ to this node,
and put $v_I=0$ at depths $0,1,2$. When $0<s\le1$, set
$b_I=|I|^s$ and $\theta=r/p$.

For $N\ge0$, retain only rows $0\le j\le N$, so that the finite
tree ends at depth $N+3$. Define $R_I^{(N)}(t)$ for $t>0$ upwards
from the leaves, setting a missing child's value to zero. There are
two rules:
\begin{equation}\label{eq:tree-superlinear-rule}
 R_I^{(N)}(t)=
 \begin{cases}
 b_I,&v_I/b_I>t,\\
 \min\{b_I,R_{I_-}^{(N)}(t)+R_{I_+}^{(N)}(t)\},&v_I/b_I\le t,
 \end{cases}
 \qquad r\ge p,
\end{equation}
and
\begin{equation}\label{eq:tree-sublinear-rule}
 R_I^{(N)}(t)=\min\left\{b_I,
       \left(\frac{\theta v_I}{t b_I^\theta}\right)^{1/(1-\theta)}
       +R_{I_-}^{(N)}(t)+R_{I_+}^{(N)}(t)\right\},\qquad r<p.
\end{equation}
A zero $v_I$ contributes zero in the second formula. Define
\begin{equation}\label{eq:explicit-sequence-functional}
 \mathcal C_{p,s;r}(v)=
 \begin{cases}
  \displaystyle\sup_{j,l}v_{j,l},&s=0,\ p\le r,\\[1mm]
  \displaystyle\left(\sum_{j,l}v_{j,l}^{p/(p-r)}\right)^{(p-r)/p},
                                         &s=0,\ r<p,\\[1mm]
  \displaystyle\lim_{N\to\infty}\int_0^\infty
           R_{\partial\mathbb D}^{(N)}(t)\,dt,&0<s\le1,\\[1mm]
  \displaystyle\sum_{j,l}v_{j,l},&s>1.
 \end{cases}
\end{equation}
The limit is increasing. Infinite values are allowed; if an entry of
$v$ is infinite, set $\mathcal C_{p,s;r}(v)=\infty$. At finite depth
both recursions use only addition, minimum and powers of the given
entries. In \eqref{eq:tree-sublinear-rule} the functions are piecewise
of the form $A+Bt^{-1/(1-\theta)}$, with finitely many breakpoints.
Thus the remaining one-dimensional integral is also explicit.
For a tail $j\ge J$, first set the earlier entries to zero and use
exactly the same rule; the intervening ancestors are not removed.
The comparison below is uniform in $J$, since it holds for every
nonnegative array with constants independent of its support.

\begin{proposition}\label{prop:explicit-sequence}
For every $0<p,r<\infty$, $s\ge0$ and nonnegative array $v$,
\begin{equation}\label{eq:sequence-explicit-equivalence}
       \mathfrak D_{p,s;r}(v)^r\approx\mathcal C_{p,s;r}(v),
\end{equation}
with constants independent of $v$ and of finite truncations.
The equality in the cases $s=0$ is exact.
For $0<s\le1$ and $r\ge p$, the finite-depth integrand in
\eqref{eq:explicit-sequence-functional} is the least dyadic
$s$-covering cost of the arcs satisfying $v_I/b_I>t$.
In particular, when $s=1$ and $r\ge p$,
\begin{equation}\label{eq:critical-shadow-maximal}
 \mathcal C_{p,1;r}(v)
     =\int_{\partial\mathbb D}
        \sup_{\substack{j,l\ge0\,;\ I_{j,l}\ni\zeta}}
                    \frac{v_{j,l}}{|I_{j,l}|}\,dm(\zeta),
\end{equation}
where $l$ always has its prescribed finite range in row $j$.
\end{proposition}

\begin{proof}
The cases $s=0$ follow from the sequence duality already proved.
For $s>1$, the coordinate estimate gives
$\|c\|_{\ell^\infty}\le C\|c\|_{\mathcal E_{p,s}}$.
Conversely, each square $S(I)$ contains at most $C|I|/\delta_j$
marked points from a row with $\delta_j\le C|I|$, so
\[
 \sum_{a_{j,l}\in S(I)}\delta_j^s
    \le C|I|\sum_{\delta_j\le C|I|}\delta_j^{s-1}
    \le C'|I|^s.
\]
Thus $\mathcal E_{p,s}$ has quasi-norm equivalent to
$\ell^\infty$, and its multiplier expression is comparable to
$\sum v_{j,l}$.

We turn to $0<s\le1$. Attach $c_{j,l}$ to $I_{j,l}$, with zero
coefficients on the first three levels. The geometric comparison
\begin{equation}\label{eq:tree-Carleson-comparison}
 \|c\|_{\mathcal E_{p,s}}^p\approx
   \sup_{I\in\mathcal T}\frac{1}{|I|^s}
         \sum_{J\subseteq I}|c_J|^p|J|^s
\end{equation}
has constants independent of the support. Indeed, all marked points
attached to descendants of $I$ lie in a square over a concentric arc
of length at most $C|I|$, or in the full disk. This gives one
inequality, since $\delta_j=8|I_{j,l}|$.
Conversely, if $a_{j,l}\in S(K)$ for an arbitrary arc $K$ of length
$h$, then $1-|a_{j,l}|=3\delta_j/4\le h$, so
$|I_{j,l}|\le h/6$. Choose a dyadic length between $h$ and $2h$
(or the full circle when $h>1/2$). The ancestor of $I_{j,l}$ at that
level intersects $K$, and at most three arcs at that level can do so.
Their subtree bounds control the square over $K$. This proves the
other inequality.

It suffices, therefore, to work on a finite dyadic tree with the
constraints
\begin{equation}\label{eq:tree-mass-constraints}
 x_I\ge0,\qquad \sum_{J\subseteq I}x_J\le b_I\quad(I\in\mathcal T),
       \qquad x_I=|c_I|^p b_I.
\end{equation}
Coordinates on the first three levels, and coordinates with $v_I=0$,
may be set to zero. The quantity to be evaluated is the maximum of
$\sum_I v_I(x_I/b_I)^\theta$ under these constraints.
The maximum exists because the feasible set is a finite compact
polytope. All constants obtained below are independent of its depth.

Suppose first that $\theta\ge1$, and put $h_I=v_I/b_I$.
For a collection $\mathcal A$ of nodes, let its covering cost be the
minimum of $\sum_{K\in\mathcal B}b_K$ over collections of dyadic
ancestors such that every $I\in\mathcal A$ is contained in some
$K\in\mathcal B$. Redundant nested covering arcs can be removed.
For the set $\{I:h_I>t\}$, this cost is
$R_{\partial\mathbb D}^{(N)}(t)$ in
\eqref{eq:tree-superlinear-rule}. To see this, at a node already in
the set its whole arc must be covered, at cost $b_I$; otherwise one
chooses between that cost and the sum of the optimal child costs.
Because $0<s\le1$, covering a whole arc by its proper dyadic
subarcs cannot cost less than $|I|^s$: their lengths sum to at
least $|I|$, and $\sum_j u_j^s\ge(\sum_j u_j)^s$.
Thus this is also the usual dyadic $s$-content of the finite union
of the selected arcs.

Every feasible $x$ assigns mass at most the covering cost to any
selected collection. Also $x_I\le b_I$ and $\theta\ge1$.
Layer integration consequently gives
\[
 \sum_Iv_I(x_I/b_I)^\theta
    \le\sum_Ih_Ix_I
    =\int_0^\infty\sum_{h_I>t}x_I\,dt
    \le\int_0^\infty R_{\partial\mathbb D}^{(N)}(t)\,dt.
\]
For the reverse comparison, order the nodes with $h_I>0$ by
nonincreasing $h_I$. Starting with the empty collection, accept a
node $I$ exactly when assigning it mass $b_I$ would leave the total
accepted mass in every node $K\supseteq I$, including $K=I$,
at most $2b_K$.
Let $\mathcal A_t$ be the accepted nodes with $h_I>t$.
If a node with $h_I>t$ was rejected, at that time one of its ancestors
$K$ had accepted mass greater than $2b_K-b_I\ge b_K$.
This remains true for $\mathcal A_t$. Choose the inclusion-maximal
such witnessing ancestors; they are disjoint. They cover all rejected
nodes above level $t$ and cost no more than the accepted mass inside
them. Cover the remaining accepted nodes by their own arcs. This
proves
\[
 R_{\partial\mathbb D}^{(N)}(t)
           \le\sum_{I\in\mathcal A_t}b_I.
\]
The final accepted collection satisfies all constraints with the
factor $2$. Assigning $x_I=b_I/2$ on that collection and zero
elsewhere is feasible. Integrating the preceding inequality gives
\[
 \max_x\sum_Iv_I(x_I/b_I)^\theta
    \ge2^{-\theta}\sum_{I\text{ accepted}}v_I
    \ge2^{-\theta}\int_0^\infty
                         R_{\partial\mathbb D}^{(N)}(t)\,dt.
\]
This proves the comparison for $r\ge p$.

Now suppose $0<\theta<1$, and put $w_I=v_I/b_I^\theta$.
We prove an exact finite-tree identity. For a subtree rooted at $I$,
let $\Psi_I(t)$ be the maximum of
$\sum_{J\subseteq I}w_Jx_J^\theta-t\sum_{J\subseteq I}x_J$
under its constraints \eqref{eq:tree-mass-constraints}.
Zero-weight variables are discarded. Strict concavity makes the
maximizer unique on the remaining coordinates, and write $D_I(t)$
for its total mass.
Without the constraint at $I$, the local variable has optimal mass
$(\theta w_I/t)^{1/(1-\theta)}$, and the two subtrees contribute
$D_{I_-}(t)$ and $D_{I_+}(t)$. If the sum is at most $b_I$, the allocation is feasible. Otherwise
choose a common penalty $u>t$ making its total mass equal to $b_I$.
The child demands are continuous, nonincreasing and tend to zero by
the induction hypothesis; the local demand has the same properties.
Thus such a $u$ exists.
The resulting allocation maximizes the constrained objective at the
original price: it maximizes the objective at price $u$, and among
allocations of mass at most $b_I$ the added term $(u-t)$ times the
mass is largest at mass $b_I$. Hence
\[
 D_I(t)=\min\left\{b_I,
        (\theta w_I/t)^{1/(1-\theta)}
                           +D_{I_-}(t)+D_{I_+}(t)\right\}.
\]
If the unconstrained sum never exceeds $b_I$, no adjustment is
needed. This also covers subtrees with zero weights.
Induction proves that $D_I=R_I^{(N)}$ from
\eqref{eq:tree-sublinear-rule}, and that this function is continuous
for $t>0$.

Comparing optimal allocations at $t$ and $t+h$, for $h>0$, yields
\[
 -hD_I(t)\le\Psi_I(t+h)-\Psi_I(t)\le-hD_I(t+h).
\]
Thus $\Psi_I'(t)=-D_I(t)$. The feasible set is compact,
$\Psi_I(t)\to\Psi_I(0)$ as $t\downarrow0$, and
$\Psi_I(t)\to0$ as $t\to\infty$.
For the last limit one may omit all constraints and maximize the
finitely many expressions $w_Jx^\theta-tx$ separately; each such
maximum tends to zero. Integrating the derivative gives
\[
 \max_x\sum_Iw_Ix_I^\theta
      =\Psi_{\partial\mathbb D}(0)
      =\int_0^\infty R_{\partial\mathbb D}^{(N)}(t)\,dt.
\]
This proves the required finite-tree identity for $r<p$.

Increasing the depth or an entry of $v$ increases both recursions.
Every finite coefficient support is contained in a depth truncation,
so passage to increasing depths and
\eqref{eq:tree-Carleson-comparison} prove
\eqref{eq:sequence-explicit-equivalence} for infinite arrays as well.
For $s=1$, the least covering cost of a finite union of dyadic arcs
is its normalized length. Layer integration then gives
\eqref{eq:critical-shadow-maximal}, first at finite depth and then
by monotone convergence. This completes the proof.
\end{proof}

The calculation makes the distinction between $r\ge p$ and $r<p$
visible. For example, put $v_{j,l}=|I_{j,l}|^s$ at every node, with
$0<s\le1$. If $r\ge p$, the covering integrand is $1$ for
$0<t<1$ and zero for $t\ge1$, so $\mathcal C_{p,s;r}(v)=1$;
the same is true after removing any finite number of rows.
If $r<p$ and $s<1$, testing one full row in
\eqref{eq:tree-mass-constraints}, with equal masses summing to $1$,
gives values growing as $2^{(j+3)(1-s)(1-r/p)}$.
If $s=1$, assign mass $|I|/K$ at every node on $K$ consecutive full
levels; the value is $K^{1-r/p}$. Thus the expression is infinite
when $r<p$. This example concerns the sequence criterion itself;
it does not assume that an arbitrary array is realized by analytic
symbols.

Rotating the angular grid, or replacing it by a fixed dyadic angular
refinement, does not change the full or tail criteria except for
comparison constants. A common refinement has a bounded number of
pieces in each old rectangle, all at the same radial scale.
The finite-sum inequalities for positive powers compare the mixed
sequence tests. For $\mathcal E_{p,s}$, duplicating a coefficient
onto these pieces, or replacing their coefficients by their maximum,
changes the square bound by at most a fixed factor, using a fixed
enlargement of the arc. Its multiplier test, and therefore
$\mathcal C$, has the same comparison. Row restrictions preserve
these arguments, so the constants are independent of the tail index.

Let the two parameter triples satisfy
\[
 0<p_i<\infty,\qquad -2<q_i<\infty,\qquad 0\le s_i<\infty,\qquad q_i+s_i>-1,
                   \qquad i=1,2.
\]
For $n\ge1$, set $b_n=(q_1+2)/p_1+n-1$ and
\[
 d\lambda_a(z)=(1-|z|^2)^{q_2}
                   (1-|\sigma_a(z)|^2)^{s_2}\,dA(z).
\]
For Borel $E\subset\mathbb D$, define
\[
 \nu_{a,0}(E)=\int_{\varphi^{-1}(E)}|\phi'(z)|^{p_2}\,d\lambda_a(z),
 \qquad
 \nu_{a,1}(E)=\int_{\varphi^{-1}(E)}
               |\phi(z)\varphi'(z)|^{p_2}\,d\lambda_a(z),
\]
and put
\[
 v_{a,j,l}
 =\frac{\nu_{a,0}(Q_{j,l})}{\delta_j^{p_2b_n}}
   +\frac{\nu_{a,1}(Q_{j,l})}{\delta_j^{p_2(b_n+1)}}.
\]
These nonnegative integrals may be infinite. No finiteness or Carleson
condition on either pull-back measure is assumed in advance.

\begin{theorem}\label{thm:F-boundedness}
Let $0<p_i<\infty$, $q_i>-2$, $0\le s_i<\infty$ and $q_i+s_i>-1$
for $i=1,2$. Let $n\ge1$ be an integer, let
$\phi\in H(\mathbb D)$ and let $\varphi$ be an analytic self-map of
$\mathbb D$. The following conditions are equivalent:
\begin{enumerate}
 \renewcommand{\labelenumi}{\textup{(\arabic{enumi})}}
 \setlength{\itemsep}{2pt}\setlength{\parskip}{0pt}
 \item $D_{\phi,\varphi}^{n}:F(p_1,q_1,s_1)\to F(p_2,q_2,s_2)$ is
       bounded.
 \item $\displaystyle\sup_{a\in\mathbb D}
                     \mathcal C_{p_1,s_1;p_2}(v_a)<\infty$.
\end{enumerate}
The quantity in \textup{(2)} is given explicitly by
\eqref{eq:explicit-sequence-functional}: its middle range uses the
covering rule for $p_2\ge p_1$ and the power rule for $p_2<p_1$.
When these conditions hold,
\[
 \|D_{\phi,\varphi}^{n}\|_{F(p_1,q_1,s_1)\to F(p_2,q_2,s_2)}
 \approx
 \frac{|\phi(0)|}{(1-|\varphi(0)|^2)^{b_n}}
       +\left(\sup_{a\in\mathbb D}\mathcal C_{p_1,s_1;p_2}(v_a)\right)^{1/p_2}.
\]
The comparison constants depend only on the two parameter triples,
$n$ and the fixed partition. In particular, neither $p_1\le p_2$ nor
$p_i\ge1$ is required.
\end{theorem}

\begin{proof}
Suppose (2) holds. By Proposition~\ref{prop:explicit-sequence},
$B=\sup_a\mathfrak D_{p_1,s_1;p_2}(v_a)$ is finite and
$B^{p_2}\approx\sup_a\mathcal C_{p_1,s_1;p_2}(v_a)$. For
$f\in F(p_1,q_1,s_1)$, set
\[
 c_{j,l}^{(0)}=\delta_j^{b_n}\sup_{Q_{j,l}}|f^{(n)}|,
       \qquad
 c_{j,l}^{(1)}=\delta_j^{b_n+1}\sup_{Q_{j,l}}|f^{(n+1)}|.
\]
Lemma~\ref{lem:F-sampling-synthesis} gives
$\|c^{(u)}\|_{\mathcal E_{p_1,s_1}}\le C\|f\|_{F(p_1,q_1,s_1)}$,
$u=0,1$. The derivative formula, the partition, and the definition of
$\mathfrak D$ imply, for every $a\in\mathbb D$,
\[
 \int_{\mathbb D}|(D_{\phi,\varphi}^{n}f)'(z)|^{p_2}\,d\lambda_a(z) \le
 C\sum_{j,l}v_{a,j,l} \bigl(|c_{j,l}^{(0)}|^{p_2} +|c_{j,l}^{(1)}|^{p_2}\bigr) \le C
 B^{p_2}\|f\|_{F(p_1,q_1,s_1)}^{p_2}.
\]
Also, \eqref{eq:F-derivative-evaluation} gives
\[
 |D_{\phi,\varphi}^{n}f(0)|
 \le C\frac{|\phi(0)|}{(1-|\varphi(0)|^2)^{b_n}}
                         \|f\|_{F(p_1,q_1,s_1)}.
\]
Taking the supremum over $a$ proves boundedness and the upper norm
estimate.

Conversely, suppose (1) holds. Put $A=(q_1+2)/p_1$, and choose $M$
as in Lemma~\ref{lem:F-sampling-synthesis} for the source parameters.
Use the kernels $g_{j,l,u}$ of Lemma~\ref{lem:two-kernels} and define
\[
 f_{c,u}(z)=\int_0^z\sum_{j,l}c_{j,l}g_{j,l,u}(\zeta)\,d\zeta,
                       \qquad u=0,1,
\]
for each finitely supported $c$. The synthesis estimate, with $M$ and
$M+1$, gives
$\|f_{c,u}\|_{F(p_1,q_1,s_1)}\le C\|c\|_{\mathcal E_{p_1,s_1}}$.
Fix $a\in\mathbb D$ and change the coefficients by independent random
signs. Khintchine's inequality, restricted successively to
$\varphi^{-1}(Q_{j,l})$, and \eqref{eq:two-kernels} with exponent
$r=p_2$ yield
\[
 \sum_{j,l}v_{a,j,l}|c_{j,l}|^{p_2} \le C\sum_{u=0}^1\mathbb E \int_{\mathbb D}
 |(D_{\phi,\varphi}^{n}f_{\varepsilon c,u})'(z)|^{p_2} \,d\lambda_a(z) \le
 C\|D_{\phi,\varphi}^{n}\|^{p_2} \|c\|_{\mathcal E_{p_1,s_1}}^{p_2}.
\]
The algebra in \eqref{eq:two-kernels} is applied at $w=\varphi(z)$
with $x=\phi'(z)$ and $y=\phi(z)\varphi'(z)$, but the kernels themselves
remain fixed. Taking the supremum over $c$ and then over $a$ proves (2)
and the lower bound for the sequence term. The lower bound in
\eqref{eq:F-derivative-evaluation}, applied at $\varphi(0)$, proves
\[
 \frac{|\phi(0)|}{(1-|\varphi(0)|^2)^{b_n}}
                  \le C\|D_{\phi,\varphi}^{n}\|.
\]
Combining these inequalities with \eqref{eq:sequence-explicit-equivalence} finishes the proof.
\end{proof}

\subsection{Compactness and the essential norm}

We retain the preceding pull-back measures and use the compact
dilations from Lemma~\ref{lem:space-basics}.

\begin{theorem}\label{thm:F-essential-norm}
Let $0<p_i<\infty$, $q_i>-2$, $0\le s_i<\infty$ and $q_i+s_i>-1$
for $i=1,2$. Let $n\ge1$ and let $\phi,\varphi$ be as in
Theorem~\ref{thm:F-boundedness}. Suppose
$D_{\phi,\varphi}^{n}:F(p_1,q_1,s_1)\to F(p_2,q_2,s_2)$ is bounded.
For $J=0,1,\ldots$, put
\begin{equation}\label{eq:F-tail-limit}
 L_J=\left[\sup_{a\in\mathbb D}\mathcal C_{p_1,s_1;p_2}
          (\{v_{a,j,l}:j\ge J\})\right]^{1/p_2},
              \qquad L=\lim_{J\to\infty}L_J,
\end{equation}
where the coordinates with $j<J$ are set equal to zero. Then
\[
 \|D_{\phi,\varphi}^{n}\|_{e,F(p_1,q_1,s_1)\to F(p_2,q_2,s_2)}
                    \approx L.
\]
In particular, the bounded operator is compact if and only if $L=0$.
All comparison constants depend only on the parameter triples, $n$
and the fixed partition, not on the two symbols.
\end{theorem}

\begin{proof}
Proposition~\ref{prop:explicit-sequence} gives, uniformly in $J$,
$L_J\approx\sup_a\mathfrak D_{p_1,s_1;p_2}(\{v_{a,j,l}:j\ge J\})$.
The numbers $L_J$ decrease and $L_0<\infty$ by
Theorem~\ref{thm:F-boundedness}, so the limit is well defined.
Testing the operator on $z^n/n!$ and $z^{n+1}/(n+1)!$ gives
$\phi,\phi\varphi\in F(p_2,q_2,s_2)$. Since $|\varphi|\le1$ and
$\phi\varphi'=(\phi\varphi)'-\phi'\varphi$, it follows that
\begin{equation}\label{eq:F-pullback-masses}
 \sup_{a\in\mathbb D}
       \bigl(\nu_{a,0}(\mathbb D)+\nu_{a,1}(\mathbb D)\bigr)<\infty.
\end{equation}

We first prove the upper estimate. By Lemma~\ref{lem:F-dilations},
$D_{\phi,\varphi}^{n}C_\rho:F(p_1,q_1,s_1)\to F(p_2,q_2,s_2)$ is compact whenever $\rho<1$.
For $\|f\|_{F(p_1,q_1,s_1)}\le1$, set $h_\rho=f-C_\rho f$. The functions
$h_\rho$ are uniformly bounded in $F(p_1,q_1,s_1)$, and, for fixed $J$, their
$n$th and $(n+1)$st derivatives tend uniformly to zero on
$|w|\le r_J$, uniformly over the unit ball of $F(p_1,q_1,s_1)$.
Split each derivative integral into $|\varphi(z)|<r_J$ and its
complement. The derivative formula gives
\[
 \int_{\mathbb D}|(D_{\phi,\varphi}^{n}h_\rho)'(z)|^{p_2}\,d\lambda_a(z) \le
 C\int_{\mathbb D}|h_\rho^{(n)}(w)|^{p_2}\,d\nu_{a,0}(w) +C\int_{\mathbb
 D}|h_\rho^{(n+1)}(w)|^{p_2}\,d\nu_{a,1}(w).
\]
On $|w|<r_J$, the two integrals tend to zero as $\rho\uparrow1$,
uniformly in $a$ and $f$, by \eqref{eq:F-pullback-masses}.
On $|w|\ge r_J$, the sampling estimate in
Lemma~\ref{lem:F-sampling-synthesis} and the definition of
$\mathfrak D$, together with \eqref{eq:sequence-explicit-equivalence}, bound their sum by
$CL_J^{p_2}\|h_\rho\|_{F(p_1,q_1,s_1)}^{p_2}\le CL_J^{p_2}$.
Here $|w|\ge r_J$ is exactly the union of the rectangles with $j\ge J$.
The scalar value
$D_{\phi,\varphi}^{n}h_\rho(0)=\phi(0)h_\rho^{(n)}(\varphi(0))$
also tends to zero uniformly in $f$. Consequently,
\begin{equation}\label{eq:F-dilation-upper}
 \limsup_{\rho\uparrow1}
       \|D_{\phi,\varphi}^{n}(I-C_\rho)\|_{F(p_1,q_1,s_1)\to F(p_2,q_2,s_2)}
                    \le C L_J.
\end{equation}
Taking the infimum over compact operators and then letting
$J\to\infty$ proves $\|D_{\phi,\varphi}^{n}\|_e\le CL$.

For the lower estimate, suppose $L>0$. Choose $a_J\in\mathbb D$
such that
$\mathfrak D_{p_1,s_1;p_2}(\{v_{a_J,j,l}:j\ge J\})\ge c_*L_J$.
Here $c_*>0$ is independent of $J$. By the definition of this quantity,
there is a finitely supported sequence $c$ with $\|c\|_{\mathcal E_{p_1,s_1}}\le1$, vanishing
for $j<J$, such that
\[
 \sum_{j\ge J,l}v_{a_J,j,l}|c_{j,l}|^{p_2}
                    \ge (c_*L_J/2)^{p_2}.
\]
Take $A=(q_1+2)/p_1$ and a fixed $M$ satisfying the two inequalities
in Lemma~\ref{lem:F-sampling-synthesis}. For $u=0,1$, let
$f_{c,u}$ be the primitives of the finite sums of $g_{j,l,u}$ used
in the proof of Theorem~\ref{thm:F-boundedness}. The same
Khintchine argument, now restricted to $|\varphi(z)|\ge r_J$,
yields
\[
 \sum_{j\ge J,l}v_{a_J,j,l}|c_{j,l}|^{p_2}
 \le C\sum_{u=0}^1\mathbb E
       \int_{\{|\varphi(z)|\ge r_J\}}
         |(D_{\phi,\varphi}^{n}f_{\varepsilon c,u})'(z)|^{p_2}
                         \,d\lambda_{a_J}(z).
\]
Select one value of $u$ and one choice of the finitely many signs for
which the integral is at least a fixed multiple of $L_J^{p_2}$.
The resulting function $f_J$ satisfies
\[
 \|f_J\|_{F(p_1,q_1,s_1)}\le C_0,\qquad
        \|D_{\phi,\varphi}^{n}f_J\|_{F(p_2,q_2,s_2)}\ge cL_J\ge cL.
\]
The local tail estimate in Lemma~\ref{lem:F-sampling-synthesis}
and $f_J(0)=0$ show that $f_J$ and all its fixed derivatives tend
locally uniformly to zero. Hence $D_{\phi,\varphi}^{n}f_J\to0$
locally uniformly. Each $f_J$ is analytic on a neighborhood of
$\overline{\mathbb D}$, since it is the primitive of a finite sum
of kernels with their poles outside that closed disk.

Lemma~\ref{lem:essential-test}, with the large target
$F(p_2,q_2,s_2)$, gives $L\le C\|D_{\phi,\varphi}^n\|_e$.
Its closedness assertion proves the compactness criterion.
\end{proof}

The supremum over $a$ in \eqref{eq:F-tail-limit} precedes the limit in
$J$. It is this uniform tail condition, not a separate limit for each
fixed $a$, that occurs in the essential norm. When $s_2=0$, the
measures do not depend on $a$, and the redundant supremum may be
omitted.

\subsection{The little spaces}

The same outer-row tests belong to the little source spaces. What
changes for a little target is the range condition on the two images
of the test polynomials.

\begin{theorem}\label{thm:F-little-operators}
Let $0<p_i<\infty$, $q_i>-2$, $s_i\ge0$ and $q_i+s_i>-1$ for
$i=1,2$, and let $n\ge1$. Use the symbols and sequence criterion of
Theorem~\ref{thm:F-boundedness}, and the tail limit $L$ in
\eqref{eq:F-tail-limit}.

\textup{(a)} If $s_1>0$, Theorems~\ref{thm:F-boundedness} and
\ref{thm:F-essential-norm} remain valid with source
$F_0(p_1,q_1,s_1)$ and target $F(p_2,q_2,s_2)$. Restriction to the
little source preserves the operator norm, and both essential norms
are comparable to $L$.

\textup{(b)} If $s_1,s_2>0$, the operator
$D_{\phi,\varphi}^n:F_0(p_1,q_1,s_1)\to F_0(p_2,q_2,s_2)$ is bounded
if and only if
\[
 \sup_a\mathcal C_{p_1,s_1;p_2}(v_a)<\infty,
       \qquad \phi,\phi\varphi\in F_0(p_2,q_2,s_2).
\]
Its norm equals the corresponding large-space operator norm, and its
essential norm is comparable to $L$. It is compact if and only if
$L=0$, equivalently if and only if the same formula defines a compact
map from $F(p_1,q_1,s_1)$ into $F_0(p_2,q_2,s_2)$.
When $s_1=0$ and $s_2>0$, all assertions in \textup{(b)} hold with
source $F(p_1,q_1,0)$ instead of a little source.
\end{theorem}

\begin{proof}
For (a), restriction cannot increase the norm. Conversely, if the
operator is bounded on the little source, then $C_\rho f$ belongs to
that source and has norm at most $\|f\|_{F(p_1,q_1,s_1)}$.
As $\rho\uparrow1$, its image converges locally to
$D_{\phi,\varphi}^nf$. Fatou in the large target proves boundedness
on the large source with no increase in norm. This proves the norm
equality and the common boundedness criterion.

Restriction of compact maps gives the upper essential-norm bound.
For the reverse bound, the test functions in
Theorem~\ref{thm:F-essential-norm} are analytic past
$\overline{\mathbb D}$ and hence belong to the little source with the
same quasi-norm. Lemma~\ref{lem:essential-test}, with the large target,
gives the lower bound by $cL$. Thus the compactness criterion is also
unchanged.

For (b), the tests $z^n/n!$ and $z^{n+1}/(n+1)!$ give the
necessary symbol memberships, and (a) gives the sequence condition
(Theorem~\ref{thm:F-boundedness} does so directly when $s_1=0$). Conversely, these memberships and
$\phi\varphi'=(\phi\varphi)'-\phi'\varphi$ imply
\[
 \lim_{|a|\uparrow1}\int_{\mathbb D}
       \bigl(|\phi'|^{p_2}+|\phi\varphi'|^{p_2}\bigr)\,d\lambda_a=0.
\]
Since $(\phi\varphi^k)'=\phi'\varphi^k+
 k\phi\varphi'\varphi^{k-1}$ for $k\ge1$, all polynomial images
belong to $F_0(p_2,q_2,s_2)$. Polynomial density in the source and
closedness of the little target prove sufficiency. The target norm
is inherited. For $s_1>0$, (a) gives equality of operator norms;
for $s_1=0$, the source itself is unchanged.

A compact map into the little target is also compact into the large
one; the preceding lower essential-norm bound therefore applies.
For the upper bound, $D_{\phi,\varphi}^nC_\rho$ is compact into the
little target: $C_\rho$ is compact from the large source into the little
source when $s_1>0$, and compact on $F(p_1,q_1,0)$ when $s_1=0$;
compose it with the bounded map just proved. By
\eqref{eq:F-dilation-upper}, its error is at most $CL_J$
in the limit $\rho\uparrow1$. Letting $J\to\infty$ gives the upper
bound by $CL$. If $L=0$, these same compact maps, defined on the
entire large source, converge in operator norm into the large target.
Their ranges lie in its closed little subspace. Thus the limiting
operator is compact into the little target as well. The reverse
implication follows by restriction.

The case $s_1=0$ uses polynomial density in $F(p_1,q_1,0)$ from
Lemma~\ref{lem:F-little-approximation}.
\end{proof}

The symbol memberships in (b) alone do not assert a little-space
range for the entire large source when $s_1>0$; that conclusion uses
$L=0$.

\subsection{Forms of the criteria in the separate parameter ranges}

The quantity $\mathcal C$ makes all the preceding $F$-space criteria
explicit. In particular, the same case distinction must be used in the
boundedness condition and in its uniform tail; one cannot replace a
nested condition by a global sum outside the range that justifies it.
The following consequences isolate the simplest forms.

\begin{corollary}\label{cor:F-explicit-regions}
Let $0<p_i<\infty$, $q_i>-2$, $0\le s_i<\infty$ and
$q_i+s_i>-1$ for $i=1,2$, and let $n\ge1$.
Use the symbols, $b_n$, $\lambda_a$ and $v_a$ from
Theorem~\ref{thm:F-boundedness}.
\begin{enumerate}
 \renewcommand{\labelenumi}{\textup{(\arabic{enumi})}}
 \setlength{\itemsep}{3pt}\setlength{\parskip}{0pt}
 \item If $s_1=0$, boundedness is equivalent to
 \[
 \begin{cases}
  \displaystyle\sup_{a,j,l}v_{a,j,l}<\infty,&p_1\le p_2,\\[1mm]
  \displaystyle\sup_a
       \left(\sum_{j,l}v_{a,j,l}^{p_1/(p_1-p_2)}
       \right)^{(p_1-p_2)/p_1}<\infty,&p_2<p_1.
 \end{cases}
 \]
 \item If $0<s_1\le1$, boundedness is equivalent to
 \[
  \sup_{a\in\mathbb D}\lim_{N\to\infty}
      \int_0^\infty R_{\partial\mathbb D}^{(N)}(t;v_a)\,dt<\infty.
 \]
 Use the covering rule \eqref{eq:tree-superlinear-rule} when
 $p_2\ge p_1$, and the power rule \eqref{eq:tree-sublinear-rule}
 when $p_2<p_1$, with $\theta=p_2/p_1$ and $b_I=|I|^{s_1}$.
 In the particular case $s_1=1$ and $p_2\ge p_1$, this is the
 boundary integral condition
 \[
  \sup_a\int_{\partial\mathbb D}
     \sup_{I_{j,l}\ni\zeta}\frac{v_{a,j,l}}{|I_{j,l}|}\,dm(\zeta)
       <\infty.
 \]
 \item If $s_1>1$, boundedness is equivalent to
 \begin{equation}\label{eq:F-Bloch-source-integral}
 \sup_a\int_{\mathbb D}
  \left(
    \frac{|\phi'(z)|^{p_2}}
         {(1-|\varphi(z)|^2)^{p_2b_n}}
    +\frac{|\phi(z)\varphi'(z)|^{p_2}}
         {(1-|\varphi(z)|^2)^{p_2(b_n+1)}}
  \right)d\lambda_a(z)<\infty.
 \end{equation}
\end{enumerate}
In each case the $p_2$th power of the operator norm is comparable
to the stated quantity plus
$|\phi(0)|^{p_2}(1-|\varphi(0)|^2)^{-p_2b_n}$.
Under boundedness, its essential norm to the power $p_2$ is comparable
to
\[
 \lim_{J\to\infty}\sup_{a\in\mathbb D}
      \mathcal C_{p_1,s_1;p_2}(\{v_{a,j,l}:j\ge J\}).
\]
In (1)--(2), evaluate the indicated formula after setting all earlier
entries to zero. In (3), the equivalent tail is obtained by restricting
the symbol integral to $|\varphi(z)|\ge R$, taking $\sup_a$ first,
and then letting $R\uparrow1$. Compactness is equivalent to a zero limit.
The same replacements apply to the operators on little spaces in
Theorem~\ref{thm:F-little-operators}, with its stated symbol
memberships unchanged.
\end{corollary}

\begin{proof}
Parts (1)--(2) are
Proposition~\ref{prop:explicit-sequence} applied to
Theorems~\ref{thm:F-boundedness}--\ref{thm:F-little-operators}.
For (3), the proposition gives $\mathcal C(v_a)=\sum_{j,l}v_{a,j,l}$.
On $Q_{j,l}$, $1-|w|^2\approx\delta_j$; summing the two pull-back
terms and using their definitions gives
\eqref{eq:F-Bloch-source-integral}. For tails, the union of all
rectangles with $j\ge J$ is exactly $\{|w|\ge r_J\}$.
The same comparison therefore holds for restricted integrals. The
limit over arbitrary $R\uparrow1$ agrees with the limit along $r_J$
by monotonicity. All constants are independent of $a,J,R$.
\end{proof}

In (3), it is the source that is a Bloch-type space, by the
identifications in Subsection~\ref{subsec:intro-spaces}; the target
may still have a local integral norm.

In (1) with $p_1\le p_2$, an equivalent geometric criterion is
\[
 \sup_{a,w}\left(
  \frac{\nu_{a,0}(\Delta(w,\eta))}{(1-|w|^2)^{p_2b_n}}
  +\frac{\nu_{a,1}(\Delta(w,\eta))}{(1-|w|^2)^{p_2(b_n+1)}}
             \right)<\infty,
\]
for any fixed $0<\eta<1$. Its boundary tail, $\lim_{R\uparrow1}\sup_{a,\,|w|>R}$ of the
parenthesized expression, is comparable to the essential norm to the
power $p_2$. This follows
from the rectangle--disk covering argument after
Corollary~\ref{cor:Hardy-automatic-compactness}, applied separately
to the two measures. It is not asserted in place of the summation
condition when $p_2<p_1$.

\begin{corollary}\label{cor:F-automatic-compactness}
Let the parameter triples be in the range of
Theorem~\ref{thm:F-boundedness}, with $s_2=0$. If either $p_2<p_1$
or $s_1>1$, every bounded operator
\[
 D_{\phi,\varphi}^n:F(p_1,q_1,s_1)\longrightarrow F(p_2,q_2,0),
                  \qquad n\ge1,
\]
is compact. For $s_1>0$, the same conclusion holds with the little
source $F_0(p_1,q_1,s_1)$.
\end{corollary}

\begin{proof}
The condition $s_2=0$ leaves one fixed array $v$, with no parameter
$a$. If $s_1>1$, boundedness gives $\sum v_{j,l}<\infty$, whose
tails tend to zero. Theorem~\ref{thm:F-essential-norm} applies.

For $p_2<p_1$, we show directly that a bounded sequence multiplier
has vanishing tails, for every $s_1\ge0$. Otherwise one could choose
successive, disjoint finite blocks of rows and coefficient arrays
$c^{(k)}$, supported in those blocks, with
\[
 \|c^{(k)}\|_{\mathcal E_{p_1,s_1}}\le1,
       \qquad \sum_{j,l}v_{j,l}|c_{j,l}^{(k)}|^{p_2}\ge\varepsilon^{p_2}
\]
for a fixed $\varepsilon>0$ and every $k$.
This follows from the finite-support definition of $\mathfrak D$
and the assumed nonvanishing tail. Disjointness gives
\[
 \left\|\sum_{k=1}^Kc^{(k)}\right\|_{\mathcal E_{p_1,s_1}}^{p_1}
       \le K,
 \qquad
 \sum_{j,l}v_{j,l}\left|\sum_{k=1}^Kc_{j,l}^{(k)}\right|^{p_2}
       \ge K\varepsilon^{p_2}.
\]
Boundedness of the multiplier bounds the second expression by
$\mathfrak D(v)^{p_2}K^{p_2/p_1}$, a contradiction as $K\to\infty$.
Thus the tails vanish, and Theorem~\ref{thm:F-essential-norm}
again gives compactness. The statement for the little source follows
from Theorem~\ref{thm:F-little-operators}(a).
\end{proof}

The fixed-array hypothesis in this corollary is essential to the
argument. When $s_2>0$, separate tail convergence for each $a$
does not establish the required uniform limit.

\smallskip\noindent
\emph{Discretizing the local parameter.}
Fix $0<\eta<1$ and a uniformly separated pseudohyperbolic net
$\{b_k\}$ whose $\eta$-disks cover $\mathbb D$.
If $|\sigma_b(a)|<\eta$, the automorphism identity gives
\[
 \frac{1-\eta}{1+\eta}
 \le
 \frac{(1-|a|^2)/|1-\overline az|^2}
      {(1-|b|^2)/|1-\overline bz|^2}
 \le\frac{1+\eta}{1-\eta},\qquad z\in\mathbb D.
\]
Indeed, this ratio equals
$(1-|c|^2)/|1-\overline c\,\sigma_b(z)|^2$, with $c=\sigma_b(a)$.
The densities of $\lambda_a$ and $\lambda_b$ therefore differ by
at most $((1+\eta)/(1-\eta))^{s_2}$.
Both $\mathfrak D(v)^{p_2}$ and $\mathcal C(v)$ are monotone and
homogeneous of degree one in $v$; for $\mathcal C$, this follows
from $R_I(t;cv)=R_I(t/c;v)$ when $c>0$.
Consequently, for every $J\ge0$,
\[
 \sup_k\mathcal C_{p_1,s_1;p_2}(v_{b_k}^{\ge J})
 \le \sup_a\mathcal C_{p_1,s_1;p_2}(v_a^{\ge J})
 \le\left(\frac{1+\eta}{1-\eta}\right)^{s_2}
       \sup_k\mathcal C_{p_1,s_1;p_2}(v_{b_k}^{\ge J}),
\]
where $v_a^{\ge J}$ denotes the array with earlier rows set to zero.
Thus the full and tail tests may be evaluated on a fixed countable net.

Finite-depth values increase with the depth and with the entries.
No data-independent stopping depth follows: an array with a single
entry equal to $1$ arbitrarily far out has all earlier truncations
zero and full $\mathcal C$-value $1$.
A certified truncation error requires an estimate of the discarded
tail. For example, directly from its definition,
$\mathfrak D(v)^{p_2}\le\mathfrak D(v^{\le N})^{p_2}+
\mathfrak D(v^{>N})^{p_2}$, where the superscripts denote row
restrictions. In a symbol criterion the tail bound must hold
uniformly over the local parameter, or equivalently over the net.

\begin{example}\label{ex:BMOA-local-dichotomy}
Let $r>0$, $\phi(z)=2+z$ and $\varphi(z)=z$. Then
\[
 D_{2+z,\operatorname{id}}^1:BMOA=F(2,0,1)
               \longrightarrow F(r,2r-2,1)
\]
is bounded if and only if $r\ge2$. Whenever it is bounded, it is not
compact. More precisely, for $r\ge2$,
$\mathcal C_{2,1;r}(\{v_{a,j,l}:j\ge J\})\approx1$ uniformly in
$a\in\mathbb D$ and $J\ge0$.
\end{example}

\begin{proof}
All target triples are admissible. Here $b_1=1$, $\phi'=1$ and
$1\le|2+z|\le3$, while
\[
 d\lambda_a(z)=(1-|z|^2)^{2r-1}
                     \frac{1-|a|^2}{|1-\overline az|^2}\,dA(z).
\]
On $Q_{j,l}$, the two terms of $v_a$ are comparable to
$\delta_j^{r+1}$ and $\delta_j$, respectively, times the same kernel.
Consequently,
\begin{equation}\label{eq:local-example-array}
 v_{a,j,l}\approx\delta_j
                \frac{1-|a|^2}{|1-\overline a a_{j,l}|^2},
\end{equation}
with constants independent of $a,j,l$.

Suppose $r\ge2$. If $\zeta\in I_{j,l}$, then
$|1-\overline a\zeta|\le C|1-\overline a a_{j,l}|$:
use $|\zeta-a_{j,l}|\le C\delta_j$ and
$|1-\overline a a_{j,l}|\ge3\delta_j/4$.
Along the nested arcs containing $\zeta$, the marked points tend to
$\zeta$. Since $|I_{j,l}|=\delta_j/8$, formula
\eqref{eq:critical-shadow-maximal} therefore gives, for every $J$,
\[
 \mathcal C_{2,1;r}(\{v_{a,j,l}:j\ge J\})
 \approx\int_{\partial\mathbb D}
              \frac{1-|a|^2}{|1-\overline a\zeta|^2}\,dm(\zeta)=1.
\]
Theorems~\ref{thm:F-boundedness} and \ref{thm:F-essential-norm}
prove boundedness and noncompactness.

If $r<2$, take $a=0$, so that $v_{0,j,l}\approx\delta_j$.
In \eqref{eq:tree-mass-constraints}, put $x_I=|I|/K$ on $K$
consecutive full levels and zero elsewhere. Every subtree constraint
holds, and the objective is comparable to $K^{1-r/2}$.
Letting $K\to\infty$ gives $\mathcal C_{2,1;r}(v_0)=\infty$, which
proves unboundedness.

For a direct boundedness check, the local Cauchy estimate gives
$\|h\|_{F(2,2,1)}\le C\|f\|_{BMOA}$ for $h=(2+z)f'$;
apply it to $f'$ on enlarged squares and use bounded overlap.
The same estimate gives $(1-|z|^2)^2|h'(z)|\le C\|f\|_{BMOA}$.
For $r\ge2$, multiplying the square integrand for $F(2,2,1)$ by
$((1-|z|^2)^2|h'|)^{r-2}$ proves the claimed $F(r,2r-2,1)$ bound.
The noncompactness also has a direct test. The polynomials $f_k=z^k$
have bounded $BMOA$ norms and tend locally to zero. For
$h_k=(2+z)f_k'$, one has
$h_k'=kz^{k-2}(2(k-1)+kz)$. On
$1-1/k<|z|<1-1/(2k)$, $k\ge4$, its modulus is at least $ck^2$.
The defining target integral at $a=0$ is therefore at least
$c_r k^{2r}k^{-(2r-1)}k^{-1}=c_r>0$.
Thus these images cannot have a norm-convergent subsequence.
\end{proof}

\section{Comparison and symbols with relatively compact range}
\label{sec:comparison-interior}

\subsection{A comparison theorem}

Increasing the weight exponent enlarges an $F$ space. The next result
compares simultaneous shifts and shows that a strictly larger target
shift forces compactness.

\begin{theorem}\label{thm:F-comparison}
Let $0<p_i<\infty$, $q_i>-2$, $s_i\ge0$ and $q_i+s_i>-1$ for
$i=1,2$, and let $n\ge1$. Suppose
$D_{\phi,\varphi}^n:F(p_1,q_1,s_1)\to F(p_2,q_2,s_2)$ is bounded.
For $0\le h\le k$, the same formula defines a bounded operator
\[
 D_{\phi,\varphi}^n:
 F(p_1,q_1+p_1h,s_1)\longrightarrow F(p_2,q_2+p_2k,s_2).
\]
Its operator norm and essential norm are at most
$C((1+|\varphi(0)|)/(1-|\varphi(0)|))^h$ times the respective
original norms. In particular, compactness is preserved. If $k>h$,
the shifted operator is compact even when the original operator is
only bounded; if also $s_2>0$, it is compact with codomain
$F_0(p_2,q_2+p_2k,s_2)$.

For $s_1>0$ the norm comparisons hold with both sources little.
For $s_1,s_2>0$ they also hold with both sources and both targets
little, provided the original little-space operator is bounded.
Constants depend only on the fixed parameters, $n,h,k$ and the partition.
\end{theorem}

\begin{proof}
Put $b=\varphi(0)$ and $K=(1+|b|)/(1-|b|)$. Schwarz--Pick gives
\begin{equation}\label{eq:comparison-Schwarz-Pick}
 1-|z|^2\le1-|\sigma_b(\varphi(z))|^2
 =\frac{(1-|b|^2)(1-|\varphi(z)|^2)}
        {|1-\overline b\varphi(z)|^2}
 \le K(1-|\varphi(z)|^2).
\end{equation}
In particular, $1-|z|^2\le2K\delta_j$ when
$\varphi(z)\in Q_{j,l}$; see also \cite{CowenMacCluer1995}.
For the shifted spaces, the target weight is
$(1-|z|^2)^{p_2k}d\lambda_a(z)$ and the source exponent is $b_n+h$.
Write $v_a^{(h,k)}$ for the resulting array. Since $k\ge h$,
\[
 v_{a,j,l}^{(h,k)}\le(2K)^{p_2h}v_{a,j,l},\qquad
 v_{a,j,l}^{(h,k)}\le(2K)^{p_2k}
                      \delta_j^{p_2(k-h)}v_{a,j,l}.
\]
The first inequality uses $(1-|z|^2)^{p_2(k-h)}\le1$; the second
uses \eqref{eq:comparison-Schwarz-Pick} on the full additional weight.
The source sequence space $\mathcal E_{p_1,s_1}$ is unchanged.
Thus monotonicity and homogeneity of $\mathfrak D$, followed by
Proposition~\ref{prop:explicit-sequence}, control both the full
sequence quantity and each tail. The scalar term satisfies
\[
 \frac{|\phi(0)|}{(1-|b|^2)^{b_n+h}}
       \le K^h\frac{|\phi(0)|}{(1-|b|^2)^{b_n}}.
\]
Theorems~\ref{thm:F-boundedness} and \ref{thm:F-essential-norm}
now give the norm comparisons. If $k>h$, the second array inequality
bounds the shifted tail, after taking its $p_2$th root, by
$C K^k2^{-J(k-h)}\sup_a\mathfrak D_{p_1,s_1;p_2}(v_a)$.
It tends to zero, proving compactness.

If $s>0$ and $u>0$, then
\begin{equation}\label{eq:strict-weight-little}
 F(p,q,s)\subset F_0(p,q+pu,s).
\end{equation}
Indeed, on $|z|>R$ the new defining integral is at most
$(1-R^2)^{pu}\|f\|_{F(p,q,s)}^p$, uniformly in its local parameter.
On $|z|\le R$ its kernel tends uniformly to zero as $|a|\uparrow1$.
Let first $|a|\uparrow1$ and then $R\uparrow1$.
For $k>h$, the polynomial tests give
$\phi,\phi\varphi\in F(p_2,q_2,s_2)$, and $k>0$ puts both functions
in the shifted little target by \eqref{eq:strict-weight-little}.
Together with the vanishing shifted tail,
Theorem~\ref{thm:F-little-operators}(b) gives compactness into that
target, also when $s_1=0$.

For little sources, apply Theorem~\ref{thm:F-little-operators}(a).
For little sources and targets, its part (b) applies because increasing
the weight exponent preserves the two little-space memberships.
The same sequence comparisons then prove the claimed norm bounds.
\end{proof}

The dependence on $\varphi(0)$ in the operator-norm comparison cannot
in general be removed when $h>0$. Indeed, take $\phi\equiv1$ and
$\varphi\equiv b\in\mathbb D$. The output is the constant
$f^{(n)}(b)$, whose norm is its modulus in both target spaces.
By \eqref{eq:F-derivative-evaluation}, the ratio of the shifted
operator norm to the original one is comparable to
$(1-|b|^2)^{-h}$, which is unbounded as $|b|\uparrow1$.
The same test applies to a little source: the kernels giving the
evaluation lower bound are analytic past $\overline{\mathbb D}$.
The output is constant and hence also belongs to every permitted
little target. When $s_1=0$ there is no distinct little source under
our convention.

\subsection{Supremum-norm targets}

The finite Hardy exponent used in the integral criteria is not needed
for a supremum-norm target. The endpoint has a direct pointwise form.

In the following table, mixed-norm parameters satisfy
$0<p,q\le\infty$, $\alpha>0$; the $F$ parameters satisfy
$0<p<\infty$, $q>-2$, $s\ge0$, $q+s>-1$, with $s>0$ for $F_0$.

\begin{theorem}\label{thm:supremum-targets}
Let $X$, the integer $n$, and the exponent $b$ be chosen from
\[
\begin{array}{c|c|c}
 X&n&b\\ \hline
 H(p,q,\alpha)&n\ge0&\alpha+1/p+n\\
 S(p,q,\alpha)&n\ge1&\alpha+1/p+n-1\\
 F(p,q,s)\text{ or }F_0(p,q,s)&n\ge1&(q+2)/p+n-1.
\end{array}
\]
Let $\phi\in H(\mathbb D)$ and let $\varphi$ be an analytic self-map.
Then:

\textup{(a)} $D_{\phi,\varphi}^n:X\to H^\infty$ is bounded if and only if
\[
 \sup_{z\in\mathbb D}
       \frac{|\phi(z)|}{(1-|\varphi(z)|^2)^b}<\infty,
\]
and its norm is comparable to this supremum.

\textup{(b)} $D_{\phi,\varphi}^n:X\to S^\infty$ is bounded if and only if
\[
 \sup_{z\in\mathbb D}\left(
       \frac{|\phi'(z)|}{(1-|\varphi(z)|^2)^b}
       +\frac{|\phi(z)\varphi'(z)|}{(1-|\varphi(z)|^2)^{b+1}}
                       \right)<\infty.
\]
Its norm is comparable to this supremum plus
$|\phi(0)|(1-|\varphi(0)|^2)^{-b}$.

In either part, under boundedness the essential norm is comparable to
the limiting supremum of the displayed pointwise expression over
$\{z:|\varphi(z)|>R\}$ as $R\uparrow1$. Thus compactness is equivalent
to that limit being zero. The supremum of an empty set is taken as zero.
All comparison constants depend only on the source parameters and $n$.
\end{theorem}

\begin{proof}
The sampling and synthesis estimates give
\[
 \sup_{\|f\|_X\le1}|f^{(n)}(w)|\approx(1-|w|^2)^{-b},
 \qquad
 \sup_{\|f\|_X\le1}|f^{(n+1)}(w)|\approx(1-|w|^2)^{-b-1}.
\]
For a mixed-norm source use Lemma~\ref{lem:discrete-estimates}; for
its derivative space integrate the kernels once; for the $F$ sources
use Lemma~\ref{lem:F-sampling-synthesis}. The one-kernel test functions
are analytic past the closed disk, so the lower bounds apply to the
little source as well. Taking suprema proves (a).

The product rule and these estimates prove the upper bound in (b).
For its reverse bound, fix $z$ and choose the rectangle containing
$\varphi(z)$. Apply Lemma~\ref{lem:two-kernels} to
$x=\phi'(z)$ and $y=\phi(z)\varphi'(z)$, using the two kernels for a
mixed-norm source and their primitives for a derivative source.
For the former use that lemma with derivative orders $n,n+1$;
for the latter use orders $n-1,n$. Their source norms are uniformly
bounded by the same synthesis estimates. This bounds the two-term
pointwise expression by $C\|D_{\phi,\varphi}^n\|$. Evaluation at zero
supplies the remaining scalar term.

For the essential-norm upper bound use the compact approximants
$D_{\phi,\varphi}^nC_\rho$ from Lemma~\ref{lem:space-basics}.
On $|\varphi(z)|\le R$, the needed derivatives of $f-C_\rho f$
converge uniformly to zero on the source unit ball. The coefficient
functions are bounded: polynomial tests give $\phi\in H^\infty$
in (a), and $\phi,\phi\varphi\in S^\infty$ in (b), which also gives
$\phi\varphi'\in H^\infty$. On the complement use the displayed
pointwise expression and the uniformly bounded source quasi-norm of
$f-C_\rho f$. The value at zero in (b) also tends to zero. Let first
$\rho\uparrow1$ and then $R\uparrow1$.

For a positive boundary limit, choose $z_j$ with
$|\varphi(z_j)|\to1$ realizing a fixed fraction of that limit. The
one-kernel test in (a), or one of the two kernels in (b), gives functions
of uniformly bounded source norm whose image norms are bounded below
by a fixed fraction of the limit. Choose the kernel order sufficiently
large as in the synthesis lemmas. Their local tail estimates show that
the functions and their fixed derivatives tend locally to zero, and
so do their images. Lemma~\ref{lem:essential-test}, with target
$H^\infty$ or $S^\infty$, proves the lower bound and the compactness
criterion.
\end{proof}

\begin{proposition}\label{prop:subcritical-zero-order}
Let $X$ be $S(p,q,\alpha)$ with $\alpha+1/p<1$, or $F(p,q,s)$
with $(q+2)/p<1$, in the parameter ranges above; the little $F$
source is also allowed. For every analytic self-map $\varphi$ and
$0<r\le\infty$,
\[
 D_{\phi,\varphi}^0:X\to H^r\text{ bounded}
 \iff D_{\phi,\varphi}^0:X\to H^r\text{ compact}
 \iff \phi\in H^r.
\]
Its norm is comparable to $\|\phi\|_{H^r}$, uniformly in $\varphi$.

For the source $S(p,q,\alpha)$ in this subcritical range,
$D_{\phi,\varphi}^0:S(p,q,\alpha)\to S^r$ is bounded if and only if
$\phi\in S^r$ and
$D_{\phi\varphi',\varphi}^0:H(p,q,\alpha)\to H^r$ is bounded.
Its norm is comparable to $\|\phi\|_{S^r}$ plus the norm of this last
operator, and their essential norms are equal. Thus
Theorem~\ref{thm:mixed-hardy} (or Theorem~\ref{thm:supremum-targets}
when $r=\infty$) gives an explicit criterion at order zero. If $q>r$, boundedness
also implies compactness in this derivative-target assertion.
\end{proposition}

\begin{proof}
In each source, $|f'(z)|\le C\|f\|_X(1-|z|)^{-a}$ for some
$0<a<1$. Radial integration gives $\|f\|_\infty\le C\|f\|_X$ and
\[
 \|f-C_\rho f\|_\infty
       \le \frac{C}{1-a}(1-\rho)^{1-a}\|f\|_X.
\]
Indeed, integrate $zf'(tz)$ for $\rho\le t\le1$ and use
$1-t|z|\ge1-t$. A fixed $C_\rho:X\to H^\infty$ is compact by
local normality. The displayed estimate therefore makes the inclusion
$X\to H^\infty$ compact. The map $h\mapsto\phi(h\circ\varphi)$ from
$H^\infty$ to $H^r$ has norm $\|\phi\|_{H^r}$ when $\phi\in H^r$.
Composition proves sufficiency and compactness, and the test $1$
proves necessity and the lower norm bound.

For the derivative target, testing $1$ gives $\phi\in S^r$.
The just proved compact inclusion makes
$f\mapsto\phi'(f\circ\varphi):S(p,q,\alpha)\to H^r$ compact.
The identity
\[
 (D_{\phi,\varphi}^0f)'=
       \phi'(f\circ\varphi)+D_{\phi\varphi',\varphi}^0(f')
\]
and the primitive right inverse therefore prove the boundedness
criterion and both norm inequalities. The scalar value at zero is
finite rank. After integrating the displayed identity, the difference
between $D_{\phi,\varphi}^0$ and
$J D_{\phi\varphi',\varphi}^0\partial$ is compact.
Differentiation and $J$ are contractions in both directions between
the relevant derivative spaces and their derivative components.
Composing compact approximants with these maps proves equality of
essential norms, as in Lemma~\ref{lem:primitive-reduction}.
\end{proof}

\subsection{Symbols with relatively compact range}

When $\varphi(\mathbb D)$ stays in a compact subdisk, the source
functions and their derivatives are uniformly controlled there.
The issue is then membership of the two functions $\phi$ and
$\phi\varphi$ in the target derivative space.

In the next theorem, the source $X$ can be $H(p,q,\alpha)$ or
$S(p,q,\alpha)$, with $0<p,q\le\infty$ and $\alpha>0$, or it can
be $F(p_1,q_1,s_1)$, with
$0<p_1<\infty$, $q_1>-2$, $s_1\ge0$ and $q_1+s_1>-1$.
We also allow $X=F_0(p_1,q_1,s_1)$ when $s_1>0$.
The target $Y$ can be $H^r$ or $S^r$, with $0<r\le\infty$, or
$F(p_2,q_2,s_2)$, with
$0<p_2<\infty$, $q_2>-2$, $s_2\ge0$ and $q_2+s_2>-1$.
The target $F_0(p_2,q_2,s_2)$ is allowed when $s_2>0$.
All little spaces carry the inherited quasi-norms.

\begin{theorem}\label{thm:relatively-compact-symbol}
Let $X,Y$ be any of the source and target spaces just specified.
Let $n\ge0$ be an integer, let $\phi\in H(\mathbb D)$, and let
$\varphi$ be an analytic self-map of $\mathbb D$ with
$\|\varphi\|_\infty<1$. The following conditions are equivalent:
\begin{enumerate}
 \renewcommand{\labelenumi}{\textup{(\arabic{enumi})}}
 \setlength{\itemsep}{2pt}\setlength{\parskip}{0pt}
 \item $D_{\phi,\varphi}^{n}:X\to Y$ is bounded.
 \item $D_{\phi,\varphi}^{n}:X\to Y$ is compact.
 \item If $Y=H^r$, then $\phi\in H^r$. For any of the derivative
       targets $S^r$, $F(p_2,q_2,s_2)$ and $F_0(p_2,q_2,s_2)$,
       one has $\phi,\phi\varphi\in Y$.
\end{enumerate}
When these conditions hold,
\[
 \|D_{\phi,\varphi}^{n}\|_{X\to Y}
 \approx
 \begin{cases}
  \|\phi\|_{H^r},&Y=H^r,\\
  \|\phi\|_Y+\|\phi\varphi\|_Y,&
       Y=S^r,\ F(p_2,q_2,s_2)\text{ or }F_0(p_2,q_2,s_2).
 \end{cases}
\]
For each fixed $0<\rho<1$, the comparison constants can be chosen
uniformly for $\|\varphi\|_\infty\le\rho$; they depend only on
$\rho$, $n$ and the parameters of $X,Y$.
\end{theorem}

\begin{proof}
Each source contains the polynomials. Testing (1) on $z^n/n!$ gives
$\phi\in Y$, and testing it on $z^{n+1}/(n+1)!$ gives
$\phi\varphi\in Y$. For $Y=H^r$, the second membership already
follows from the first because $|\varphi|\le1$. These tests prove
(1)$\Rightarrow$(3), as well as the lower operator-norm estimates.
They include $n=0$, with the usual convention $0!=1$.

Suppose (3) holds, and fix $0<\rho<1$ with
$\|\varphi\|_\infty\le\rho$. The local estimate \eqref{eq:source-local-interior} holds on each
listed source, and on its little subspace by restriction.

If $Y=H^r$, take $m=n$ in \eqref{eq:source-local-interior}. Then
$|D_{\phi,\varphi}^{n}f(z)|\le C\|f\|_X|\phi(z)|$ for every
$z\in\mathbb D$, so
$\|D_{\phi,\varphi}^{n}f\|_{H^r}\le C\|\phi\|_{H^r}\|f\|_X$.
This proves boundedness and the upper norm estimate for this target.

Consider a derivative target. The product rule gives
\[
 (D_{\phi,\varphi}^{n}f)'
   =\phi'(f^{(n)}\circ\varphi)
                +\phi\varphi'(f^{(n+1)}\circ\varphi).
\]
Since $\phi\varphi'=(\phi\varphi)'-\phi'\varphi$, the local estimates
with $m=n,n+1$ imply
\[
 \begin{aligned}
 |(D_{\phi,\varphi}^{n}f)'(z)|
  &\le C\|f\|_X\bigl(|\phi'(z)|+|(\phi\varphi)'(z)|\bigr),\\
 |D_{\phi,\varphi}^{n}f(0)|
  &\le C\|f\|_X|\phi(0)|.
 \end{aligned}
\]
For $Y=S^r$, take Hardy means of the derivative, or its supremum
when $r=\infty$. For a large $F$
target, integrate its $p_2$th power against $\lambda_a$ and take the
supremum over $a$. In either case, the defining quasi-norm and the
inequality for a power of a sum give
\[
 \|D_{\phi,\varphi}^{n}f\|_Y
       \le C\bigl(\|\phi\|_Y+\|\phi\varphi\|_Y\bigr)\|f\|_X.
\]
For $Y=F_0(p_2,q_2,s_2)$ the same bound holds in the large space.
Moreover, the two defining derivative integrals for $\phi$ and
$\phi\varphi$ tend to zero as $|a|\uparrow1$. The displayed
pointwise estimate therefore puts $D_{\phi,\varphi}^{n}f$ in the
little target as well. This proves (3)$\Rightarrow$(1) and all the
upper norm estimates.
It remains to prove compactness under (3). Choose
$\|\varphi\|_\infty<\eta<1$. The function $\varphi/\eta$ is an
analytic self-map with relatively compact range, and
$\phi(\varphi/\eta)=\eta^{-1}\phi\varphi$. Thus the boundedness
implication already proved applies to
$D_{\phi,\varphi/\eta}^{n}:X\to Y$. Direct differentiation gives
\begin{equation}\label{eq:interior-factorization}
 D_{\phi,\varphi}^{n}
       =\eta^{-n}D_{\phi,\varphi/\eta}^{n}C_\eta,
                \qquad C_\eta f(z)=f(\eta z).
\end{equation}
The dilation is compact by Lemmas~\ref{lem:space-basics} and
\ref{lem:F-little-approximation}.
The factorization proves (3)$\Rightarrow$(2), and
(2)$\Rightarrow$(1) is immediate. Uniformity of the norm constants
follows from \eqref{eq:source-local-interior} and the two polynomial
tests.
\end{proof}

For fixed symbols with $\|\varphi\|_\infty<1$ and a fixed target,
the boundedness and compactness conditions are independent of the
order and of the source within the stated ranges. The norm estimates
need not be uniform in those parameters.

The following is the weighted Dirichlet specialization.

\begin{corollary}\label{cor:Dirichlet-interior}
Let $\beta>-1$, let $n\ge1$ be an integer, let
$\phi\in\mathcal D_\beta$, and let $\varphi$ be an analytic self-map
of $\mathbb D$ with $\|\varphi\|_\infty<1$. The following conditions
are equivalent:
\begin{enumerate}
 \renewcommand{\labelenumi}{\textup{(\arabic{enumi})}}
 \setlength{\itemsep}{2pt}\setlength{\parskip}{0pt}
 \item $D_{\phi,\varphi}^{n}:\mathcal D_\beta\to\mathcal D_\beta$
       is bounded.
 \item $D_{\phi,\varphi}^{n}:\mathcal D_\beta\to\mathcal D_\beta$
       is compact.
 \item $\phi\varphi\in\mathcal D_\beta$.
\end{enumerate}
Condition \textup{(3)} is also equivalent to $\phi\varphi'\in A_\beta^2$.
\end{corollary}

\begin{proof}
Apply Theorem~\ref{thm:relatively-compact-symbol} with
both spaces equal to $F(2,\beta,0)$. Since $\phi'\varphi\in A_\beta^2$, the identity
$(\phi\varphi)'=\phi'\varphi+\phi\varphi'$ gives the last equivalence.
\end{proof}

\section{Multipliers and finite-rank approximation}
\label{sec:multipliers-approximation}

\subsection{A characterization involving all weights}

We consider the weights in either $F(p,q,s)$ or its little space.
For either such space $X$, let
\[
 \mathcal M(X)=\{g\in H(\mathbb D):gf\in X\text{ for every }f\in X\},
                  \qquad M_gf=gf.
\]
Since $1\in X$, every multiplier belongs to $X$. The next lemma
also puts it in $H^\infty$, so $\mathcal M(X)\subset X\cap H^\infty$.
Automatic boundedness of multipliers in this setting is also used
in \cite[Section~6]{PauZhao2014}. The following standard argument
includes the little spaces and $p<1$.

\begin{lemma}\label{lem:F-multiplier-bound}
Let $0<p<\infty$, $q>-2$, $s\ge0$ and $q+s>-1$. Let $X$ be
$F(p,q,s)$, or $F_0(p,q,s)$ when $s>0$. If $g\in\mathcal M(X)$,
then $M_g:X\to X$ is bounded and
$\|g\|_\infty\le\|M_g\|_{X\to X}$.
\end{lemma}

\begin{proof}
By Lemmas~\ref{lem:F-dilations} and \ref{lem:F-little-approximation},
$X$ is a quasi-Banach space, and convergence in $X$ implies local
uniform convergence. The graph of $M_g$ is closed: if $f_j\to f$ and
$gf_j\to h$ in $X$, their local limits give $h=gf$. The closed graph
theorem for quasi-Banach spaces therefore makes $M_g$ bounded; see
\cite{KaltonPeckRoberts1984}. This version of the theorem does not
require local convexity.

For $w\in\mathbb D$, the evaluation $e_w(f)=f(w)$ is a continuous
linear functional on $X$. It is nonzero because $1\in X$. The identity
$e_wM_g=g(w)e_w$ gives
\[
 |g(w)|\|e_w\|_{X^*}
     =\|e_wM_g\|_{X^*}
     \le\|e_w\|_{X^*}\|M_g\|_{X\to X}.
\]
Cancel $\|e_w\|_{X^*}>0$ and take the supremum over $w$.
\end{proof}

We equip $\mathcal M(X)$ with the multiplier quasi-norm
$\|g\|_{\mathcal M(X)}=\|M_g\|_{X\to X}$.

\begin{theorem}\label{thm:F-all-weights}
Let $0<p<\infty$, $q>-2$, $s\ge0$ and $q+s>-1$. Let $X$ be
$F(p,q,s)$, or $F_0(p,q,s)$ when $s>0$. Fix an integer $n\ge1$
and an analytic self-map $\varphi$ of $\mathbb D$. The following
conditions are equivalent:
\begin{enumerate}
 \renewcommand{\labelenumi}{\textup{(\arabic{enumi})}}
 \setlength{\itemsep}{2pt}\setlength{\parskip}{0pt}
 \item For every $\phi\in X$, the operator
       $D_{\phi,\varphi}^{n}:X\to X$ is bounded.
 \item For every $\phi\in X$, the operator
       $D_{\phi,\varphi}^{n}:X\to X$ is compact.
 \item $\|\varphi\|_\infty<1$ and $\varphi\in\mathcal M(X)$.
\end{enumerate}
Under these conditions, for every fixed $0<\rho<1$ with
$\|\varphi\|_\infty\le\rho$ and every $\phi\in X$,
\[
 c\|\phi\|_X
   \le\|D_{\phi,\varphi}^{n}\|_{X\to X}
   \le C_\rho\bigl(1+\|\varphi\|_{\mathcal M(X)}\bigr)\|\phi\|_X.
\]
Here $c>0$ depends only on $X,n$, and $C_\rho$ depends only on
$X,n,\rho$, not on the two symbols. In \textup{(1)} no uniform bound
on the operator norms as $\phi$ varies is assumed.
\end{theorem}

\begin{proof}
Suppose (1) holds. For each fixed $f\in X$, the function
$g=f^{(n)}\circ\varphi$ satisfies
$g\phi=D_{\phi,\varphi}^{n}f\in X$ for every $\phi\in X$.
Thus $g\in\mathcal M(X)$ and, by
Lemma~\ref{lem:F-multiplier-bound}, $g\in H^\infty$.

Consider the linear map $f\mapsto f^{(n)}\circ\varphi$ from $X$ to
$H^\infty$. Its graph is closed. Indeed, if $f_j\to f$ in $X$ and
$f_j^{(n)}\circ\varphi\to h$ in $H^\infty$, local uniform convergence
of the derivatives gives $h(z)=f^{(n)}(\varphi(z))$ for every
$z\in\mathbb D$. The closed graph theorem, now with codomain
$H^\infty$, supplies a constant $C$ depending on $X,n,\varphi$ such
that
\[
 \|f^{(n)}\circ\varphi\|_\infty\le C\|f\|_X,
                         \qquad f\in X.
\]
For the large space, the derivative-evaluation estimate is
\eqref{eq:F-derivative-evaluation}. It holds with the same comparison
constants on the little space: the upper bound follows by restriction,
and the one-kernel primitives used for the lower bound are analytic
past $\overline{\mathbb D}$, so belong to the little space by
Lemma~\ref{lem:F-little-approximation}. Consequently,
\[
 (1-|\varphi(z)|^2)^{-(q+2)/p-n+1}
   \le C'\sup_{\|f\|_X\le1}|f^{(n)}(\varphi(z))|
   \le C'',\qquad z\in\mathbb D.
\]
The exponent $(q+2)/p+n-1$ is positive. Hence
$\|\varphi\|_\infty<1$.

The polynomial $f(z)=z^{n+1}/(n+1)!$ belongs to $X$. Applying (1) to
this $f$ gives $\phi\varphi\in X$ for every $\phi\in X$.
Thus $\varphi\in\mathcal M(X)$, proving (3).

Conversely, suppose (3) holds and fix $\phi\in X$.
Then $\phi,\phi\varphi\in X$. Theorem~\ref{thm:relatively-compact-symbol},
with this source and target, proves that $D_{\phi,\varphi}^{n}$ is
compact. This proves (3)$\Rightarrow$(2); the implication
(2)$\Rightarrow$(1) is immediate.

Finally, the upper estimate in
Theorem~\ref{thm:relatively-compact-symbol} gives
\[
 \|D_{\phi,\varphi}^{n}\|_{X\to X}
   \le C_\rho\bigl(\|\phi\|_X+\|\phi\varphi\|_X\bigr)
   \le C_\rho\bigl(1+\|\varphi\|_{\mathcal M(X)}\bigr)\|\phi\|_X.
\]
Testing on $z^n/n!$, whose image is $\phi$, proves the lower estimate.
\end{proof}

One parameter range removes the multiplier-algebra notation altogether.
For the large space with $s>0$ and $(q+2)/p<1$, the identity
$\mathcal M(F(p,q,s))=F(p,q,s)$ is
\cite[Theorem~6.2(i)]{PauZhao2014}.
Here the same product estimate applies to the little space and the
$s=0$ endpoint. Indeed, if $(q+2)/p<1$, the derivative estimate and radial integration imply
$\|f\|_\infty\le C\|f\|_X$, for either large or little $X$ in
Theorem~\ref{thm:F-all-weights}. The product rule gives
$\|fg\|_X\le C\|f\|_X\|g\|_X$; for the little space, the
same bound on each defining integral preserves its vanishing property.
Hence $\mathcal M(X)=X$ in this range, since $1\in X$ also gives
the reverse inclusion. Condition \textup{(3)} of the theorem becomes
simply $\|\varphi\|_\infty<1$ and $\varphi\in X$.
For the remaining parameters we retain $\mathcal M(X)$.
Further large-space identifications are given in
\cite[Theorem~6.2(ii)--(iii)]{PauZhao2014} on the parameter ranges
specified there.

The positive-order hypothesis cannot be dropped: $F(2,-1/2,0)$
is an algebra by the preceding argument, so
$D_{\phi,\operatorname{id}}^0=M_\phi$ is bounded for every weight
in that space although $\|\operatorname{id}\|_\infty=1$.

If the weights are restricted to $\mathcal M(X)$, boundedness for
every such weight is equivalent simply to boundedness of
$D_{1,\varphi}^{n}:X\to X$. The same equivalence holds for compactness.
Indeed, $1\in\mathcal M(X)$ and
$D_{\phi,\varphi}^{n}=M_\phi D_{1,\varphi}^{n}$, with $M_\phi$
bounded. Thus the criteria in
Theorems~\ref{thm:F-boundedness}--\ref{thm:F-little-operators}, with
$\phi=1$, also answer the corresponding question with a quantifier
over multiplier weights; compare
\cite[Problem~8.6]{BarikSenPaul2026}.

\subsection{Approximation numbers and nuclearity}

We quantify Theorem~\ref{thm:relatively-compact-symbol} by Taylor
expansion on an interior disk. The approximants work for all its
source and target spaces, including the quasi-Banach cases.

For a bounded linear map $A:X\to Y$, its approximation numbers are
\[
 a_N(A)=\inf\{\|A-K\|_{X\to Y}:K:X\to Y\text{ is bounded and linear},
                   \ \operatorname{rank}K<N\},\qquad N\ge1.
\]
This definition applies also to the quasi-norms used here. We write
$X^*$ for the continuous linear functionals on $X$, with
$\|\ell\|_{X^*}=\sup_{\|f\|_X\le1}|\ell(f)|$.
For Banach spaces $X,Y$ and $0<t\le1$, an operator is called
$t$-nuclear if it has a representation
$Af=\sum_k\ell_k(f)y_k$, where $\ell_k\in X^*$, $y_k\in Y$ and
$\sum_k(\|\ell_k\|_{X^*}\|y_k\|_Y)^t<\infty$; the case $t=1$
is nuclearity. See \cite{Pietsch1980} for these operator-ideal notions.

\begin{theorem}\label{thm:interior-approximation}
Let $X,Y$ be any of the source and target spaces specified before
Theorem~\ref{thm:relatively-compact-symbol}. Let $n\ge0$ be an integer,
let $\phi\in H(\mathbb D)$, and let $\varphi$ be an analytic self-map
of $\mathbb D$ with $\|\varphi\|_\infty<1$. Suppose that
$D_{\phi,\varphi}^{n}:X\to Y$ is bounded.
\begin{enumerate}
 \renewcommand{\labelenumi}{\textup{(\arabic{enumi})}}
 \setlength{\itemsep}{3pt}\setlength{\parskip}{0pt}
 \item For $0<\rho<\eta<1$ with $\|\varphi\|_\infty\le\rho$,
 \[
  a_N(D_{\phi,\varphi}^{n})
       \le C\eta^N\|D_{\phi,\varphi}^{n}\|_{X\to Y},
                          \qquad N\ge1.
 \]
 The constant depends only on $X,Y,n,\rho,\eta$, not on the symbols
 or $N$. In particular,
 \[
  \limsup_{N\to\infty}a_N(D_{\phi,\varphi}^{n})^{1/N}
                         \le\|\varphi\|_\infty.
 \]
 \item Put $\ell_k(f)=f^{(n+k)}(0)/k!$ and
 $y_k=\phi\varphi^k$, with $\varphi^0=1$. Then $\ell_k\in X^*$,
 $y_k\in Y$, and
 \begin{equation}\label{eq:interior-rank-one}
  D_{\phi,\varphi}^{n}
       =\sum_{k=0}^{\infty}\ell_k\otimes y_k,
       \qquad (\ell_k\otimes y_k)f=\ell_k(f)y_k,
 \end{equation}
 with convergence in the operator quasi-norm. Moreover,
 \[
  \sum_{k=0}^{\infty}
        (\|\ell_k\|_{X^*}\|y_k\|_Y)^t<\infty
                         \qquad\text{for every }t>0.
 \]
 Consequently, when $X,Y$ are Banach spaces, the operator is
 $t$-nuclear for every $0<t\le1$.
\end{enumerate}
In particular, every bounded operator of this form from
$\mathcal D_{\beta_1}$ to $\mathcal D_{\beta_2}$, with
$\beta_1,\beta_2>-1$ and $\|\varphi\|_\infty<1$, belongs to all
Schatten classes $\mathcal S_t$, $t>0$.
\end{theorem}

\begin{proof}
Fix $0<\rho<\eta<1$ as in (1), and choose a radius $R$ with
$\rho/\eta<R<1$. In particular, $\rho<R$ and $\rho/R<\eta$.
Cauchy's coefficient formula for $f^{(n)}$, together with
\eqref{eq:source-local-interior}, gives
\begin{equation}\label{eq:interior-coefficient-bound}
 |\ell_k(f)|\le R^{-k}\max_{|w|=R}|f^{(n)}(w)|
       \le C_{R,n,X}R^{-k}\|f\|_X,\qquad k\ge0.
\end{equation}
Thus $\ell_k\in X^*$ for every listed source.

For brevity put
\[
 B=\begin{cases}
       \|\phi\|_{H^r},&Y=H^r,\\
       \|\phi\|_Y+\|\phi\varphi\|_Y,&\text{otherwise}.
    \end{cases}
\]
The polynomial tests in Theorem~\ref{thm:relatively-compact-symbol}
give $B\le C_{X,n}\|D_{\phi,\varphi}^{n}\|_{X\to Y}$.
The same theorem's derivative calculation shows, for every
$u\in H(\mathbb D)$, that
\begin{equation}\label{eq:interior-target-bound}
 \|\phi(u\circ\varphi)\|_Y
     \le C_Y B\left(\max_{|w|\le\rho}|u(w)|
                          +\max_{|w|\le\rho}|u'(w)|\right).
\end{equation}
For $H^r$ only the first maximum is needed. For a derivative target
use the product rule and $\phi\varphi'=(\phi\varphi)'-\phi'\varphi$;
the same estimate preserves the defining vanishing integrals of a
little target.
Taking $u(w)=w^k$ gives $y_k\in Y$ and
\[
 \|y_k\|_Y\le C_{Y,\rho}B(k+1)\rho^k,\qquad k\ge0.
\]
Combining this with \eqref{eq:interior-coefficient-bound}, for every $t>0$,
\[
 \sum_{k=0}^{\infty}(\|\ell_k\|_{X^*}\|y_k\|_Y)^t
       \le C B^t\sum_{k=0}^{\infty}(k+1)^t(\rho/R)^{kt}<\infty.
\]

For $N\ge0$, put
\[
 K_Nf=\sum_{k=0}^{N-1}\ell_k(f)y_k,\qquad
 u_{N,f}(w)=f^{(n)}(w)-\sum_{k=0}^{N-1}\ell_k(f)w^k,
\]
with both sums empty when $N=0$. Then $K_N:X\to Y$ is bounded and
has rank at most $N$. The Taylor series for $f^{(n)}$ and
\eqref{eq:interior-coefficient-bound} give
\[
 \max_{|w|\le\rho}|u_{N,f}(w)|
       +\max_{|w|\le\rho}|u_{N,f}'(w)|
   \le C_{X,n,R,\rho}(N+1)(\rho/R)^N\|f\|_X.
\]
Indeed, after factoring out $\|f\|_X$, the first bound is a geometric
tail. For the derivative, sum
$\rho^{-1}\sum_{k\ge N}k(\rho/R)^k$.
Since $(D_{\phi,\varphi}^{n}-K_N)f=\phi(u_{N,f}\circ\varphi)$,
equation \eqref{eq:interior-target-bound} yields
\[
 \|D_{\phi,\varphi}^{n}-K_N\|_{X\to Y} \le C B(N+1)(\rho/R)^N \le
 C'\|D_{\phi,\varphi}^{n}\|_{X\to Y}\eta^N.
\]
The last step uses $\rho/R<\eta$. This proves \eqref{eq:interior-rank-one} in operator quasi-norm. Since
$\operatorname{rank}K_N\le N<N+1$, it also bounds $a_{N+1}$.
Replacing $N$ by $N-1$ and absorbing the fixed factor $\eta^{-1}$
proves (1) for every $N\ge1$.

For the limiting rate, choose any
$\eta\in(\|\varphi\|_\infty,1)$ and then
$\rho\in(\|\varphi\|_\infty,\eta)$. The estimate in (1) gives a
limit superior at most $\eta$; let $\eta$ decrease to
$\|\varphi\|_\infty$. This argument includes
$\varphi\equiv0$, when the operator has rank at most one.

When $X,Y$ are Banach spaces, the representation and summability in
(2) give $t$-nuclearity by the stated definition for each
$0<t\le1$. No approximation property of either space is required.
For the last assertion, the defining norms of $F(2,\beta_1,0)$
and $F(2,\beta_2,0)$ are Hilbert norms.
The bounds persist under the equivalent Hilbert norms on the weighted
Dirichlet spaces. For compact maps between Hilbert spaces, the
approximation numbers are the singular values; see
\cite{Simon2005}. Their geometric bound makes their $t$th powers
summable for every $t>0$, proving membership in $\mathcal S_t$.
\end{proof}

In particular, every weight in Theorem~\ref{thm:F-all-weights} gives
these geometric approximations. For $p\ge1$, the resulting operators
are $t$-nuclear for all $0<t\le1$.

\section{Schatten classes and exact norms}\label{sec:Schatten-exact}

\subsection{A trace-ideal criterion}

We now restrict the trace-ideal question to Hilbert spaces and use
the coefficient norm on $\mathcal D_\beta$ defined in
Subsection~\ref{subsec:intro-spaces}. For $\beta>-1$, its square is
comparable to $|f(0)|^2+\|f'\|_{A_\beta^2}^2$, and hence to the
squared norm of $F(2,\beta,0)$. Indeed,
\[
 \|z^k\|_{A_\gamma^2}^2
   =\frac{\Gamma(k+1)\Gamma(\gamma+1)}{\Gamma(k+\gamma+2)}
       \approx(k+1)^{-\gamma-1},\qquad \gamma>-1.
\]
\begin{lemma}\label{lem:Dirichlet-derivative-reduction}
Let $\beta\in\mathbb R$, let $m\ge0$ be an integer, and suppose
$\beta+2m>1$. Differentiation $f\mapsto f^{(m)}$ maps
$\mathcal D_\beta$ boundedly onto $A_{\beta+2m-2}^2$ and has the
bounded right inverse $J^m$, where $J$ is integration from zero and
$J^0=I$. Moreover,
\begin{equation}\label{eq:Dirichlet-higher-derivative}
 \|f^{(m)}\|_{A_{\beta+2m-2}^2}^2
 \approx\sum_{k\ge m}(k+1)^{1-\beta}|c_k|^2,
               \qquad f(z)=\sum_{k\ge0}c_kz^k,
\end{equation}
and
\[
 \sup_{\|f\|_{\mathcal D_\beta}\le1}|f^{(m)}(w)|
       \approx(1-|w|^2)^{-(\beta+2m)/2}.
\]
\end{lemma}

\begin{proof}
Put $\gamma=\beta+2m-2>-1$. Orthogonality and the displayed monomial
norm formula give, for $k\ge m$,
$\|(z^k)^{(m)}\|_{A_\gamma^2}^2\approx
(k+1)^{2m-\gamma-1}=(k+1)^{1-\beta}$.
Summation proves \eqref{eq:Dirichlet-higher-derivative}, first for
polynomials and then for power series by monotone convergence.
Integrating a Bergman function $m$ times and assigning zero to the
first $m$ coefficients proves boundedness of the right inverse.
The Bergman evaluation norm is comparable to
$(1-|w|^2)^{-(\gamma+2)/2}$, by its reproducing kernel
$(\gamma+1)(1-\overline wz)^{-\gamma-2}$; the derivative map and its
right inverse give both bounds in the last assertion.
\end{proof}

These norm equivalences will be used for comparisons, not for exact
values of norms or singular values.

Inner products are linear in the first variable. For Hilbert spaces
$H_1,H_2$ and $0<t<\infty$, the Schatten class
$\mathcal S_t(H_1,H_2)$ consists of compact operators whose singular
values $s_N$ satisfy $\sum_Ns_N^t<\infty$, with
$\|B\|_{\mathcal S_t}^t=\sum_Ns_N(B)^t$. We use the ideal inequality
$\|UBV\|_{\mathcal S_t}\le\|U\|\|B\|_{\mathcal S_t}\|V\|$ and the
quasi-triangle inequality for this quantity; see \cite{Simon2005}.
For $0<t\le2$ and a finite matrix $(b_{ik})$, we also use
\begin{equation}\label{eq:Schatten-matrix-entries}
 \|(b_{ik})\|_{\mathcal S_t}^t\le\sum_{i,k}|b_{ik}|^t.
\end{equation}
To verify it, apply concavity of $x^{t/2}$ to the spectral measure of
$B^*B$ at each coordinate vector. This gives
$\operatorname{tr}(B^*B)^{t/2}\le\sum_k\|Be_k\|_2^t$, and each
column satisfies $\|Be_k\|_2^t\le\sum_i|b_{ik}|^t$. For $t\ge2$
the spectral-measure inequality is reversed, so
$\sum_k\|Be_k\|_2^t\le\|B\|_{\mathcal S_t}^t$ for an orthonormal
basis, or for any finite part of one.

We recall the positive-measure Bergman embedding criterion of
Luecking \cite{Luecking1987}; see also \cite{Zhu2007}. For
$\gamma>-1$, a positive Borel measure $\mu$ and $0<t<\infty$, the
embedding $I_\mu:A_\gamma^2\to L^2(\mu)$ belongs to $\mathcal S_t$
if and only if
\begin{equation}\label{eq:Bergman-trace-criterion}
 \sum_{j,l}\left(\frac{\mu(Q_{j,l})}{\delta_j^{\gamma+2}}\right)^{t/2}
      <\infty,
 \qquad
 \|I_\mu\|_{\mathcal S_t}^t
      \approx\sum_{j,l}
           \left(\frac{\mu(Q_{j,l})}{\delta_j^{\gamma+2}}\right)^{t/2}
\end{equation}
when these conditions hold. The criterion is often stated for
fixed-radius hyperbolic disks. It has the displayed form for our
partition as well: the marked points are uniformly separated, each
rectangle lies in a disk of a fixed hyperbolic radius about its marked
point, and each such disk meets only a bounded number of rectangles
of comparable size. These comparisons preserve the sums for every
positive exponent, including exponents less than one.

The sufficiency argument below uses this scalar embedding theorem.
Necessity requires separating two derivative terms in a Schatten
norm. Hilbert--Schmidt separation for sums of weighted
composition--differentiation operators was studied in
\cite{ArroussiLiuTongYang2024,BarikSenPaul2026}. We treat $t\ge2$
by the basis inequality and $0<t<2$ by a separated finite-matrix
argument.

For the next two lemmas only, enumerate the rectangles and marked
points as $Q_k,a_k,\delta_k$ instead of using the double index $(j,l)$.
Write
$d(z,w)=\frac12\log\bigl((1+|\sigma_z(w)|)/(1-|\sigma_z(w)|)\bigr)$
for hyperbolic distance.

\begin{lemma}\label{lem:Schatten-kernel-tail}
Let $A>0$, $n\ge1$, $t>0$, and let $M$ be a positive integer with
$(M-A)t>1$. Put
\[
 H_{ik}=\frac{\delta_i^{A+n-1}\delta_k^{M-A}}
                   {|1-\overline{a_k}a_i|^{M+n-1}}.
\]
Then
\begin{equation}\label{eq:Schatten-kernel-tail}
 \lim_{R\to\infty}\sup_i
       \sum_{d(a_i,a_k)>R}H_{ik}^{\,t}=0.
\end{equation}
For each $R>0$, the marked points can be divided into finitely many
sets within each of which distinct points have distance greater
than $R$.
\end{lemma}

\begin{proof}
The spacing of the marked points gives a fixed $r_0>0$ such that the
disks $\{w:d(w,a_k)<r_0\}$ are disjoint. On each disk,
$1-|w|^2\approx\delta_k$,
$|1-\overline wa_i|\approx|1-\overline{a_k}a_i|$, and its area is
comparable to $\delta_k^2$, with constants independent of $i,k$.
Consequently, with $b=A+n-1>0$ and $R>r_0$,
\[
 \sum_{d(a_i,a_k)>R}H_{ik}^{\,t}
 \le C(1-|a_i|^2)^{bt}
       \int_{d(a_i,w)>R-r_0}
       \frac{(1-|w|^2)^{(M-A)t-2}}
            {|1-\overline wa_i|^{(M+n-1)t}}\,dA(w).
\]
Under the substitution $w=\sigma_{a_i}(\zeta)$, the right-hand side
becomes
\[
 C\int_{|\zeta|>\tanh(R-r_0)}
    \frac{(1-|\zeta|^2)^v}{|1-\overline{a_i}\zeta|^\lambda}\,dA(\zeta),
 \qquad v=(M-A)t-2,\quad \lambda=(M-2A-n+1)t.
\]
All factors involving $1-|a_i|^2$ cancel. Here $v>-1$ and
$v+2-\lambda=bt>0$. Choose
$0<\varepsilon<\min\{v+1,bt\}$. The elementary subcritical kernel
estimate gives
\[
 \sup_{a\in\mathbb D}\int_{\mathbb D}
       \frac{(1-|\zeta|^2)^{v-\varepsilon}}
            {|1-\overline a\zeta|^\lambda}\,dA(\zeta)<\infty.
\]
For completeness, this follows by angular integration: for
$\lambda>1$, the angular integral is bounded by
$C(1-r)^{1-\lambda}$; for $\lambda=1$ it is bounded by
$C\log(e/(1-r))$; and for $\lambda<1$ it is uniformly bounded.
The remaining radial integrals converge because
$v-\varepsilon>-1$ and $\lambda<v-\varepsilon+2$.
Factoring $(1-|\zeta|^2)^\varepsilon$ out on the restricted region
now bounds the preceding tail by
$C(1-\tanh^2(R-r_0))^\varepsilon$, uniformly in $i$. This proves
\eqref{eq:Schatten-kernel-tail}.

Finally, any disk of a fixed hyperbolic radius contains at most a
fixed number of marked points. Indeed, the disjoint disks of radius
$r_0$ about these points lie in a slightly larger disk, and their
areas with respect to $(1-|z|^2)^{-2}\,dA(z)$ are equal and positive.
Join two marked points when their distance is at most $R$. The
resulting graph has uniformly bounded degree. A greedy coloring of
its countable set of vertices uses finitely many colors, and each
color class has the required separation.
\end{proof}

\begin{lemma}\label{lem:two-jet-Schatten}
Let $\alpha\in\mathbb R$, let $n\ge0$ be an integer with
$\alpha+2n>1$, and let $t>0$. Let $\tau$ be a positive
$\sigma$-finite Borel measure on $\mathbb D$, let $u,v$ be finite-valued
measurable functions, and let $\varphi$ be an analytic self-map of
$\mathbb D$. Set
\[
 \begin{aligned}
  \mathcal Lf(z)&=u(z)f^{(n)}(\varphi(z))
                  +v(z)f^{(n+1)}(\varphi(z)),\\
  d_k&=\delta_k^{-\alpha-2n}
            \int_{\varphi^{-1}(Q_k)}|u|^2\,d\tau
       +\delta_k^{-\alpha-2n-2}
            \int_{\varphi^{-1}(Q_k)}|v|^2\,d\tau.
 \end{aligned}
\]
The formula defines an operator in
$\mathcal S_t(\mathcal D_\alpha,L^2(\tau))$ if and only if
$\sum_kd_k^{t/2}<\infty$. In that case
\[
 \|\mathcal L\|_{\mathcal S_t}^t\approx\sum_kd_k^{t/2}.
\]
The constants depend only on $\alpha,n,t$ and the partition.
\end{lemma}

\begin{proof}
Put $\gamma=\alpha+2n-2>-1$ and consider
$Bh=u(h\circ\varphi)+v(h'\circ\varphi)$ on $A_\gamma^2$.
Lemma~\ref{lem:Dirichlet-derivative-reduction} gives
$\mathcal L=B\partial^n$ and $B=\mathcal L J^n$, with the derivative
map and its right inverse bounded. The ideal inequality therefore
reduces both the criterion and the norm comparison to $B$.
The exponents in $d_k$ are now $\gamma+2$ and $\gamma+4$.

Suppose $\sum_kd_k^{t/2}<\infty$, and denote the weighted pull-back
measures of $|u|^2d\tau$ and $|v|^2d\tau$ by $\mu_0$ and $\mu_1$.
They are finite, since $d_k$ is bounded and
$\sum_k\delta_k^{\gamma+2}<\infty$.
By \eqref{eq:Bergman-trace-criterion}, the embeddings from
$A_\gamma^2$ into $L^2(\mu_0)$ and from $A_{\gamma+2}^2$ into
$L^2(\mu_1)$ belong to $\mathcal S_t$. The corresponding weighted
composition maps into $L^2(\tau)$ have the same singular values:
their squared norms on each function give the same quadratic forms
as the embeddings. Differentiation maps $A_\gamma^2$ boundedly into
$A_{\gamma+2}^2$, as follows from the monomial norm formula.
The ideal and quasi-triangle inequalities yield
\[
 \|B\|_{\mathcal S_t}^t\le C\sum_kd_k^{t/2}.
\]

Conversely, suppose $B\in\mathcal S_t$. Testing $1$ and $z$ gives
$u=B1\in L^2(\tau)$ and $v=Bz-u\varphi\in L^2(\tau)$, so each
$d_k$ is finite. Set $A=(\gamma+2)/2$ and fix an integer $M$ such
that $M>A+1$ and $(M-A)t>1$. For $e=0,1$, put
\[
 g_{k,e}(z)=\frac{\delta_k^{M+e-A}}
                    {(1-\overline{a_k}z)^{M+e}}.
\]
Since $H(2,2,(\gamma+1)/2)=A_\gamma^2$ with equivalent norms,
Lemma~\ref{lem:discrete-estimates} gives
\begin{equation}\label{eq:Schatten-atom-synthesis}
 \left\|\sum_kc_kg_{k,e}\right\|_{A_\gamma^2}
                      \le C\left(\sum_k|c_k|^2\right)^{1/2}
\end{equation}
for every finite sum. The bound is also uniform if $e=e_k$ varies
with $k$, by splitting the sum into its two possible values.

Let $\Pi_i$ be multiplication by $\mathbf1_{\varphi^{-1}(Q_i)}$
in $L^2(\tau)$. Lemma~\ref{lem:two-kernels}, with orders $0,1$,
gives
\begin{equation}\label{eq:Schatten-local-diagonal}
 d_i\le C\sum_{e=0}^1\|\Pi_i B g_{i,e}\|_{L^2(\tau)}^2.
\end{equation}
Use Lemma~\ref{lem:Schatten-kernel-tail} with its order equal to $1$,
so that
$H_{ik}=\delta_i^A\delta_k^{M-A}|1-\overline{a_k}a_i|^{-M}$.
For $w\in Q_i$, direct differentiation gives
\[
 \delta_i^A|g_{k,e}(w)|+
 \delta_i^{A+1}|g_{k,e}'(w)|\le C H_{ik}.
\]
Here $|1-\overline{a_k}w|\approx|1-\overline{a_k}a_i|$, and both
$\delta_i$ and $\delta_k$ are bounded by a fixed multiple of this
denominator. Consequently,
\begin{equation}\label{eq:Schatten-local-offdiagonal}
 \|\Pi_i B g_{k,e}\|_{L^2(\tau)}\le C d_i^{1/2}H_{ik}.
\end{equation}

For $t\ge2$, let $V_e:\ell^2\to A_\gamma^2$ be the bounded
synthesis map in \eqref{eq:Schatten-atom-synthesis}. If
$(\mathbf e_k)$ is the coordinate basis of $\ell^2$, the basis
inequality gives
\[
 \sum_k\|BV_e\mathbf e_k\|_{L^2(\tau)}^t
       \le\|BV_e\|_{\mathcal S_t}^t
       \le C\|B\|_{\mathcal S_t}^t.
\]
Raise \eqref{eq:Schatten-local-diagonal} to the power $t/2$, sum over
$i$, and use that $\Pi_i$ is a contraction. This proves the lower
norm bound when $t\ge2$.

Now let $0<t<2$. Fix the constants in
\eqref{eq:Schatten-local-diagonal}--\eqref{eq:Schatten-local-offdiagonal}.
Given $\varepsilon>0$, choose $R$ by
Lemma~\ref{lem:Schatten-kernel-tail} such that
$\sup_i\sum_{d(a_i,a_k)>R}H_{ik}^t<\varepsilon$, and partition the
marked points into finitely many $R$-separated sets.
Take a finite subset $F$ of one such set, omitting indices with
$d_i=0$. For each $i\in F$, choose $e_i\in\{0,1\}$ so that
\[
 \|\Pi_i B g_{i,e_i}\|_{L^2(\tau)}\ge c d_i^{1/2},
 \qquad h_i=\frac{\Pi_i B g_{i,e_i}}
                      {\|\Pi_i B g_{i,e_i}\|_{L^2(\tau)}}.
\]
The $h_i$ are orthonormal because their supports are disjoint.
Thus $Uh=(\langle h,h_i\rangle)_{i\in F}$ is a contraction, and
$V_Fc=\sum_{i\in F}c_i g_{i,e_i}$ has uniformly bounded norm by
\eqref{eq:Schatten-atom-synthesis}.
The matrix $Z=UBV_F$ satisfies
$\|Z\|_{\mathcal S_t}^t\le C\|B\|_{\mathcal S_t}^t$ and
\[
 Z_{ii}=\|\Pi_i B g_{i,e_i}\|_{L^2(\tau)}\ge c d_i^{1/2},
 \qquad |Z_{ik}|\le C d_i^{1/2}H_{ik}\quad(i\ne k).
\]
Let $Z_0$ be its diagonal part. By
\eqref{eq:Schatten-matrix-entries},
\[
 \|Z-Z_0\|_{\mathcal S_t}^t
 \le C\sum_{i\in F}d_i^{t/2}\sum_{\substack{k\in F\\k\ne i}}H_{ik}^t
 \le C\varepsilon\sum_{i\in F}d_i^{t/2}.
\]
The quasi-triangle inequality now yields
\[
 c^t\sum_{i\in F}d_i^{t/2}
 \le\|Z_0\|_{\mathcal S_t}^t
 \le C_t\|B\|_{\mathcal S_t}^t+
                         C_t'\varepsilon\sum_{i\in F}d_i^{t/2}.
\]
Choose $\varepsilon$ so that $C_t'\varepsilon<c^t/2$ and absorb
the last term. The bound is independent of the finite set $F$.
Exhaust each separated set and add over the finitely many sets to get
$\sum_i d_i^{t/2}\le C\|B\|_{\mathcal S_t}^t$.
Together with the initial derivative reduction, this proves the
criterion and both norm estimates, including $n=0$.
\end{proof}

Return to the double index $(j,l)$. Let $\beta_1\in\mathbb R$,
$\beta_2>-1$, $n\ge0$ and $\beta_1+2n>1$. For the analytic
symbols $\phi,\varphi$, set
\[
 \begin{aligned}
  \nu_0(E)&=\int_{\varphi^{-1}(E)}|\phi'(z)|^2
                         (1-|z|^2)^{\beta_2}\,dA(z),\\
  \nu_1(E)&=\int_{\varphi^{-1}(E)}|\phi(z)\varphi'(z)|^2
                         (1-|z|^2)^{\beta_2}\,dA(z),\\
  v_{j,l}&=\frac{\nu_0(Q_{j,l})}{\delta_j^{\beta_1+2n}}
             +\frac{\nu_1(Q_{j,l})}{\delta_j^{\beta_1+2n+2}}.
 \end{aligned}
\]
For $\beta_1>-1$ and $n\ge1$, these agree with the arrays in
Section~\ref{sec:F-spaces} at $(p_i,q_i,s_i)=(2,\beta_i,0)$, with
the redundant local parameter suppressed. The coefficient reduction
extends their use to $\beta_1+2n>1$. No preliminary measure
condition is imposed.

\begin{theorem}\label{thm:Dirichlet-Schatten}
Let $\beta_1\in\mathbb R$, $\beta_2>-1$, and let $n\ge0$ be
an integer such that $\beta_1+2n>1$. Let $0<t<\infty$,
$\phi\in H(\mathbb D)$, and let $\varphi$ be an analytic self-map of
$\mathbb D$. The following conditions are equivalent:
\begin{enumerate}
 \renewcommand{\labelenumi}{\textup{(\arabic{enumi})}}
 \setlength{\itemsep}{2pt}\setlength{\parskip}{0pt}
 \item $D_{\phi,\varphi}^{n}:\mathcal D_{\beta_1}\to
                  \mathcal D_{\beta_2}$ belongs to $\mathcal S_t$.
 \item $\displaystyle\sum_{j,l}v_{j,l}^{t/2}<\infty$.
 \item For one, equivalently every, fixed $0<\eta<1$,
 \[
  \int_{\mathbb D}\left(
   \frac{\nu_0(\Delta(w,\eta))}{(1-|w|^2)^{\beta_1+2n}}
   +\frac{\nu_1(\Delta(w,\eta))}{(1-|w|^2)^{\beta_1+2n+2}}
       \right)^{t/2}\frac{dA(w)}{(1-|w|^2)^2}<\infty.
 \]
\end{enumerate}
When these conditions hold,
\begin{equation}\label{eq:Dirichlet-Schatten-norm}
 \|D_{\phi,\varphi}^{n}\|_{\mathcal S_t}^t
   \approx
   \frac{|\phi(0)|^t}
        {(1-|\varphi(0)|^2)^{t(\beta_1+2n)/2}}
           +\sum_{j,l}v_{j,l}^{t/2}.
\end{equation}
In particular, Hilbert--Schmidt membership is equivalent to
\[
 \int_{\mathbb D}\left(
  \frac{|\phi'(z)|^2}{(1-|\varphi(z)|^2)^{\beta_1+2n}}
  +\frac{|\phi(z)\varphi'(z)|^2}
        {(1-|\varphi(z)|^2)^{\beta_1+2n+2}}
       \right)(1-|z|^2)^{\beta_2}\,dA(z)<\infty.
\]
The sum in \eqref{eq:Dirichlet-Schatten-norm} may also be replaced
by the integral in \textup{(3)}. For $t=2$ it may be replaced by
the preceding direct symbol integral.
All comparison constants depend only on $\beta_1,\beta_2,n,t$ and
the fixed partition (and on $\eta$ for the disk-integral version).
\end{theorem}

\begin{proof}
Write
\[
 Bf=\phi'(f^{(n)}\circ\varphi)
          +\phi\varphi'(f^{(n+1)}\circ\varphi),
 \qquad \ell(f)=\phi(0)f^{(n)}(\varphi(0)).
\]
Lemma~\ref{lem:two-jet-Schatten}, with
$\alpha=\beta_1$, $u=\phi'$, $v=\phi\varphi'$ and
$d\tau=(1-|z|^2)^{\beta_2}\,dA$, characterizes $B$ as a map from
$\mathcal D_{\beta_1}$ to $L^2(\tau)$. Whenever it is bounded its
range consists of analytic functions, so it takes values in
$A_{\beta_2}^2$ and has the same singular values with this codomain.
The evaluation estimate in
Lemma~\ref{lem:Dirichlet-derivative-reduction} gives
\[
 \|\ell\|_{\mathcal D_{\beta_1}^*}
       \approx\frac{|\phi(0)|}{(1-|\varphi(0)|^2)^{(\beta_1+2n)/2}}.
\]
In particular, this functional is always bounded for fixed analytic
symbols.

The maps $h\mapsto h'$ from $\mathcal D_{\beta_2}$ to
$A_{\beta_2}^2$ and $Jg(z)=\int_0^zg(\zeta)\,d\zeta$ in the
opposite direction are bounded. Evaluation at zero is bounded as
well, and $\|1\|_{\mathcal D_{\beta_2}}=1$. The identity
\[
 D_{\phi,\varphi}^{n}f=\ell(f)\,1+JBf
\]
therefore proves, by the ideal and quasi-triangle inequalities,
\[
 \|D_{\phi,\varphi}^{n}\|_{\mathcal S_t}^t
       \approx\|\ell\|_{\mathcal D_{\beta_1}^*}^{\,t}
                     +\|B\|_{\mathcal S_t}^t
\]
whenever the terms on either side are finite. For the reverse
inequalities, compose $D_{\phi,\varphi}^{n}$ with evaluation at zero
and with differentiation. This proves \textup{(1)}$\Longleftrightarrow$\textup{(2)}, and
Lemma~\ref{lem:two-jet-Schatten} gives
\eqref{eq:Dirichlet-Schatten-norm}.

To compare (2) and (3), a disk $\Delta(w,\eta)$ with $w\in Q_i$
meets only a bounded number of rectangles $Q_k$ with comparable
boundary distances. Thus its two normalized masses are bounded by
$C\sum_{k\in N(i)}v_k$, for a uniformly bounded neighbor collection.
The invariant area of $Q_i$ is comparable to one, so integration and
the inequality for a positive power of a finite sum bound (3) by
$C\sum_kv_k^{t/2}$. Conversely, subdivide each $Q_k$ into a fixed
finite number of sets with pseudohyperbolic diameter less than
$\eta/4$. For each such set there is a disk of centers, of fixed
positive invariant area, whose $\eta$-disks contain that set.
Its two normalized masses therefore bound from below the corresponding
partial contributions to $v_k$. Summing those finitely many
contributions and using bounded overlap of the center disks gives
$\sum_kv_k^{t/2}$ bounded by (3). Both arguments work for every
$t>0$, using constants for a finite sum rather than a triangle
inequality that would fail at a small exponent.

For $t=2$, the sequence sum is $\sum_{j,l}v_{j,l}$. On
$Q_{j,l}$ one has $1-|w|^2\approx\delta_j$. Summing over the
partition turns this sum into the two integrals in the statement,
up to fixed constants. The definition of the pull-back measures
then gives the direct symbol integral in the statement. This proves
the Hilbert--Schmidt assertion as well.
\end{proof}

The scalar term in \eqref{eq:Dirichlet-Schatten-norm} cannot be
omitted from a norm estimate. For example, if $\phi$ and $\varphi$
are nonzero constants, both measures vanish but the operator is
nonzero and has rank one.

For $\beta_2>1$, there is also a direct scalar route to the entire
Schatten criterion. The target $\mathcal D_{\beta_2}$ is
$A_{\beta_2-2}^2$ with an equivalent Hilbert norm, and
Lemma~\ref{lem:Dirichlet-derivative-reduction} reduces the source to
$A_{\beta_1+2n-2}^2$. Define
\[
 \mu(E)=\int_{\varphi^{-1}(E)}|\phi(z)|^2
                          (1-|z|^2)^{\beta_2-2}\,dA(z).
\]
With the Bergman integral norms on the reduced spaces, the weighted
composition map and $I_\mu:A_{\beta_1+2n-2}^2\to L^2(\mu)$ have
the same quadratic form $B^*B$, hence the same singular values.
The derivative map, its bounded right inverse and the equivalent target
norm transfer Schatten membership and quasi-norm estimates in both
directions; they do not assert equality of singular values in the
original coefficient norms.
Consequently \cite{Luecking1987} gives the equivalent test
$\sum_{j,l}[\mu(Q_{j,l})/\delta_j^{\beta_1+2n}]^{t/2}<\infty$.

When $\beta_2=1$, the target is $H^2$ and there is a boundary
version of the same reduction. Assume $\phi\not\equiv0$, and put
$\gamma=\beta_1+2n-2>-1$.
The operator reduces, with a bounded derivative map and its right
inverse, to $Bh=\phi(h\circ\varphi):A_\gamma^2\to H^2$.
Since $A_\gamma^2=H(2,2,(\gamma+1)/2)$ with equivalent norms,
Theorem~\ref{thm:mixed-boundary} shows that boundedness forces
$\phi\in H^2$ and $|\varphi^*|<1$ almost everywhere.
For its boundary measure $\mu_\partial$, one has
$\|Bh\|_{H^2}^2=\int|h|^2\,d\mu_\partial$.
Thus $B$ and the corresponding Bergman embedding have the same
singular values, and \cite{Luecking1987} gives the equivalent
Schatten condition
$\sum_{j,l}[\mu_\partial(Q_{j,l})/\delta_j^{\beta_1+2n}]^{t/2}<\infty$,
together with those two boundary hypotheses.
Conversely, these hypotheses and the scalar condition give a bounded
$B$ by Theorem~\ref{thm:mixed-boundary}; the quadratic-form identity
then gives $B\in\mathcal S_t$ and hence
$D_{\phi,\varphi}^n\in\mathcal S_t$.
The zero weight is immediate.

\subsection{Monomial weights and linear symbols}

We now use the coefficient Hilbert norms exactly, for arbitrary real
Dirichlet indices. The monomials have orthogonal images, which gives
exact formulas. The falling factorial is written as $k!/(k-n)!$.

\begin{theorem}\label{thm:monomial-exact}
Let $\beta_1,\beta_2\in\mathbb R$, let $m,n\ge0$ be integers, and
let $\rho\in\mathbb C$ with $|\rho|\le1$. Take
$\phi(z)=z^m$ and $\varphi(z)=\rho z$.

\textup{(a)} Suppose $0<|\rho|\le1$. For $k\ge n$, put
\begin{equation}\label{eq:monomial-weights}
 \sigma_k=\frac{k!}{(k-n)!}|\rho|^{k-n}
        (k+1)^{(\beta_1-1)/2}
        (k-n+m+1)^{(1-\beta_2)/2}.
\end{equation}
Then $D_{z^m,\rho z}^{n}:\mathcal D_{\beta_1}\to
\mathcal D_{\beta_2}$ is bounded if and only if
$\sup_{k\ge n}\sigma_k<\infty$. Under this condition,
\[
 \|D_{z^m,\rho z}^{n}\|=\sup_{k\ge n}\sigma_k,
 \qquad
 \|D_{z^m,\rho z}^{n}\|_e=\limsup_{k\to\infty}\sigma_k.
\]
It is compact if and only if $\sigma_k\to0$. In the compact case,
its nonzero singular values are the nonincreasing rearrangement of
$(\sigma_k)_{k\ge n}$. For every $t>0$,
\[
 D_{z^m,\rho z}^{n}\in\mathcal S_t
       \iff\sum_{k=n}^{\infty}\sigma_k^t<\infty,
 \qquad
 \|D_{z^m,\rho z}^{n}\|_{\mathcal S_t}^t
       =\sum_{k=n}^{\infty}\sigma_k^t
\]
when the sum is finite.

\textup{(b)} If $0<|\rho|<1$, the operator belongs to every
$\mathcal S_t$, $t>0$, and
\[
 \lim_{N\to\infty}a_N(D_{z^m,\rho z}^{n})^{1/N}=|\rho|.
\]
If $|\rho|=1$, the exact parameter conditions are
\begin{enumerate}
 \renewcommand{\labelenumi}{\textup{(\arabic{enumi})}}
 \setlength{\itemsep}{2pt}\setlength{\parskip}{0pt}
 \item boundedness $\iff\beta_2-\beta_1\ge2n$;
 \item compactness $\iff\beta_2-\beta_1>2n$;
 \item membership in $\mathcal S_t$
          $\iff\beta_2-\beta_1>2n+2/t$.
\end{enumerate}
At $\beta_2-\beta_1=2n$, the essential norm is exactly $1$.

\textup{(c)} If $\rho=0$, the operator has rank one, with its only
nonzero singular value equal to
\[
 n!(n+1)^{(\beta_1-1)/2}(m+1)^{(1-\beta_2)/2}.
\]
In particular it belongs to every Schatten class. All assertions
include $n=0$, with $0!=1$.
\end{theorem}

\begin{proof}
The functions $e_k^{(\beta)}(z)=(k+1)^{(\beta-1)/2}z^k$ form an
orthonormal basis of $\mathcal D_\beta$. For $k<n$ the image of
$e_k^{(\beta_1)}$ is zero, whereas for $k\ge n$ and $\rho\ne0$,
\[
 D_{z^m,\rho z}^{n}e_k^{(\beta_1)}
   =\frac{k!}{(k-n)!}\rho^{k-n}(k+1)^{(\beta_1-1)/2}
       (k-n+m+1)^{(1-\beta_2)/2}e_{k-n+m}^{(\beta_2)}.
\]
The target basis indices are distinct. Orthogonality therefore gives,
for every finite sum $f=\sum_kc_ke_k^{(\beta_1)}$,
\[
 \|D_{z^m,\rho z}^{n}f\|_{\mathcal D_{\beta_2}}^2
             =\sum_{k\ge n}|c_k|^2\sigma_k^2.
\]
This proves the boundedness criterion and the exact norm. When the
sequence is bounded, density of polynomials yields the bounded
extension, which agrees with the analytic formula by local uniform
convergence.

Keeping the first finitely many nonzero columns gives a finite-rank
approximation, and the norm of the discarded part is the supremum
of the remaining $\sigma_k$. Hence the essential norm is at most
$\limsup_k\sigma_k$. Conversely, for any compact operator $K$,
$\|Ke_k^{(\beta_1)}\|\to0$, since an orthonormal sequence is weakly
null. Thus
\[
 \|D_{z^m,\rho z}^{n}-K\|
   \ge\limsup_{k\to\infty}
         \|(D_{z^m,\rho z}^{n}-K)e_k^{(\beta_1)}\|
   \ge\limsup_{k\to\infty}\sigma_k.
\]
Taking the infimum over $K$ proves the asserted essential norm and
compactness criterion. In the compact case, the positive operator
$(D_{z^m,\rho z}^{n})^*D_{z^m,\rho z}^{n}$ is diagonal in the
source basis, with diagonal entries $0$ for $k<n$ and $\sigma_k^2$
for $k\ge n$. Its eigenvalues give the singular-value and Schatten
formulas in (a).

Put $\gamma=n+(\beta_1-\beta_2)/2$. Directly from
\eqref{eq:monomial-weights},
\[
 \frac{\sigma_k}{|\rho|^{k-n}k^\gamma}\longrightarrow1
                                    \qquad(k\to\infty).
\]
For $|\rho|=1$, this gives all three conditions in (b), since
$\sum_k k^{t\gamma}<\infty$ if and only if $t\gamma<-1$. If $\gamma=0$, then $\sigma_k\to1$,
which proves the exact essential norm at the boundary.

If $0<|\rho|<1$, exponential decay makes
$\sum_k\sigma_k^t$ finite for every $t>0$. Moreover,
$\sigma_{k+1}/\sigma_k\to|\rho|<1$, so the sequence is eventually
strictly decreasing. It is positive and tends to zero; consequently,
a sufficiently late tail lies below every earlier term. Its
nonincreasing rearrangement therefore satisfies
$a_N(D_{z^m,\rho z}^{n})=\sigma_{n+N-1}$ for all sufficiently large
$N$. Taking $N$th roots proves the limit in (b).

Finally, if $\rho=0$, the formula is simply
$D_{z^m,0}^{n}f=f^{(n)}(0)z^m$. The functional $f\mapsto f^{(n)}(0)$
has norm $n!(n+1)^{(\beta_1-1)/2}$ on the source, and $z^m$ has
norm $(m+1)^{(1-\beta_2)/2}$ in the target. Their product is the
only nonzero singular value of this rank-one operator. This proves
(c) without assigning a value to $0^0$.
\end{proof}

For $\beta_1,\beta_2>-1$ and $0<|\rho|<1$, part (b) also shows
that the exponential rate in Theorem~\ref{thm:interior-approximation}
is sharp. At $|\rho|=1$ and $\beta_2=\beta_1+2n$, the operator is
bounded and noncompact, although it is weakly compact as a map
between Hilbert spaces. Thus weak compactness cannot replace
compactness in a statement covering all the source and target
parameters of this paper.

\begin{samepage}
\vskip.05in
\noindent\textbf{Conflicts of interest.} The authors declare no conflict of interest.

\vskip.05in
\noindent\textbf{Data availability statement.} Data sharing is not applicable to
this article as no datasets were generated or analysed during the current study.
\end{samepage}

\end{document}